\documentclass[reqno]{amsart}
\usepackage{amsmath}
\usepackage{amssymb}
\usepackage{hyperref}
\usepackage{a4wide}
\hypersetup{ colorlinks = true, urlcolor = blue, linkcolor = blue, citecolor = red }
\usepackage{xcolor}
\usepackage{enumitem}
\usepackage{esint}
\makeatletter
\def\namedlabel#1#2{\begingroup
	#2%
	\def\@currentlabel{#2}%
	\phantomsection\label{#1}\endgroup
}
\usepackage{varwidth}
\usepackage{tasks}
\usepackage[normalem]{ulem}
\DeclareMathOperator{\dv}{div}
\DeclareMathOperator{\loc}{loc}

\newcommand{\RR}{\mathbb{R}}
\newcommand{\mA}{\mathcal{A}}
\newcommand{\Om}{\Omega}
\newcommand{\na}{\nabla}
\newcommand{\pa}{\partial}
\newcommand{\sig}{\sigma}
\newcommand{\La}{\Lambda}
\newcommand{\al}{\alpha}

\newcommand{\de}{\delta}

\newcommand{\ep}{\epsilon}

\newcommand{\om}{\omega}
\newcommand{\la}{\lambda}

\newcommand{\data}{\mathit{data}}
\newcommand{\vh}{v_h}

\newcommand{\vhla}{v_h^{\La}}
\newcommand{\uu}{u_0}

\newcommand{\lai}{\la_i}
\newcommand{\laj}{\la_j}
\newcommand{\vla}{v^\La}
\newcommand{\Qlw}{{Q_l^\lambda(w)}}
\newcommand{\uik}{{K_iU_i}}

\newcommand{\Qo}{Q_{R_2,S_2}}
\newcommand{\vp}{\varphi}
\DeclareMathOperator*{\esssup}{ess\,sup}
\newcommand{\NN}{\mathbb{N}}
\nonstopmode

\theoremstyle{plain}
\newtheorem{theorem}{Theorem}[section]
\newtheorem{lemma}[theorem]{Lemma}

\newtheorem{definition}[theorem]{Definition}
\newtheorem{proposition}[theorem]{Proposition}

\newtheorem{remark}[theorem]{Remark}
\def\Xint#1{\mathchoice
	{\XXint\displaystyle\textstyle{#1}}%
	{\XXint\textstyle\scriptstyle{#1}}%
	{\XXint\scriptstyle\scriptscriptstyle{#1}}%
	{\XXint\scriptstyle\scriptscriptstyle{#1}}%
	\!\int}
\def\XXint#1#2#3{{\setbox0=\hbox{$#1{#2#3}{\int}$}
		\vcenter{\hbox{$#2#3$}}\kern-.5\wd0}}

\def\Yint#1{\mathchoice
	{\YYint\displaystyle\textstyle{#1}}%
	{\YYint\textstyle\scriptstyle{#1}}%
	{\YYint\scriptstyle\scriptscriptstyle{#1}}%
	{\YYint\scriptscriptstyle\scriptscriptstyle{#1}}%
	\!\iint}
\def\YYint#1#2#3{{\setbox0=\hbox{$#1{#2#3}{\iint}$}
		\vcenter{\hbox{$#2#3$}}\kern-.51\wd0}}
\def\longdash{{-}\mkern-3.5mu{-}} 

\def\fiint{\Yint\longdash}

\def\Xint#1{\mathchoice
	{\XXint\displaystyle\textstyle{#1}}%
	{\XXint\textstyle\scriptstyle{#1}}%
	{\XXint\scriptstyle\scriptscriptstyle{#1}}%
	{\XXint\scriptscriptstyle\scriptscriptstyle{#1}}%
	\!\int}
\def\XXint#1#2#3{{\setbox0=\hbox{$#1{#2#3}{\int}$ }
		\vcenter{\hbox{$#2#3$ }}\kern-.6\wd0}}

\def\dashint{\Xint-}

\usepackage{nameref}
\makeatletter
\let\orgdescriptionlabel\descriptionlabel
\renewcommand*{\descriptionlabel}[1]{%
	\let\orglabel\label
	\let\label\@gobble
	\phantomsection
	\edef\@currentlabel{#1}%
	\let\label\orglabel
	\orgdescriptionlabel{#1}%
}
\makeatother
\numberwithin{equation}{section}
\DeclareRobustCommand{\rchi}{{\mathpalette\irchi\relax}}
\newcommand{\irchi}[2]{\raisebox{\depth}{$#1\chi$}} 
\def\Xint#1{\mathchoice
    {\XXint\displaystyle\textstyle{#1}}%
    {\XXint\textstyle\scriptstyle{#1}}%
    {\XXint\scriptstyle\scriptscriptstyle{#1}}%
    {\XXint\scriptscriptstyle\scriptscriptstyle{#1}}%
    \!\int}
\def\XXint#1#2#3{\setbox0=\hbox{$#1{#2#3}{\int}$}
    \vcenter{\hbox{$#2#3$}}\kern-0.5\wd0}
\def\fint{\Xint-}
\def\dashint{\Xint{\raise4pt\hbox to7pt{\hrulefill}}}

\def\XXiint#1#2#3{\setbox0=\hbox{$#1{#2#3}{\iint}$}
    \vcenter{\hbox{$#2#3$}}\kern-0.5\wd0}

\usepackage[normalem]{ulem}

\begin{document}
	
\title[Lipschitz truncation]{Lipschitz truncation method for parabolic double-phase systems and applications}

\author{Wontae Kim}
\address[Wontae Kim]{Department of Mathematics, Aalto University, P.O. BOX 11100, 00076 Aalto, Finland}
\email{wontae.kim@aalto.fi}

\author{Juha Kinnunen}
\address[Juha Kinnunen]{Department of Mathematics, Aalto University, P.O. BOX 11100, 00076 Aalto, Finland}
\email[Corresponding author]{juha.k.kinnunen@aalto.fi}

\author{Lauri S\"arki\"o}
\address[Lauri S\"arki\"o]{Department of Mathematics, Aalto University, P.O. BOX 11100, 00076 Aalto, Finland}
\email{lauri.sarkio@aalto.fi}

\everymath{\displaystyle}

\makeatletter
\@namedef{subjclassname@2020}{\textup{2020} Mathematics Subject Classification}
\makeatother

\begin{abstract}
We discuss a Lipschitz truncation technique for parabolic double-phase problems of $p$-Laplace type in order to prove
energy estimates and uniqueness results for the Dirichlet problem. 
Moreover, we show existence for a non-homogeneous double-phase problem.
The Lipschitz truncation method is based on a Whitney-type covering result and a related partition of unity in the intrinsic geometry for the double-phase problem.
\end{abstract}

\keywords{Parabolic double-phase systems, existence theory, Lipschitz truncation method}
\subjclass[2020]{35K55, 35K65,  35D30, 35A01, 35A02}
\maketitle

\section{Introduction}
This paper discusses weak solutions $u=u(z)=u(x,t)$ to parabolic double-phase systems of type
\[
    u_t-\dv(|\na u|^{p-2}\na u+a(z)|\na u|^{q-2}\na u)=-\dv(|F|^{p-2}F+a(z)|F|^{q-2}F)
\]
in $\Omega_T=\Omega\times(0,T)$, where $\Omega$ is a bounded domain in $\mathbb{R}^n$ and $T>0$. Here $2\le p<q<\infty$ and the coefficient $a(\cdot)$ is non-negative and H\"older continuous. It is natural to assume that the gradient of a weak solution and the source term satisfy
\[
    \iint_{\Omega_T}\left(|\na u|^p+a(z)|\na u|^q\right)\,dz+\iint_{\Omega_T}\left(|F|^p+a(z)|F|^q\right)\,dz<\infty.
\]
Observe that the coefficient function $a(\cdot)$ may vanish in a subset of the domain and thus the condition above does not imply that $|\na u|\in L^q(\Omega_T)$.
One of the challenges is that the structure of the double-phase systems is not dominated by the $p$-growth or the $q$-growth. 
Indeed, the intrinsic geometry depends on the location and it switches between two phases.

Energy estimates are of fundamental importance in existence and regularity results for partial differential equations.
There is a delicate issue which does not occur in the elliptic case: A weak solution to the parabolic double-phase system under the natural function space assumption above is not an admissible test function after mollification in time, see Remark~\ref{rem_testfcn}.
A similar problem arises in the uniqueness of the Dirichlet problem since the difference of two solutions cannot be applied as a test function. Proposition~\ref{prop_liptr1} and Proposition~\ref{uni_prop} contain two versions of Lipschitz truncation for parabolic double-phase problems.
These are applied to prove the energy estimate in Theorem~\ref{thm_cac} and the uniqueness of solutions in Theorem~\ref{main2}.
As a corollary we obtain the higher integrability result for the gradient in \cite{KKM} without the additional assumption $|\na u|\in L^q(\Omega_T)$, see Remark~\ref{rem_hi}. 
The existence of solutions to a non-homogeneous Dirichlet problem \eqref{e2} is discussed in Theorem~\ref{main3}.

The main idea of the Lipschitz truncation is to construct an admissible test function by keeping it unchanged in the set where the strong maximal function of the gradient is small and redefining it in the complement by applying a partition of unity related to a Whitney covering.  
The Whitney covering in Proposition~\ref{prop_whitney} is constructed in an intrinsic geometry for the parabolic double-phase problem and it involves the $p$-intrinsic and $(p,q)$-intrinsic phases as in \cite{KKM}. 
A careful analysis is needed to combine the phases.
Since the intrinsic geometry depends on the location, it is challenging to obtain uniform bounds in a Vitali covering argument.
The partition of unity in Proposition~\ref{prop_whitney} is subordinate to the Whitney covering and reflects the double-phase structure. 
Maximal function estimates imply the convergence of the Lipschitz truncation to the original test function.
The existence proof is based on an approximation with regularized problems, Calder\'on-Zygmund estimates in \cite{CZ}, compactness arguments and the Minty-Browder approach.

The Lipschitz truncation was introduced for the elliptic setting in \cite{MR751305,MR970512}.
The corresponding technique for parabolic $p$-Laplace systems was developed
in \cite{MR1948889} and discussed further in \cite{MR4281251,MR4351758,MR3059060,MR3672391}. 
It has also been used to prove existence result for parabolic problems in \cite{MR3985550,MR2668872}. 
The regularity theory of elliptic double-phase problems has been studied in \cite{MR3348922,MR3294408,MR3447716,MR3985927,MR2076158}.
Very little is known about the parabolic double-phase problems. 
Local higher integrability of the gradient has been studied in \cite{KKM},
the existence of weak solutions to parabolic double-phase systems was proved in the homogeneous case in \cite{MR3532237} and in the scalar case with non-divergence data in \cite{MR3985549}. Our existence result seems to be new in the vectorial case with non-homogeneous data.
We are confident that our results can be applied to solve further regularity problems for the parabolic double-phase systems.

\section{Notation and main results}

\subsection{Notation} 
Let $\Omega$ be a bounded open subset in $\mathbb{R}^n$ and $\Omega_T=\Omega\times(0,T)$, $T>0$.  
A point in $\RR^{n+1}$ is denoted as $z=(x,t)$ or $w=(y,s)$, where $x,y \in \RR^n$ and $t,s\in \RR$.
A ball with center $x_0\in\RR^n$ and radius $\rho>0$ is denoted as
\[
    B_\rho(x_0)=\{x\in \RR^n:|x-x_0|<\rho\}
\]
 and a cube centered at the origin is denoted as
\[
 D_\rho=\{(x_1,...,x_n)\in \RR^n: |x_i|<\rho,\,i\in \{1,\dots,  n\}\}.
\]
 Let $t_0\in\RR$ and $\la>0$. Parabolic cylinders with quadratic scaling in time are denoted as
\[
    Q_\rho(z_0)=B_\rho(x_0)\times I_\rho(t_0)
    \quad\text{and}\quad 
     C_\rho=D_\rho\times I_\rho(0),
\]
where
\[
 I_\rho(t_0)=(t_0-\rho^2,t_0+\rho^2).
\]

We use the following notation for the double-phase functional. For the non-negative coefficient function $a(\cdot)$ extended to $\RR^{n+1}$, see \eqref{extended_a}, we define a function $H(z,s):\RR^{n+1}\times \RR^+\longrightarrow\RR^+$ as
\[
    H(z,s)=s^p+a(z)s^q.
\]
For $\la\geq1$, a $p$-intrinsic cylinder centered at $z_0=(x_0,t_0)$ is 
\begin{align}\label{def_Q_cylinder}
	Q_{\rho}^\la(z_0)=
	B_{\rho}(x_0)\times I_{\rho}^\la(t_0),
 \quad I_{\rho}^\la(t_0)=(t_0-\la^{2-p}\rho^2,t_0+\la^{2-p}\rho^2)
\end{align}
and a $(p,q)$-intrinsic cylinder centered at $z_0=(x_0,t_0)$ is
\[
 G_{\rho}^\la(z_0)=B_{\rho}(x_0)\times J_{\rho}^\la(t_0),
 \quad J_\rho^{\la}(t_0)=\left(t_0-\frac{\la^2}{H(z_0,\la)}\rho^2,t_0+\frac{\la^2}{H(z_0,\la)}\rho^2\right).
\]
The dependence from $z_0$ is omitted in the notation $J_\rho^{\la}(t_0)$ since $\la$ will be chosen so that the denominator $H(z_0,\la)$ remains fixed regardless of $z_0$ whenever it appears. Here, $\la^{2-p}$ and $\tfrac{\la^2}{H(z_0,\la)}$ are called the scaling factors in time. Note that if $\la=1$, then $Q_\rho(z_0)=Q^\la_{\rho}(z_0)$. Moreover, we write
\[
    cQ_\rho^\la(z_0)=Q_{c\rho}^\la(z_0)
    \quad\text{and}\quad 
    cG_\rho^\la(z_0)=G_{c\rho}^\la(z_0)
\]
 for $c>0$. We also consider parabolic cylinders with arbitrary scaling in time and denote
 \begin{align}\label{def_ell}
     Q_{r,s}(z_0)=B_r(x_0)\times \ell_{s}(t_0),\quad \ell_s(t_0)=(t_0-s,t_0+s),\quad r,s>0.
 \end{align}
The center point, height and radius may be omitted if irrelevant. In this case $Q$ means an arbitrary cylinder of the above type.  
 
 The $(n+1)$-dimensional Lebesgue measure of a set $E\subset\RR^{n+1}$ is denoted as $|E|$.
For $f\in L^1(\Omega_T,\RR^N)$ and a measurable set $E\subset\Om_T$ with $0<|E|<\infty$, we denote the integral average of $f$ over $E$ as
\[
	(f)_{E}=\frac{1}{|E|}\iint_{E}f\,dz=\fiint_{E}f\,dz.
\]

\subsection{Main results}
Let $2\le p<q<\infty$. In this paper we consider energy estimates of weak solutions to parabolic problems of double-phase type,
\begin{align}\label{11}
	u_t-\dv\mA(z,\na u)=-\dv(|F|^{p-2}F+a(z)|F|^{q-2}F)\quad\text{in}\ \Omega_T,
\end{align}
where $\mA(z,\xi):\Omega_T\times \RR^{Nn}\longrightarrow \RR^{Nn}$ with $N\ge1$ is a Carath\'eodory vector field satisfying the structure assumptions
\begin{align}\label{12}
		\mA(z,\xi)\cdot \xi\ge \nu(|\xi|^p+a(z)|\xi|^q)
      \quad\text{and}\quad 
		|\mA(z,\xi)|\le L(|\xi|^{p-1}+a(z)|\xi|^{q-1})
\end{align}
for a.e. $z\in \Omega_T$ and every $\xi\in \RR^{Nn}$ with constants $0<\nu\le L<\infty$.
The source term $F:\Omega_T\longrightarrow\RR^{Nn}$ is a vector field satisfying
\[
	\iint_{\Om_T}H(z,|F|)\,dz<\infty.
\]
Throughout this paper we assume that the coefficient function $a:\Omega_T\longrightarrow\RR^+$ satisfies
\[
	q\le p+\frac{2\alpha}{n+2},\quad 0\le a\in C^{\alpha,\alpha/2}(\Omega_T)\quad\text{for some}\ \alpha\in(0,1].
\]
Here $a\in C^{\alpha,\alpha/2}(\Omega_T)$ means that $a\in L^{\infty}(\Omega_T)$ and there exists a constant $[a]_{\alpha,\alpha/2;\Omega_T}=[a]_\alpha>0$, such that
\begin{align}\label{15}
		|a(x,t)-a(y,t)|\le [a]_{\alpha,\alpha/2;\Omega_T}|x-y|^\alpha
  \quad\text{and}\quad 
		|a(x,t)-a(x,s)|\le [a]_{\alpha,\alpha/2;\Omega_T}|t-s|^\frac{\alpha}{2}
\end{align}
for every $x,y\in\Omega$ and $t,s\in (0,T)$. It is not clear whether the range for $q$ is optimal. However, considering the corresponding elliptic theory this condition with the parabolic scaling deficit $\tfrac{p}{2}$ seems natural.

\begin{definition}\label{def_weak}
A function $u\in C(0,T;L^2(\Omega,\RR^N))\cap L^1(0,T;W^{1,1}(\Omega,\RR^N))$ satisfying
\[
\iint_{\Omega_T}\left(H(z,|u|)+H(z,|\na u|)\right)\,dz <\infty
\]
is a weak solution to \eqref{11} if
\[
    \iint_{\Omega_T}\left(-u\cdot \varphi_t+\mA(z,\na u)\cdot \na\varphi\right)\,dz
    =\iint_{\Omega_T}\left( |F|^{p-2}F\cdot \na \varphi+a(z)|F|^{q-2}F\cdot \na \varphi\right)\,dz
\]
for every $\varphi\in C_0^\infty(\Omega_T,\RR^N)$.
\end{definition}

In general, the time derivative of a weak solution exists only in the distributional sense. 
Steklov averages are applied to mollify the function in the time direction. Let $f\in L^1(\Omega_T)$ and $0<h<T$. We define the Steklov average $[f]_h(x,t)$ for all $0<t<T$ as
\begin{align*}
	[f]_h(x,t)=
	\begin{cases}
		\fint_t^{t+h}f(x,s)\,ds,&0<t<T-h,\\
		0,&T-h\le t.
	\end{cases}
\end{align*}
We refer to \cite{MR1230384} for standard properties of the Steklov average. 
In particular we note that for any $f\in L^2(\Omega_T)$, we have $[f]_h\in W^{1,2}((0,T-h);L^2(\Omega))$ with
\[
	\pa_t[f]_h=\frac{f(x,t+h)-f(x,t)}{h}
\]
for $(x,t) \in \Omega_{T-h}$.
The proof of the following lemma can be found for example in \cite[Lemma 8.1]{MR4351758}.
\begin{lemma}\label{lem21}
	Let $f\in L^1(\Omega_T)$ and $h>0$. Then there exists a constant $c=c(n)$ such that
	\[
		\fiint_{Q_{r,s}(z_0)}[f]_h\,dz\le c\fiint_{[Q_{r,s}(z_0)]_h}f\,dz,
	\]
where $[Q_{r,s}(z_0)]_h=Q_{r,s+h}(z_0)$.
\end{lemma}

\begin{remark}\label{rem_testfcn}
The Steklov averages are applied in order to have
\begin{align*}
\begin{split}
    &\int_{\Omega\times \{t\}}\left(\pa_t[u]_h\cdot \varphi+ \left[|\na u|^{p-2}\na u+a|\na u|^{q-2}\na u\right]_h\cdot \na\varphi\right)\,dx\\
    &=\int_{\Omega\times \{t\}}\left[|F|^{p-2}F+a|F|^{q-2}F\right]_h\cdot \na\varphi\,dx
\end{split}
\end{align*}
for every $\varphi\in C_0^\infty(\Omega,\RR^N)$ and for a.e. $0<t<T-h$. 
To obtain an energy estimate, we would like to apply the solution itself as a test function in the definition of weak solutions. 
This is possible if the right-hand side of
\begin{align*}
        &\int_{\Omega\times \{t\}}\left|\left[|\na u|^{p-2}\na u+ a|\na u|^{q-2}\na u\right]_h \right|| \na \varphi|\,dx\\
        &\le \int_{\Omega\times \{t\}}\left(\left|\left[|\na u|^{p-2}\na u\right]_h\right||\na \varphi|+ \left|\left[a|\na u|^{q-2}\na u\right]_h\right|| \na \varphi|\right)\,dx
\end{align*}
is finite for test functions in the function space of the weak solution. The first term in the integral is finite from H\"older's inequality and the property of the Steklov average as in the $p$-Laplace system case ($a(\cdot)\equiv0$). However, if we write open the second term in the integral,
\[
    \int_{\Omega\times\{t\}}\biggl|\fint_{t}^{t+h}a(x,\tau)|\na u(x,\tau)|^{q-2}\na u(x,\tau)\,d\tau \biggr| | \na \varphi(x)|\,dx,
\]
and let $\varphi = [u]_h\zeta$ for any cutoff function $\zeta$ in $\Omega_T$, it becomes apparent that this is difficult to bound since $[u]_h(\cdot,t)$ does not have the same time variable as $a(\cdot)$. The same problem occurs also for the source term.
The Lipschitz truncation technique is applied to overcome this problem.
\end{remark}

Our first main result is an  energy estimate for weak solutions as in Definition~\ref{def_weak}.

\begin{theorem} \label{thm_cac}
    Assume that $u$ is a weak solution of \eqref{11}. Then there exists a constant $c=c(p,q,\nu,L)$ such that
    \begin{align*}
        \begin{split}
            &\sup_{t\in \ell_{S_1}(t_0)}\fint_{B_{R_{1}}(x_0)}\frac{|u-(u)_{Q_{R_1,S_1}(z_0)}|^2}{S_1}\,dx+\fiint_{Q_{R_1,S_1}(z_0)} H(z,|\na u|)\,dz\\
		&\quad\le c\fiint_{Q_{R_2,S_2}(z_0)}\left(H\left(z,\frac{|u-(u)_{Q_{R_2,S_2}(z_0)}|}{R_2-R_1}\right) + \frac{|u-(u)_{Q_{R_2,S_2}(z_0)}|^2}{S_2-S_1} + H(z,|F|)\right)\,dz
        \end{split}
    \end{align*}
    for every $Q_{R_2,S_2}(z_0)\subset \Omega_T$ with $R_2,S_2>0$, $R_1\in[R_2/2,R_2)$ and $S_1\in[S_2/2^2,S_2)$.
\end{theorem}

\begin{remark}\label{rem_hi}
Theorem~\ref{thm_cac} and \cite[Remark 2.2]{KKM} imply that \cite[Theorem 1.1]{KKM} holds for
weak solutions \eqref{11} as in Definition~\ref{def_weak}.
In \cite{KKM} the result was proved under a technical assumption $|\na u|\in L^q(\Om_T)$.
\end{remark}

We discuss existence and uniqueness results for the Dirichlet problem
\begin{align}\label{e2}
\begin{cases}
    u_t-\dv \mathcal{A}(z,\na u)=-\dv(|F|^{p-2}F+a(z)|F|^{q-2}F)&\quad\text{in}\  C_{R}\\
    u=g&\quad\text{on}\ \pa_p C_{R},
\end{cases}
\end{align}
where $\mathcal{A}(z,\xi)=b(z)(|\xi|^{p-2}\xi+a(z)|\xi|^{q-2}\xi)$ and $b(z):C_R\longrightarrow\RR^+$ is a measurable function satisfying the ellipticity condition
\begin{align}\label{d9}
    \nu\le b(z)\le L
    \quad\text{for a.e.}\ z\in  C_R,
\end{align}
for some constants $0<\nu\le L<\infty$.
Here $ C_R=D_{R}\times I_R$ and $\pa_p C_R$ is the parabolic boundary defined as the union of $D_R\times\{-R^2\}$ and $\pa D_R\times (-R^2,R^2)$.
Moreover, $g\in C^\infty(\RR^{n+1},\RR^N)$ and $u=g$ on $\pa_p C_R$ means that
\[
 u-g\in L^p(I_R,W_0^{1,p}(D_R,\RR^N))
 \quad\text{and}\quad
 \lim_{\ep\to0^+}\fint_{-R^2}^{-R^2+\ep}\int_{D_R}|u-g|^2\,dx\,dt=0.
\]

Assume that $u$ and $w$ are weak solutions of \eqref{e2}. Then 
\begin{align}\label{e5}
    (u-w)_t-\dv (\mathcal{A}(z,\na u)-\mathcal{A}(z,\na w))=0\quad\text{in}\ C_{R}
\end{align}
and $u-w=0$ on $\pa_p C_{R}$ in the sense that on the lateral boundary we have
\begin{align}\label{e7}
    u-w\in L^p(I_R;W_0^{1,p}(D_R,\RR^N)),
\end{align}
and on the initial boundary we have
\begin{align}\label{e8}
    \lim_{\ep\to0^+}\fint_{-R^2}^{-R^2+\epsilon}\int_{D_R}|u-w|^2\,dx\,dt=0.
\end{align}
In the case of the parabolic $p$-Laplace system $(a(\cdot)\equiv0)$, $u-w$ is an admissible test function for \eqref{e5} and the uniqueness follows in a straightforward manner. However, $u-w$ may not be an admissible test function when $a(\cdot)\not\equiv0$. Again, we will apply the Lipschitz truncation method to settle this problem. The uniqueness result is stated in the following theorem.
\begin{theorem}\label{main2}
Assume that $u$ and $w$ are weak solutions to the Dirichlet boundary problem \eqref{e2}. Then $u= w$ a.e. in $ C_R$.
\end{theorem}

Then we discuss the existence of a weak solution to the Dirichlet problem above. 
In order to be able to apply the estimates in \cite{CZ}, we assume that the coefficient $b(\cdot)$ satisfies a VMO condition
\begin{align}\label{la5}
    \lim_{r\to 0^+}\sup_{|I|\le r^2}\sup_{x_0\in D_R}\fint_{I\cap I_R}\fint_{B_r(x_0)\cap D_R} |b(x,t)-(b)_{(B_r(x_0)\times I)\cap  C_R}|\,dx\,dt=0.
\end{align}
 
\begin{theorem}\label{main3}
Assume that $g\equiv 0$ and that \eqref{la5} holds true.
Then there exists a weak solution to the Dirichlet problem \eqref{e2}. 
Moreover, there exists a sequence $(u_l)_{l\in \mathbb{N}}$ of weak solutions  $u_l\in L^q(I_R;W_0^{1,q}(D_R,\RR^N))$, $l\in\mathbb{N}$, to the problem
\begin{align*}
    \begin{cases}
        \pa_tu_l-\dv \mA_l(z,\na u_l)=-\dv \left(|F_l|^{p-2}F_l+a_l(z)|F_l|^{q-2}F_l\right)&\text{in}\ C_R\\
        u_l=0&\text{on}\ \pa_p C_R,
    \end{cases}
\end{align*}
 such that
\[
    \lim_{l\to\infty}\iint_{ C_R}H(z,|\na u-\na u_l|)\,dz=0.
\]
Here $(F_l)_{l\in \mathbb{N}}$ is a sequence of truncations $F_l\in L^\infty( C_R,\RR^{Nn})$, $l\in \mathbb{N}$, of $F$ satisfying
\[
    \lim_{l\to\infty}\iint_{ C_R}H(z,|F-F_l|)\,dz=0,
\]
and $\mA_l$ is a perturbed $q$-Laplace operator from $\mA$, defined with a positive and decreasing sequence $(\ep_l)_{l\in\mathbb{N}}$ as
\[
    \mA_l(z,\xi)=b(z)\left(|\xi|^{p-2}\xi+a_l(z)|\xi|^{q-2}\xi\right),\quad a_l(z)=a(z)+\ep_l,\quad \lim_{l\to\infty}\ep_l=0.
\]
\end{theorem}

\section{Whitney decomposition}
In this section we provide a Whitney covering result in Proposition~\ref{prop_whitney} which is applied in the Lipschitz truncation method for parabolic double-phase problems in Section~\ref{sec_liptr}. The classical Whitney covering lemma decomposes an open set $U$ into a family of open sets $\{U_i\}_{i\in\mathbb{N}}$ according to a given metric $d(\cdot,\cdot)$. It has been employed for parabolic $p$-Laplace systems with the parabolic $p$-Laplace metric to prove the Lipschitz truncation method, see \cite{MR4351758,MR1948889}. We extend this by decomposing $U$ into $\{U_i\}_{i\in\mathbb{N}}$ according to pointwise metrics depending on the coefficient $a$.

The decomposition is based on the strong maximal function defined for $f\in L^1_{\loc}(\RR^{n+1})$ as
\begin{align} \label{strong_M}
	Mf(z)=\sup_{z\in Q}\fiint_{Q} |f|\,dw,
\end{align}
where the supremum is taken over all parabolic cylinders of the type in \eqref{def_ell} and satisfying $z\in Q$. 
It follows immediately that
\[
	Mf(z)\le \sup_{t\in (t_1,t_2)} \fint_{t_1}^{t_2} \sup_{x\in B}\fint_{B} |f|\,dy \,ds =M_t(M_x(|f|))(z),
\]
where $M_x$ and $M_t$ are the noncentered Hardy--Littlewood maximal operators with respect to space and time. 
By repeated application of the Hardy-Littlewood-Wiener maximal function theorem we obtain the following bound for the strong maximal function.
\begin{lemma}\label{str}
Let $1<\sig<\infty$ and $f\in L^{\sigma}(\RR^{n+1})$. Then there exists a constant $c=c(n,\sig)$ such that
\[
	\iint_{\RR^{n+1}}|Mf|^{\sig}\,dz\le c\iint_{\RR^{n+1}}|f|^\sig\,dz.
\]
\end{lemma}

\subsection{Initial definitions}\label{subsec_app_1}
We start by extending $a$ to $\RR^{n+1}$. 
Since $s\mapsto s^\alpha$ is a concave and subadditive function on $\RR^+$, the extended function 
\begin{align} \label{extended_a}
	\tilde{a}(z)=\inf_{w=(y,s)\in \Om_T}\left\{a(w)+[a]_{\alpha}(|x-y|^\alpha+|t-s|^{\frac{\alpha}{2}})\right\}, 
 \quad z=(x,t)\in\RR^{n+1},
\end{align}   
belongs to $C^{\alpha,\alpha/2}(\RR^{n+1})$ (see \cite[Theorem 2.7]{MR4306765} for the details).
With slight abuse of notation we denote $a(z) = \min\{\Tilde{a}(z),\|a\|_{L^\infty(\Omega_T)}\}$. Note that $\|a\|_{L^\infty(\mathbb{R}^{n+1})}=\|a\|_{L^\infty(\Omega_T)}$. 

Let $f\in L^p(\RR^{n+1})$, $f\geq 0$,  be such that
\begin{align}\label{t_31}
	\iint_{\RR^{n+1}}\left(f^p+af^q\right)\,dz<\infty,
\end{align}
and let $d=\tfrac{q-1+p}{2}$. By Lemma~\ref{str} there exists $\La_0>1+\|a \|_{L^\infty(\mathbb{R}^{n+1})}$ such that
\[
	\iint_{\RR^{n+1}}\left(M(f^d+(af^q)^\frac{d}{p})(z)\right)^\frac{p}{d}\,dz
 \le c(n,p,q) \iint_{\RR^{n+1}}\left(f^p+af^q\right)\,dz\le\La_0.
\]
Let $\La>\La_0$. We consider the set
\begin{align}\label{def_E}
	E(\La)=\left\{z\in \RR^{n+1}:M(f^d+(af^q)^\frac{d}{p})(z)\le\La^\frac{d}{p}\right\},
\end{align}
where $M(\cdot)$ is the strong maximal operator defined in \eqref{strong_M}.
Since the maximal function is lower semicontinuous, $E(\La)^c=\RR^{n+1}\setminus E(\La)$
is an open subset of $\RR^{n+1}$. Moreover, Chebyshev's inequality implies that
\begin{align} \label{convergence}
\lim_{\La \to \infty}\Lambda|E(\La)^c| 
\le\lim_{\La \to \infty}\iint_{E(\La)^c}\left(M(f^d+(af^q)^\frac{d}{p})(z)\right)^\frac{p}{d}\,dz
= 0.
\end{align}

Assuming $E(\La)^c\ne\emptyset$, we construct a Whitney decomposition of $E(\La)^c$. Since $s\mapsto s^p+a(z)s^q$ is an increasing function on $\mathbb{R}^+$, for every $z\in E(\La)^c$  there exists a unique $\la_z>1$ such that $\la_z^p+a(z)\la_z^q=\La$.
Let 
\begin{align}\label{app9}
	K=2^\frac{1}{2}+800[a]_\alpha \left(\frac{1}{|B_1|}\iint_{\RR^{n+1}}\left(M(f^d+(af^q)^\frac{d}{p})(z)\right)^\frac{p}{d}\,dz\right)^\frac{\alpha}{n+2}.
\end{align}
We consider a family of metrics $\{d_z(\cdot,\cdot)\}_{z\in E(\La)^c}$ given by
\begin{align*}
    d_z(z_1,z_2)=
    \begin{cases}
        \max\left\{|x_1-x_2|,\sqrt{\la_z^{p-2}|t_1-t_2|}\right\},&\text{if}\ K^2\la_z^p\ge a(z)\la_z^q,\\
        \max\left\{|x_1-x_2|,\sqrt{\La\la_z^{-2}|t_1-t_2|}\right\},&\text{if}\ K^2\la_z^p< a(z)\la_z^q.
    \end{cases}
\end{align*}
For every $z\in E(\La)^c$ we define a distance from $z$ to $E(\La)$ as
\begin{align}\label{app6}
	4r_z= d_z(z,E(\La))=\inf_{w\in E(\La)}d_z(z,w).
\end{align}
Using this distance, we construct a family of open subsets $\{U_z^\La\}_{z\in E(\La)^c}$ defined as
\begin{equation*}
	U_z^{\La}=
 \begin{cases}
 Q^{\la_z}_z,&\text{if}\ K^2\la_z^p\ge a(z)\la_z^q,\\
G^{\la_z}_z, &\text{if}\ K^2\la_z^p< a(z)\la_z^q,
\end{cases}
\end{equation*}
where $Q^{\la_z}_z=B_{r_z}(x)\times I^{\la_z}_{r_z}(t)$ is a $p$-intrinsic cylinder and 
$G^{\la_z}_z=B_{r_z}(x)\times J^{\la_z}_{r_z}(t)$ is a $(p,q)$-intrinsic cylinder.
The case $K^2\la_z^p\ge a(z)\la_z^q$ is called the $p$-intrinsic case and $K^2\la_z^p<a(z)\la_z^q$ the $(p,q)$-intrinsic case.

In the rest of this section we use a Vitali type argument to find a countable collection of points $z_i \in E(\La)^c$ such that the corresponding subfamily $\{U^\La_{z_i}\}_{i\in\mathbb{N}}$ satisfies the properties in Proposition~\ref{prop_whitney}. The proof of the proposition is divided into Lemmas~\ref{app:lem1} - \ref{app:lem10}. Here, and in the remaining of the paper, we denote 
\[
	\la_i=\la_{z_i},\quad r_i=r_{z_i},\quad d_i(\cdot,\cdot)=d_{z_i}(\cdot,\cdot),\quad Q_i=Q_{z_i}^{\la_i},\quad G_i=G_{z_i}^{\la_i},
\]
and
\begin{equation}\label{def_Ui_37}
U_i  = B_i \times I_i = 
 \begin{cases} 
 Q_i,&\text{if}\ K^2\lai^p \geq a(z_i)\lai^q,\\
 G_i,&\text{if}\ K^2\lai^p < a(z_i)\lai^q.
 \end{cases}
\end{equation}
Moreover, we write 
\begin{align}\label{def_I}
d_i(U_i,E(\La))=\inf_{z\in U_i,w\in E(\La)}d_i(z,w), \quad \mathcal{I}_i=\{j\in\mathbb{N}: 2U_i\cap 2U_j\ne\emptyset\} 
\end{align}
and
\begin{align}\label{def_Ki}
K_i =
\begin{cases}
    50K^2 &\text{if}\ U_i=Q_i,\\
    50K&\text{if}\ U_i=G_i.
\end{cases}
\end{align}

\begin{proposition} \label{prop_whitney} 
Let $K$ be as  in \eqref{app9} and $E(\La)$ as in \eqref{def_E}. There exists a collection $\{U_{i}\}_{i\in\mathbb{N}}$ of cylinders defined as in \eqref{def_Ui_37} satisfying the following properties:
	\begin{enumerate}[label=(\roman*),series=theoremconditions]
		\item\label{i} $\cup_{i\in\mathbb{N}}U_i=E(\La)^c$.
		\item\label{ii} $\tfrac{1}{6K}U_i\cap\tfrac{1}{6K}U_j= \emptyset$ for every $i,j\in \mathbb{N}$ with $i\ne j$.
		\item\label{iii} $3r_i\le d_i(U_i,E(\La))\le 4r_i$ for every $i\in \mathbb{N}$.
		\item\label{iv} $4U_i\subset E(\La)^c$ and $5U_i\cap E(\La)\ne\emptyset$ for every $i\in \mathbb{N}$.
		\item\label{v} $(12K)^{-1}r_j\le r_i\le 12Kr_j$ for every $i\in \NN$, $j\in \mathcal{I}_i$.
		\item\label{vi}$K^{-\frac{2}{p}}\la_j\le \la_i\le K^\frac{2}{p}\la_j$ for every $j\in \mathcal{I}_i$. Moreover, if $U_i=G_i$, then $2^{-\frac{1}{p}}\la_j\le \la_i\le 2^\frac{1}{p}\la_j$.
  \item\label{Uij}There exists a constant $c = c(n,K)$ such that $|U_i|\leq c|U_j|$, for every $i\in \mathbb{N}$ and $j\in \mathcal{I}_i$.
 \item\label{vii} $2U_j \subset K_iU_i$ for every $i\in \NN$, $j\in \mathcal{I}_i$.
 
      \item\label{UUi} If $U_i=Q_i$, then there exists a constant $c=c([a]_\alpha,\alpha,K)$ such that $r_i^\alpha\la_i^q\le c\la_i^p$.

		\item\label{viii} If $U_i=G_{i}$, then 
			$\tfrac{a(z_i)}{2}\le a(z)\le 2a(z_i)$ for every $z\in 200K Q_{r_i}(z_i)$.
		\item\label{ix} For any $i\in\mathbb{N}$, the cardinality of $\mathcal{I}_i$, denoted by $|\mathcal{I}_i|$, is finite. Moreover, there exists a constant $c= c(n,K)$, such that $|\mathcal{I}_i|\le c$.
	\end{enumerate}
	Additionally, there exists a partition of unity $\{\om_i\}_{i\in\mathbb{N}}$ subordinate to the Whitney decomposition $\{2U_i\}_{i\in \mathbb{N}}$  with the following properties:
	\begin{enumerate}[resume*=theoremconditions]
		\item\label{x} $0\le \om_i\le 1$, $\om_i\in C_0^\infty(2U_i)$ for every $i\in\mathbb{N}$ and $\sum_{i\in\mathbb{N}}\om_i=1$ on $E(\La)^c$.
		\item \label{xi}There exists a constant $c=c(n,K)$ such that $\lVert\na \om_j\rVert_{\infty}\le cr_i^{-1}$,
   for every $i\in\mathbb{N}$ and $j\in \mathcal{I}_i$.
		\item\label{xii} There exists a constant $c=c(n,K)$ such that 
  \begin{align*}
			\lVert \pa_t\om_j\rVert_{\infty}\le 
			\begin{cases}	
    cr_i^{-2}\la_i^{p-2},&\text{if}\ U_i=Q_i,\\
           cr_i^{-2}\la_i^{-2}\La,&\text{if}\ U_i=G_i,
           \end{cases}
		\end{align*}
  for every $i\in\mathbb{N}$ and $j\in \mathcal{I}_i$.
	\end{enumerate}

\end{proposition}

\subsection{Estimates for scaling factors}
In this subsection, we provide lemmas that allow us to compare the scaling factors of neighboring cylinders. The first lemma states that coefficient function $a(\cdot)$ is comparable in the standard cylinder for $(p,q)$-intrinsic case.

\begin{lemma}\label{app:lem1}Assume $z\in E(\La)^c$ and $K^2\la_z^p<a(z)\la_z^q$. Then
	$\tfrac{a(z)}{2}\le a(w)\le 2a(z)$ for every $w\in 200 K Q_{r_z}(z)$.
 
\end{lemma}
\begin{proof}
	We claim that
	\begin{align}\label{app12}
		 2[a]_{\alpha}\left(200 K r_z\right)^\alpha<a(z).
	\end{align}
 Assume for contradiction that $2[a]_{\alpha}\left(200 K r_z\right)^\alpha\ge a(z)$. Since $G^{\la_z}_z\subset E(\La)^c$ and $a(z)\la_z^q<\La$, we have
    \[
    	a(z)\la_z^q<\fiint_{G_{z}^{\la_z}}\left(M(f^d+(af^q)^\frac{d}{p})(w)\right)^\frac{p}{d}\,dw.
    \]
    Using the facts that $\La=\la_z^p+a(z)\la_z^q$ and $\la_z^p<a(z)\la_z^q$, we obtain
    \begin{align*}
        \begin{split}
            a(z)\la_z^q&<\fiint_{G_{z}^{\la_z}}\left(M(f^d+(af^q)^\frac{d}{p})(w)\right)^\frac{p}{d}\,dw\\
            &=\frac{\La\la_z^{-2}}{2|B_1|r_z^{n+2}}\iint_{G_{z}^{\la_z}}\left(M(f^d+(af^q)^\frac{d}{p})(w)\right)^\frac{p}{d}\,dw\\
            &<\frac{a(z)\la_z^{q-2}}{|B_1|r_z^{n+2}}\iint_{G_{z}^{\la_z}}\left(M(f^d+(af^q)^\frac{d}{p})(w)\right)^\frac{p}{d}\,dw.
        \end{split}
    \end{align*}
    Dividing both sides by $a(z)\la_z^{q-2}r_z^{-n-2}$ and raising to power $\tfrac{\al}{n+2}$, it follows from \eqref{app9} that
    \begin{align}\label{app:16}
        r_z^\alpha\la_z^\frac{2\alpha}{n+2}\le \left(\frac{1}{|B_1|}\iint_{\RR^{n+1}}\left(M(f^d+(af^q)^\frac{d}{p})(w)\right)^\frac{p}{d}\,dw\right)^\frac{\alpha}{n+2}\le \frac{1}{800[a]_{\alpha}}K.
    \end{align}
    Since $K^2\la_z^p<a(z)\la_z^q$ and $q\le p+\tfrac{2\alpha}{n+2}$, the counter assumption and \eqref{app:16} imply
    \[
        K^2\la_z^p<a(z)\la_z^q\le 2[a]_\alpha(200 Kr_z)^\alpha\la_z^p\la_z^\frac{2\alpha}{n+2}
         < 400[a]_\alpha K\la_z^p r_z^\alpha\la_z^\frac{2\alpha}{n+2}\le \tfrac{1}{2}K^2\la_z^p.
    \]
       This is a contradiction and thus \eqref{app12} holds true.

       By \eqref{app12} we have
\[
	2[a]_{\alpha}\left(200 K r_z\right)^\alpha<a(z)\le\inf_{w\in 200K Q_{r_z}(z)}a(w) +[a]_{\alpha}(200 Kr_z)^\alpha
\]
and thus
\[
	[a]_{\alpha}(200 Kr_z)^\alpha<\inf_{w\in200 KQ_{r_z}(z)}a(w).
\]
It follows that
\[
	\sup_{w\in 200 K Q_{r_z}(z)}a(w)\le \inf_{w\in 200 K Q_{r_z}(z)}a(w)+[a]_{\alpha}(200 K r_z)^\alpha<2\inf_{w\in 200 K Q_{r_z}(z)}a(w).
\]
This completes the proof.
\end{proof}

The second lemma states that if $a(\cdot)$ is comparable at two points, then $\la_{(\cdot)}$ is comparable at these points.
\begin{lemma}\label{app:lem2}
    Assume $z,w\in E(\La)^c$ and $\tfrac{a(z)}{2}\le a(w)\le 2a(z)$. Then $2^{-\frac{1}{p}}\la_w\le\la_z\le 2^\frac{1}{p}\la_w.$
Moreover, the above inequality holds provided $K^2\la_z^p<a(z)\la_z^q$ and $w\in 200 K Q_{r_z}(z)$.
\end{lemma}

\begin{proof}
	The second statement of the lemma follows from Lemma~\ref{app:lem1} and the first statement of this lemma. To prove the first statement, it suffices to show $\la_z\le 2^\frac{1}{p}\la_w$. Assume for contradiction that $\la_z>2^\frac{1}{p}\la_w$. Then
	\[
			\La=\la_z^p+a(z)\la_z^q>2\la_w^p+2^\frac{q}{p}a(z)\la_w^q,
	\]
    and using the assumption on $a(\cdot)$, we get
    \[
    	\La>2\la_w^p+2^{\frac{q}{p}-1}a(w)\la_w^q> \la_w^p+a(w)\la_w^q=\La.
    \]
    This is a contradiction and the proof is completed.
\end{proof}

We need the estimate between $\la_{z}$ and $\la_{w}$ also in the $p$-intrinsic case. For this, we will use the continuity of $a(\cdot)$ and the following lemma.
\begin{lemma}\label{app:lem3}
	Assume $z\in E(\La)^c$ and $K^2\la_z^p\ge a(z)\la_z^q$. Then $[a]_{\alpha}(50 K r_z)^{\alpha}\la_z^q\le (K^2-1)\la_z^p.$
\end{lemma}
\begin{proof}
	We argue as in the proof of Lemma~\ref{app:lem1}. Observe that $Q_z^{\la_z}\subset E(\La)^c$ and therefore
	\begin{align*}
        \la_z^p
        &<\fiint_{Q_{z}^{\la_z}}\left(M(f^d+(af^q)^\frac{d}{p})(w)\right)^\frac{p}{d}\,dw\\
        &=\frac{\la_z^{p-2}}{2|B_1|r_z^{n+2}}\iint_{Q_{z}^{\la_z}}\left(M(f^d+(af^q)^\frac{d}{p})(w)\right)^\frac{p}{d}\,dw.
	\end{align*}
After dividing both sides and raising to power $\tfrac{\al}{n+2}$, it follows from \eqref{app9} that
 \[
     r_z^{\alpha}\la_z^\frac{2\alpha}{n+2}
     <\left(\frac{1}{|B_1|}\iint_{\RR^{n+1}}\left(M(f^d+(af^q)^\frac{d}{p})(w)\right)^\frac{p}{d}\,dw\right)^\frac{\alpha}{n+2}\le \frac{1}{50[a]_\alpha}(K-1).
 \]
Using this and the fact that $q\le p+\tfrac{2\alpha}{n+2}$ we obtain
\begin{align*}
\begin{split}
    [a]_\alpha (50 K r_z)^\alpha \la_z^q
    &\le 50 [a]_\alpha \left(\frac{1}{|B_1|}\iint_{\RR^{n+1}}\left(M(f^d+(af^q)^\frac{d}{p})(w)\right)^\frac{p}{d}\,dw\right)^\frac{\alpha}{n+2} K\la_z^p\\
    &\le (K-1)K\la_z^p.
\end{split}
\end{align*}
Since $(K-1)K\le K^2-1$, the proof is completed.
\end{proof}

Finally, we state the comparability of $\la_{(\cdot)}$ in the $p$-intrinsic case.
\begin{lemma}\label{app:lem4}
	Assume $z\in E(\La)^c$ and $K^2\la_z^p\ge a(z)\la_z^q$. If $w\in 50 KQ_{r_z}(z)$, then 
	$\la_z\le K^\frac{2}{p}\la_w$.
\end{lemma}
\begin{proof}
	   Assume for contradiction that $\la_w<K^{-\frac{2}{p}}\la_z$. By the H\"older regularity in \eqref{15} we have
	\[
			\La=\la_w^p+a(w)\la_w^q
			\le \la_w^p+a(z)\la_w^q+[a]_{\alpha}(50 Kr_z)^{\alpha}\la_w^q.
	\]
 Therefore the counter assumption and Lemma~\ref{app:lem3} imply that
    \begin{align*}
    \begin{split}
        \La&\le \frac{1}{K^2}(\la_z^p+a(z)\la_z^q)+\frac{[a]_{\alpha}(50 Kr_z)^\alpha}{K^2}\la_z^q\\
        &\le\frac{1}{K^2}(\la_z^p+a(z)\la_z^q)+\frac{K^2-1}{K^2}\la_z^p< \la_z^p+a(z)\la_z^q=\La.
    \end{split}
    \end{align*}
    This is a contradiction and the proof is completed.
\end{proof}

\subsection{Vitali-type covering argument}
We select a countable family of intrinsic cylinders from $\{U_z^{\La}\}_{z\in E(\La)^c}$ such that they cover $E(\La)^c$ and $\tfrac{1}{6K}$-times the chosen cylinders are pairwise disjoint. We start by claiming that $r_z$, $z\in E(\La)^c$, are uniformly bounded.
Since $|E(\La)^c|<\infty$ and $\La>1+\|a\|_{L^\infty(\mathbb{R}^{n+1})}$, there exist $\la=\la(\La,\lVert a\rVert_{L^\infty(\RR^{n+1})})>1$ and $R=R(\La,\lVert a\rVert_{L^\infty(\RR^{n+1})})>0$ such that 
\begin{align}\label{app37}
	\la^p+\lVert a\rVert_{L^\infty(\RR^{n+1})}\la^q = \La
\end{align}
and $\left|B_R\times \left(-\la^2\La^{-1}R^2,\la^{2}\La^{-1}R^2\right)\right|=|E(\La)^c|$. It is easy to see that $\la\le\la_z\le \La$ for every $z\in E(\La)^c$.
Therefore, if there exists $z\in E(\La)^c$ such that $r_z>R$, then
\[
   |E(\La)^c|= \left|B_R\times \left(-\la^2\La^{-1}R^2,\la^{2}\La^{-1}R^2\right)\right|< |U_z^{\La}|\le |E(\La)^c|.
\]
This is a contradiction and we conclude that $r_z\le R$ for every $z\in E(\La)^c$.

Let $\mathcal{F}=\left\{\tfrac{1}{6K}U_z^\La\right\}_{z\in E(\La)^c}$ and 
\[
	\mathcal{F}_j=\left\{\tfrac{1}{6K}U^{\La}_z\in \mathcal{F}: \frac{R}{2^j} < r_z\le \frac{R}{2^{j-1}}\right\},
 \quad j\in\mathbb{N}.
\]
Since $r_z$, $z\in E(\La)^c$, are uniformly bounded, we have $\mathcal{F}=\cup_{j\in \mathbb{N}}\mathcal{F}_j$. 
We obtain a disjoint subcollection $\mathcal{G}_j\subset \mathcal{F}_j$ as follows. 
Let $\mathcal{G}_1$ be a maximal disjoint collection of cylinders in $\mathcal{F}_1$. 
We proceed recursively.
If $\mathcal{G}_1,\mathcal{G}_2,...,\mathcal{G}_{k-1}$ have been selected, let $\mathcal{G}_{k}$ be any maximal disjoint collection of 
\[
	\biggl\{\tfrac{1}{6K}U_z^{\La}\in \mathcal{F}_k: \tfrac{1}{6K}U_z^{\La}\cap \tfrac{1}{6K}U^{\La}_{w}=\emptyset\ \text{for every}\ \tfrac{1}{6K}U_{w}^\La\in \bigcup_{j=1}^{k-1} \mathcal{G}_j  \biggr\}.
\]
Note that for $\la$ defined in \eqref{app37} and $z\in \mathcal{G}_j$, we have
\[
	\tfrac{1}{2^j6K}\left(B_{R}(x)\times \left(t-\la^2\La^{-1}R^2,t+\la^2\La^{-1}R^2\right)\right)\subset \tfrac{1}{6K}U_z^\La\subset E(\La)^c.
\]
Since $\mathcal{G}_j$ is a pairwise disjoint collection and $|E(\La)|^c<\infty$, this implies
that the cardinality of $\mathcal{G}_j$ is finite. Let $\mathcal{G}=\cup_{j\in\mathbb{N}}\mathcal{G}_j$.
By the construction, $\mathcal{G}$ is a countable subfamily of pairwise disjoint cylinders. 

In the rest of this subsection we prove that the cylinders in $\mathcal{G}$ cover $E(\La)^c$. 
Let $\tfrac{1}{6K}U_{w}^\La\in \mathcal{F}$. There exists $i\in\mathbb{N}$ such that $\tfrac{1}{6K}U_{w}^\La\in \mathcal{F}_i$ and by the maximal disjointness of $\mathcal{G}_j$, there exists a cylinder $\tfrac{1}{6K}U_{z}^\La\in \cup_{j=1}^{i}\mathcal{G}_j$ such that
\begin{align}\label{app46}
	\tfrac{1}{6K}U_{w}\cap \tfrac{1}{6K}U_{z}\ne\emptyset.
\end{align}
Note that since $\tfrac{1}{6K}U_{w}\in  \mathcal{F}_i$ and $\tfrac{1}{6K}U_{z}\in \mathcal{F}_j$ for some $j\in\{1,\dots,i\}$, we have
\begin{align}\label{app47}
	 r_{w}\le2r_{z}.
\end{align}

We claim that $6K$ is the constant in the Vitali-type covering, that is,
\begin{align}\label{app48}
	\tfrac{1}{6K}U_{w}\subset U_{z}.
\end{align}
Let $z=(x,t)$ and $w=(y,s)$. For the inclusion in the spatial direction, observe that
\eqref{app46} implies
\[
	\tfrac{1}{6K}B_{r_w}(y)\cap \tfrac{1}{6K}B_{r_z}\ne\emptyset.
\]
From \eqref{app47} and the proof of the standard Vitali covering lemma, we obtain
\[
	\tfrac{1}{6K}B_{r_w}(y)\subset \tfrac{5}{6K}B_{r_z}(x)\subset B_{r_z}(x).
\]
On the other hand, it follows from \eqref{app46} that $Q_{r_w}(w)\cap  Q_{r_z}(z)\ne \emptyset$.
Again, it easily follows from the standard Vitali covering argument with \eqref{app47} that $Q_{r_w}(w)\subset 5Q_{r_z}(z)$. In particular, we have
\begin{align}\label{app51}
	w\in 5Q_{r_z}(z)\subset 50K Q_{r_z}(z).
\end{align}

It remains to prove the inclusion in \eqref{app48} in the time direction. We divide the proof into four cases.

\textit{Case 1: $U_{w}^{\La}=Q^{\la_w}_{w}$ and $U_{z}^{\La}=Q^{\la_z}_{z}$.} 
By \eqref{app47}, we get
\begin{align*}
	|\tau-t|&\le |\tau-s|+|s-t|
	\le\left|\tfrac{1}{6K}I_{r_w}^{\la_w}(s)\right|+\tfrac{1}{2}\left|\tfrac{1}{6K}I_{r_z}^{\la_z}(t)\right| \\
  &=2\la_w^{2-p}\left(\frac{r_w}{6K}\right)^2+\la_z^{2-p}\left(\frac{r_z}{6K}\right)^2\\
  &\le (8\la_w^{2-p}+\la_z^{2-p})\left(\frac{r_z}{6K}\right)^2
\end{align*}
for every $\tau\in \tfrac{1}{6K}I_{r_w}^{\la_w}(s)$.
Lemma~\ref{app:lem4} and \eqref{app51} imply $\la_z\le K^\frac{2}{p}\la_w$ and thus we have
\begin{align*}
	|\tau-t|&\le (8K^\frac{2(p-2)}{p}+1)\la_z^{2-p}\left(\frac{r_z}{6K}\right)^2\\
 &\le (9K^2)\la_z^{2-p}\left(\frac{r_z}{6K}\right)^2\le \la_z^{2-p}r_z^2=|I_{r_z}^{\la_z}(t)|,
\end{align*}
where we also applied the fact that $\tfrac{p-2}{p}\le 1$. Therefore, the conclusion in \eqref{app48} holds in this case.

\textit{Case 2: $U^{\La}_{w}=G^{\la_w}_{w}$ and $U^{\La}_{z}=Q^{\la_z}_{z}$.} In this case we have $G^{\la_w}_{w}\subset Q^{\la_w}_{w}$. The conclusion follows from the argument in the previous case.

\textit{Case 3: $U^{\La}_{w}=Q^{\la_w}_{w}$ and $U^{\La}_{z}=G^{\la_z}_{z}$.} 
By \eqref{app47} we get
\begin{align*}
		|\tau-t|&\le |\tau-s|+|s-t|
		\le\left|\tfrac{1}{6K}I_{r_w}^{\la_w}(s)\right|+\tfrac{1}{2}\left|\tfrac{1}{6K}J^{\la_z}_{r_z}(t)\right| \\
  &=2\la_w^{2-p}\left(\frac{r_w}{6K}\right)^2+\la_z^{2}\La^{-1}\left(\frac{r_z}{6K}\right)^2\\
  &\le \left(8\la_w^{2-p}+\la_z^2\La^{-1}\right)\left(\frac{r_z}{6K}\right)^2
\end{align*}
for every $\tau\in \tfrac{1}{6K}I_{r_w}^{\la_w}(s)$.
In this case we have $a(w)\la_w^q\le K^2\la_w^p$, and thus we obtain
\[
		\la_w^{2-p}
		=2\frac{\la_w^2}{\la_w^p+\la_w^p}\le 2K^2\frac{\la_w^2}{\la_w^p+a(w)\la_w^q}=2K^2\la_w^2\La^{-1}.
\]
Lemma~\ref{app:lem2} and \eqref{app51} imply that
\[
	2K^2\la_w^2\La^{-1}\le 2K^2(2^\frac{1}{p}\la_z)^2\La^{-1}\le 4K^2\la_z^2\La^{-1}.
\]
It follows that
\[
	|\tau-t|\le (32K^2+1)\la_z^2\La^{-1}\left(\frac{r_z}{6K}\right)^2\le 33K^2\la_z^2\La^{-1}\left(\frac{r_z}{6K}\right)^2\le |J_{r_z}^{\la_z}(t)|,
\]
and thus \eqref{app48} holds in this case.

\textit{Case 4: $U^{\La}_{w}=G^{\la_w}_{w}$ and $U^{\La}_{z}=G^{\la_z}_{z}$.}  
By \eqref{app47} we get
\begin{align*}
		|\tau-t|&\le |\tau-s|+|s-t|
		\le\left|\tfrac{1}{6K}J_{r_w}^{\la_w}(s)\right|+\tfrac{1}{2}\left|\tfrac{1}{6K}J^{\la_z}_{r_z}(t)\right| \\
 &=2\la_w^{2}\La^{-1}\left(\frac{r_w}{6K}\right)^2+\la_z^{2}\La^{-1}\left(\frac{r_z}{6K}\right)^2\\
 &\le \left(8\la_w^2\La^{-1}+\la_z^2\La^{-1}\right)\left(\frac{r_z}{6K}\right)^2
\end{align*}
for every $\tau\in \tfrac{1}{6K}I_{r_w}^{\la_w}(s)$.
Applying \eqref{app51} and Lemma~\ref{app:lem2} gives
\[
    |\tau-t|\le 17\la_z^2\La^{-1}\left(\frac{r_z}{6K}\right)^2\le |J_{r_z}^{\la_z}(t)|.
\]
This completes the proof of \eqref{app48} in the final case. Thus we have found a countable covering family $\{U_i\}_{i\in\mathbb{N}}$ of intrinsic cylinders defined as in \eqref{def_Ui_37} and with $\tfrac{1}{6K}U_i$ being pairwise disjoint.

\subsection{Properties of the covering}  
In this subsection we show that the Whitney decomposition $\{U_i\}_{i\in\mathbb{N}}$ defined in the previous subsection satisfies the remaining properties  in Proposition~\ref{prop_whitney} 
\begin{lemma}\label{app:lem5}
We have $3r_i\le d_i(U_i,E(\La))\le 4r_i$ for every $i\in\mathbb{N}$.
\end{lemma}
\begin{proof}
	From the choice of $r_i$ in \eqref{app6}, we have
	\[
		d_{i}(U_i,E(\La))\le d_{i}(z_i,E(\La))=4r_i.
	\]
    Therefore the second inequality holds. To show the first inequality, note that for any $z\in U_i$, triangle inequality gives
    \[
    	d_{i}(z,E(\La))\ge d_i(z_i,E(\La))-d_{i}(z,z_i)\ge 4r_i-r_i=3r_i.
    \]
     Taking infimum over $z\in U_i$, we get $d_i(U_i,E(\La))\ge 3r_i$.
    This completes the proof.
\end{proof}

Next we will discuss properties \ref{v}-\ref{vii} in Proposition~\ref{prop_whitney}. First we show that $r_{j}$, $\la_{j}$ and $d_j(\cdot,\cdot)$ are comparable in $j\in\mathcal{I}_i$, defined as in \eqref{def_I}, by using the lemmas in Section~\ref{subsec_app_1}.
\begin{lemma}\label{app:lem6}
	We have $(12K)^{-1}r_j\le r_i\le 12K r_j$ for every $i\in \NN$ and $j\in \mathcal{I}_i$.
	Moreover, if $U_i=G_i$, then $r_j\le 12 r_i$.
\end{lemma}

\begin{proof}
	To prove the first statement it is enough to show $r_i\le 12K r_j$ as there is no difference in the roles of $i$ and $j$. Let $w\in 2U_i\cap 2U_j$. Using Lemma~\ref{app:lem5} and triangle inequality, we have
	\[
		3r_i\le d_{i}(z_i,E(\La))\le d_{i}(z_i,w)+d_i(w,E(\La))\le 2r_i+d_i(w,E(\La)).
	\]
    On the other hand, Lemma~\ref{app:lem5} and triangle inequality give
    \[
    	d_{j}(w,E(\La))\le d_j(w,z_j)+d_j(z_j,E(\La))\le 2r_j+4r_j=6r_j.
    \]
    Combining these estimates, we have
    \begin{align}\label{app76}
    	r_i\le d_i(w,E(\La))\quad \text{and}\quad d_{j}(w,E(\La))\le 6r_j.
    \end{align}
    To finish the proof we consider four cases to compare the $d_i(z,w)$ and $d_j(z,w)$.
      
    \textit{Case 1: $U_i=Q_i$ and $U_j=Q_j$.} Using $\la_i^p\le \La$ and $p\ge2$, we have
    \[
    	d_{i}(z,w)
    	=\max\left\{|x-y|,\sqrt{\la_i^{p-2}|s-t|}\right\}\le \max\left\{|x-y|,\sqrt{\La^\frac{p-2}{p}|s-t|}\right\}.
    \]
    Since $\La=\la_j^p+a(z_j)\la_j^q$ and $a(z_j)\la_j^q\le K^2\la_j^p$, we have
    \[
    	\La^\frac{p-2}{p}=(\la_j^p+a(z_j)\la_j^q)^\frac{p-2}{p}\le (2K^2\la_j^p)^\frac{p-2}{p}
    \]
and therefore
    \[
    		d_{i}(z,w)
    		\le (2K)^\frac{p-2}{p}\max\left\{|x-y|,\sqrt{\la_j^{p-2}|s-t|}\right\}\le 2K  d_{j}(z,w).
    \]
    The above inequality and \eqref{app76} imply that $r_i\le 12K r_j$. 
    
\textit{Case 2: $U_i=Q_i$ and $U_j=G_j$.} Since $\la_i^p+a(z_i)\la_i^q =\la_j^p+a(z_j)\la_j^q=\La$ and $p\ge2$, we have
\[
	\la_i^{p-2}\le \La^{\frac{p-2}{p}}=\La(\la_j^p+a(z_j)\la_j^q)^{-\frac{2}{p}}\le \La\la_j^{-2}.
\]
    Thus, it follows that
    \begin{align*}
    	d_i(z,w)&=\max\left\{|x-y|,\sqrt{\la_i^{p-2}|s-t|}\right\}\\
     &\le \max\left\{|x-y|,\sqrt{\La\la_j^{-2}|s-t|}\right\}=d_j(z,w),
    \end{align*}
    and the above inequality along with \eqref{app76} gives $ r_i\le 6r_j$.

    \textit{Case 3: $U_i=G_i$ and $U_j=Q_j$.}
    Note that $2U_i\cap 2U_j\ne\emptyset$ implies $2Q_{r_i}(z_i)\cap 2Q_{r_j}(z_j)\ne\emptyset$.  This fact together with the previous case implies that $z_j\in 14Q_{r_i}(z_i)$. It follows from Lemma~\ref{app:lem2} that $\la_j\le 2^\frac{1}{p}\la_i$, which, together with the fact that $K^2\la_j^p\ge a(z_j)\la_j^q$, gives
     \[
     	\La\la_i^{-2}= (\la_j^p+a(z_j)\la_j^q)\la_i^{-2}\le 2K^2\la_j^p\la_i^{-2}\le 4K^2\la_j^{p-2}.
     \]
     Therefore, we obtain 
     \begin{align*}
     	d_i(z,w)&=\max\left\{|x-y|,\sqrt{\La\la_i^{-2}|s-t|}\right\}\\
      &\le 2K\max\left\{|x-y|,\sqrt{\la_j^{p-2}|s-t|}\right\}=2Kd_j(z,w)
     \end{align*}
     and conclude with \eqref{app76} that $r_i\le 12K r_j$.
    
\textit{Case 4: $U_i=G_i$ and $U_j=G_j$.} We may assume $r_j\le r_i$. Then $z_j\in 4Q_{r_i}(z_i)$ and therefore Lemma~\ref{app:lem2} gives
$2^{-\frac{1}{p}}\la_j\le \la_i\le 2^\frac{1}{p}\la_j$.
Thus
\begin{align*}
	\begin{split}
		d_i(z,w)
		&=\max\left\{|x-y|,\sqrt{\La\la_i^{-2}|t-s|}\right\}\\
  &\le 2 \max\left\{|x-y|,\sqrt{\La\la_j^{-2}|t-s|}\right\}=2d_j(z,w).
	\end{split}
\end{align*}
Combining this estimate with \eqref{app76}, we obtain $ r_i\le 12r_j$ in the final case. The second statement of the lemma follows the second and fourth cases and we are finished.
\end{proof}

It follows from the previous lemma that $2Q_{r_j}(z_j)\subset 50 KQ_{r_i}(z_i)$ for all $j\in \mathcal{I}_i$. We write again the conclusions of Lemma~\ref{app:lem2} and Lemma~\ref{app:lem4}.

\begin{lemma}\label{app:lem7}
We have $K^{-\frac{2}{p}}\la_j\le \la_i\le K^\frac{2}{p}\la_j$ for every $i\in \NN$ and $j\in \mathcal{I}_i$. Moreover, if $U_i=G_i$, then $2^{-\frac{1}{p}}\la_j\le \la_i\le 2^\frac{1}{p}\la_j$.
\end{lemma}
\begin{proof}
    It is enough to prove the first inequality of the second statement. If $U_j=G_j$, then the first statement gives $\la_j\le 2^\frac{1}{p}\la_i \le K^\frac{2}{p}\la_i$. If $U_j=Q_j$, then $\la_j\le K^\frac{2}{p}\la_i$ by Lemma~\ref{app:lem4}, since $z_i\in 50KQ_{r_j}(z_j)$.
\end{proof}

It follows from the previous two lemmas that the measures of neighboring cylinders are comparable.
\begin{lemma}\label{app:lem_Uij}
    There exists $c = c(n,K)$ such that $ |U_i|\leq c|U_j|$ for any $i\in \mathbb{N}$ and $j\in \mathcal{I}_i$.
\end{lemma}
\begin{proof}
  If $U_j=Q_j$, it follows from $I_i\subset I_{r_i}^{\lai}(t_i)$, Lemma~\ref{app:lem6} and Lemma~\ref{app:lem7} that
  \[
 |U_i| \leq |Q_i| = 2|B_1|r_i^{n+2}\lai^{2-p}\leq 2|B_1|(12K)^{n+2}K^\frac{2(p-2)}{p} r_j^{n+2}\la_j^{2-p}\leq (12K)^{n+4}|Q_j|.
  \]
The case $U_i=G_i$ and $U_j=G_j$ follows similarly from Lemma~\ref{app:lem6} and Lemma~\ref{app:lem7}.

Finally, if $U_i=Q_i$ and $U_j=G_j$, we have by $a(z_i)\lai^q \leq K^2\lai^p$, Lemma~\ref{app:lem6} and Lemma~\ref{app:lem7} that
\begin{align*}
\begin{split}
    |Q_i| &=  2|B_1|r_i^{n+2}\lai^{2-p} = \frac{4|B_1|r_i^{n+2}\lai^2}{\lai^p+\lai^p}\leq \frac{4K^2|B_1|r_i^{n+2}\lai^2}{\lai^p+a(z_i)\lai^q}\\
    &\leq \frac{8\cdot(12)^{n+2}K^2|B_1|r_j^{n+2}\laj^2}{\La} = c(n,K)|G_j|.
\end{split}
\end{align*}
This completes the proof.
\end{proof}

In the following two lemmas we consider the inclusion property of $U_j$ for $j\in\mathcal{I}_i$. The first lemma is for the $p$-intrinsic case and the second lemma is for the $(p,q)$-intrinsic case. We remark that the covering constant for $p$-intrinsic case is $50K^2$ while for $(p,q)$-intrinsic case it is $50K$. 
\begin{lemma}\label{app:lem8}
	Let $i\in \NN$ be such that $U_i=Q_i$. We have $2U_j\subset 50K^2Q_i$ for every $j\in \mathcal{I}_i$.
\end{lemma}

\begin{proof}
	Since $2B_i\cap 2B_j\ne \emptyset$ and Lemma~\ref{app:lem6} gives $r_j\le 12K r_i$, it is easy to see that $2B_j\subset 50K B_i$. It remains to show the inclusion in the time direction. As $I_j\subset I_{r_j}^{\la_j}(t_j)$, we have for $\tau\in 2I_j$  that
    \[
    		|\tau-t_i|\le |\tau-t_j|+|t_j-t_i|\le |2I_{r_j}^{\la_j}(t_j)|+\tfrac{1}{2}|2I_{r_i}^{\la_i}(t_i)|=2\la_j^{2-p}\left(2r_j\right)^2+\la_i^{2-p}\left(2r_i\right)^2.
    \]
	Applying $r_j\le 12K r_i$ and $\la_i\le K^\frac{2}{p}\la_j$, which follow from Lemma~\ref{app:lem6} and Lemma~\ref{app:lem7}, we get
	\[
		|\tau-t_i|\le 2(K^{-\frac{2}{p}}\la_i)^{2-p}\left(24Kr_i\right)^2+\la_i^{2-p}\left(2r_i\right)^2=(2\cdot 24^2\cdot K^\frac{2(p-2)}{p}K^2+2^2)\la_i^{2-p}r_i^2.
	\]
    Since $\tfrac{2(p-2)}{p}+2\le 4$ and
    \[
    	2\cdot 24^2\cdot K^\frac{2(p-2)}{p}K^2+2^2\le (48K^2)^2+2^2\le (48K^2+2)^2\le (50K^2)^2,
    \]
    we get $|\tau-t_i|\le \la_i^{2-p}(50K^2r_i)^2$ and thus $2U_j\subset 50K^2 Q_i$. 
\end{proof}

\begin{lemma}\label{app:lem9}
Let $i\in \NN$ be such that $U_i=G_i$. We have $2U_j\subset 50K G_i$ for every $j\in \mathcal{I}_i$.
\end{lemma}

\begin{proof}
From Lemma~\ref{app:lem6}, we have $r_j\le 12r_i$.
    It follows that $2B_j\subset 50B_i$.
    To prove the inclusion in the time direction, we suppose first that $U_j=Q_j$. For $\tau\in 2I_{r_j}^{\la_j}(t_j)$, we have
   \[
   	|\tau-t_i|\le |\tau-t_j|+|t_j-t_i|\le |2I_{r_j}^{\la_j}(t_j)|+\tfrac{1}{2}|2J_{r_i}^{\la_i}(t_i)|=2\la_j^{2-p}\left(2r_j\right)^2+\la_i^{2}\La^{-1}\left(2r_i\right)^2.
   \]
   Since $K^2\la_j^p\ge a(z_j)\la_j^q$, $\La=\la_j^p+a(z_j)\la_j^q$ and Lemma~\ref{app:lem7} implies $\la_j\le 2^\frac{1}{p}\la_i$, we have
   \[
   	\la_j^{2-p}=2\la_j^2\frac{1}{\la_j^p+\la_j^p}\le 2K^2\la_j^2\frac{1}{\la_j^p+a(z_j)\la_j^q}\le 2K^2\la_j^2\La^{-1}\le 2\cdot2^\frac{2}{p}K^2\la_i^{2}\La^{-1},
   \]
   and using that $r_j\le 6r_i$ by the second case in the proof of Lemma~\ref{app:lem6}, we obtain
   \[
   	|\tau-t_i|\le (48^2K^2+2^2)\la_i^2\La^{-1}r_i^2\le \la_i^2\La^{-1}(50Kr_i)^2.
   \]
   Therefore, $2Q_j\subset 50K G_i$. On the other hand, if $U_j=G_j$, then for any $\tau \in 2J_{r_j}^{\la_j}(t_j)$ we have
   \[
       |\tau-t_i|\le |2J_{r_j}^{\la_j}(t_j)|+\tfrac{1}{2}|2J_{r_i}^{\la_i}(t_i)|=2\la_j^2\La^{-1}(2r_j)^2+\la_i^2\La^{-1}(2r_i)^2.
   \]
  Using $r_j\le 12r_i$ and $\la_j^2\le 2^\frac{2}{p}\la_i^2\le 2\la_i^2$, we get
  \[
      |\tau-t_i|\le (48^2+2^2)\la_i^2\La^{-1}r_i^2\le \la_i^2\La^{-1}(50r_i)^2.
  \]
Therefore, $2G_j\subset 50K G_i$ and the proof is completed.
   
\end{proof}

Finally, we show that the cardinality of $\mathcal{I}_i$, denoted by $|\mathcal{I}_i|$, is uniformly bounded.
\begin{lemma}\label{app:lem10}
There exists a constant $c=c(n,K)$ such that $|\mathcal{I}_i|\le c$ for every $i\in\mathbb{N}$.
\end{lemma}
\begin{proof}
 Observe that $2U_j\subset 50K^2 U_i$ for all $j\in\mathcal{I}_i$ by Lemma~\ref{app:lem8} and Lemma~\ref{app:lem9}. Since the cylinders $\left\{\tfrac{1}{6K}U_j\right\}_{j\in \mathcal{I}_i}$ are disjoint, we obtain with Lemma~\ref{app:lem_Uij} that
     \[
     		|50K^2 U_i|\ge \biggl|\bigcup_{j\in \mathcal{I}_i}\tfrac{1}{6K}U_j\biggr|
       =\sum_{j\in \mathcal{I}_i}\left|\tfrac{1}{6K}U_j\right|
       \ge \sum_{j\in \mathcal{I}_i} c(n,K)\left|U_i\right|
       =c(n,K)|\mathcal{I}_i|\left|U_i\right|,
     \]
     and thus $|\mathcal{I}_i| \leq c(n,K)$.
\end{proof}

\subsection{Partition of unity} To finish the proof of Proposition~\ref{prop_whitney}, we construct a partition of unity subordinate to the Whitney decomposition $\{2U_i\}_{i\in \mathbb{N}}$.
For each $i\in\mathbb{N}$, there exists $\psi_i\in C_0^\infty(2U_i)$ such that
\[
	0\le \psi_i\le 1,\quad \psi_i\equiv1\ \text{in}\ U_i,\quad \lVert \na\psi_i\rVert_{\infty}\le 2r_i^{-1}
\]
and
\begin{align*}
	\lVert \pa_t\psi_i\rVert_{\infty}\le
	\begin{cases}
		2\la_i^{p-2}r_i^{-2},&\text{if}\ U_i=Q_i,\\
		2\La\la_i^{-2}r_i^{-2},&\text{if}\ U_i=G_i.
	\end{cases}
\end{align*}
Since $E(\La)^c\subset\bigcup_{i\in \mathbb{N}}2U_{i}$ and $|\mathcal{I}_i|$ is finite for each $i\in \mathbb{N}$, the function
\[
	\om_i(z)= \frac{\psi_i(z)}{\sum_{j\in\mathbb{N}}\psi_j(z)}=\frac{\psi_i(z)}{\sum_{j\in\mathcal{I}_i}\psi_j(z)}
\]
is well-defined and satisfies
\[
	\om_i\in C_0^\infty(2U_i),\quad0\le \om_i\le 1\ \text{in}\ 2U_i,\quad\sum_{i\in\mathbb{N}}\om_i(z)\equiv1.
\]
The family $\{\om_i\}_{i\in \mathbb{N}}$ is a partition of unity subordinate to $\{2U_i\}_{i\in\mathbb{N}}$. It follows from Lemma~\ref{app:lem6}, Lemma~\ref{app:lem7} and Lemma~\ref{app:lem10} that there exists $c=c(n,K)$ such that
\begin{align*}
	\lVert \na\om_i\rVert_{\infty}\le cr_i^{-1}\quad\text{and}\quad\lVert\pa_t\om_i\rVert_{\infty}\le \begin{cases}
		c\la_i^{p-2}r_i^{-2},&\text{if}\ U_i=Q_i,\\
		c\La\la_i^{-2}r_i^{-2},&\text{if}\ U_i=G_i.
	\end{cases}
\end{align*}
Moreover, for any $j\in \mathcal{I}_i$ we have
\begin{align*}
	\lVert \na\om_j\rVert_{\infty}\le cr_i^{-1}\quad\text{and}\quad\lVert \pa_t \om_j\rVert_{\infty}\le
	\begin{cases}
		c\la_i^{p-2}r_i^{-2},&\text{if}\ U_i=Q_i,\\
		c\La\la_i^{-2}r_i^{-2},&\text{if}\ U_i=G_i.
	\end{cases}
\end{align*}
As the proof of the above display is based on the comparability of the scaling factors of $i$ and $j$, which is already proved in the proof of Lemma~\ref{app:lem6}, we omit the details.

\section{Lipschitz truncation} \label{sec_liptr}
In this section we discuss the Lipschitz truncation method which allows us to apply weak solutions to \eqref{11} as test functions in Definition~\ref{def_weak}.
The main idea of the Lipschitz truncation technique is to keep a function unchanged in the good set $E(\La)$, defined as in \eqref{def_E}, and redefine it in the bad set $\RR^{n+1}\setminus E(\La)$ by applying the partition of unity  in Proposition~\ref{prop_whitney}.  
The properties of the Lipschitz truncation are summarized in Proposition~\ref{prop_liptr1}.
For simplicity we have written the argument in this section with respect to Theorem~\ref{thm_cac}, while the modifications to show Theorem~\ref{main2} are explained in Section~\ref{subsec_uni}. 
In this section $u$ is a weak solution to \eqref{11} and we consider the domain $\Qo(z_0)$ and $R_1,S_1$ as in the statement of Theorem~\ref{thm_cac}. 

\subsection{Construction of the Lipschitz truncation} Let $f\in L^p(\RR^{n+1})$ be 
\[
    f=(|\na u|+ |u-u_0| + |F|)\chi_{Q_{R_2,S_2}(z_0)},
    \quad u_0=(u)_{Q_{R_2,S_2}(z_0)},
\]
where $u$ and $F$ are extended to zero outside $\Qo(z_0)$.
By Definition~\ref{def_weak} the assumption \eqref{t_31} is satisfied. With respect to $f$, let $\La$, $E(\La)$ and $K$ be defined as in Section~\ref{subsec_app_1}. There exists a Whitney covering $\{U_{i}\}_{i\in\mathbb{N}}$ as in Proposition~\ref{prop_whitney}

 Let $0<h_0<\tfrac{S_2-S_1}{4}$ be small enough and $\eta\in C_0^\infty(B_{R_2}(x_0))$ and $\zeta\in C_0^\infty(\ell_{S_2-h_0}(t_0))$ be standard cutoff functions satisfying $0\le \eta\le 1$, $0\le \zeta\le 1$, $\eta\equiv 1\ \text{in}\ B_{R_1}(x_0)$, $\zeta\equiv 1\ \text{in}\ \ell_{S_1}(t_0)$ and
\begin{align}\label{42_1}
\lVert\na \eta\rVert_{\infty}\le \frac{3}{R_2-R_1},\quad
\lVert\pa_t\zeta\rVert_{\infty}\le \frac{3}{S_2-S_1}.
\end{align}
For $0<h<h_0$, we would like to apply
\[
    v_h = [u-u_0]_h\eta\zeta
\]
as a test function in the proof of the energy estimate.
We define the Lipschitz truncation of $\vh$ as
\begin{align}\label{t_421}
	\vhla(z)=\vh(z)-\sum_{i\in\mathbb{N}}(\vh(z)-v_h^i)\om_i(z),
\end{align}
where $z\in \RR^{n+1}$ and 
\begin{align*}
	v_h^i=
	\begin{cases}
		(\vh)_{2U_i},&\text{if}\ 2U_i\subset Q_{R_2,S_2}(z_0),\\
		0,&\text{otherwise}.
	\end{cases}
\end{align*}
As $\vh$ vanishes outside $Q_{R_2,S_2}(z_0)$, the definition of $\vh^i$ implies that $\vhla\equiv 0$ in  $Q_{R_2,S_2}(z_0)^c$. Note as well that each term on the right-hand side of \eqref{t_421} belongs to $W^{1,2}(\ell_{S_2-h_0}(t_0);L^2(B_{R_2}(x_0),\RR^N))$ and by Proposition~\ref{prop_whitney}\,\ref{ix} only finitely many of these terms are nonzero on a neighborhood of any point. Thus 
\[
v_h^\La \in W_0^{1,2}(\ell_{S_2-h_0}(t_0);L^2(B_{R_2}(x_0),\RR^N)).
\]
Similarly, we denote
\begin{align}\label{t_421_2}
	v=(u-u_0)\eta\zeta 
 \quad\text{and}\quad
 \vla(z)=v(z)-\sum_{i\in\mathbb{N}}(v-v^i)\om_i,
\end{align}
where
\begin{align}\label{422_2}
	v^i=
	\begin{cases}
		(v)_{2U_i},&\text{if}\ 2U_i\subset Q_{R_2,S_2}(z_0),\\
		0,&\text{otherwise}.
	\end{cases}
\end{align}

\begin{remark} 
If $z \in E(\La)^c$, then $z \in 2U_i$ for some $i\in \mathbb{N}$ by Proposition~\ref{prop_whitney}\,\ref{i}. By Proposition~\ref{prop_whitney}\,\ref{x} we therefore have
\[
    \vhla(z)=\vh(z)-\vh(z)\sum_{j\in\mathbb{N}}\om_j+ \sum_{j\in\mathbb{N}}\vh^j\om_j(z)= \sum_{j\in\mathbb{N}}\vh^j\om_j(z) =\sum_{j\in \mathcal{I}_i}\vh^j\om_j(z).
\]
On the other hand, if $z\in E(\La)$, then $\om_j(z)=0$ for every $j\in \mathbb{N}$ by Proposition~\ref{prop_whitney}\,\ref{i}. 
Thus 
\begin{align} \label{t_325_1}
	v_h^\Lambda(z) =
	\begin{cases}
		\sum_{j\in \mathcal{I}_i} \vh^j\om_j(z), &\text{if}\ z\in 2U_i\ \text{for some}\ i\in \mathbb{N},\\
		\vh(z),&\text{if}\ z\in E(\La).  
	\end{cases}
\end{align}
As the cardinality of $\mathcal{I}_i$ is finite, this implies that for any $z$ in $E(\La)^c$ there is a neighborhood where $\vhla$ is a finite linear combination of smooth functions and therefore $\vhla\in C^\infty(E(\La)^c)$. By the same argument without the Steklov average, we have $\vla\in C^\infty(E(\La)^c)$. 
\end{remark}

\subsection{Fundamental estimates}
In this subsection we provide lemmas used in Section~\ref{subsec_liptr_props} to estimate terms containing $\vhla$ by $\La$. To achieve the desired convergence of the Lipschitz truncation, we need estimates based on \eqref{claim_subintrinsic2}. However, \eqref{claim_subintrinsic2} does not hold for $4K_i[U_i]_h$ as $a$ might not be comparable in $[U_i]_h$. On the other hand, to estimate $v_h^\La$ by $f$ we need to increase the domain of integration by $h$ because of Lemma~\ref{lem21}. Therefore, we show the Lipschitz regularity of $\vhla$ with crude estimates based on \eqref{claim_subintrinsic1}. This allows us to use $\vhla$ as a test function, and after taking the limit $h \to 0$ we can use for $\vla$ the finer estimates based on \eqref{claim_subintrinsic2}.

Apart from estimating $\La$, the following lemma is a standard consequence of the definition of $E(\La)$ and the Whitney decomposition. Unless specified otherwise, in this section the constants $c$ depend only on
\[
    \data = n,N,p,q,\alpha,\nu,L,[a]_{\alpha},R_1,R_2,S_1,S_2,K,\|a\|_{L^\infty(\mathbb{R}^{n+1})}.
\]

\begin{lemma} \label{lem_subintrinsic}
   Let $1\leq s\leq d$. For any cylinder $Q\subset \mathbb{R}^{n+1}$ such that  $Q \cap E(\La) \neq \emptyset$, we have 
	\begin{align}\label{claim_subintrinsic1}
	    \begin{split}
	        \fiint_{Q}f^s \,dz &\le \La^\frac{s}{p}. 
	    \end{split}
	\end{align}  
Moreover, there exists a constant $c=c(\data)$ such that
\begin{align}\label{claim_subintrinsic2}
	    \begin{split}
	        \fiint_{4\uik} f^s\,dz&\le c\lai^s, 
	    \end{split}
	\end{align}
 where $K_iU_i$ is defined in \eqref{def_Ui_37} and \eqref{def_Ki}.
\end{lemma}

\begin{proof}
	Let $w \in Q \cap E(\La)$. By the definitions in \eqref{def_E} and \eqref{strong_M}, we have
	\begin{align*}
	        \fiint_Qf^s\,dz &\le \left(\fiint_{Q}f^d\,dz\right)^{\frac{s}{d}} 
         \le \left(M(f^d)(w)\right)^{\frac{s}{d}}\\
         &\leq \left(M(f^d+(af^q)^{\frac{d}{p}})(w)\right)^{\frac{s}{d}} \leq \La^{\frac{s}{p}}, 
	\end{align*}
which proves the first claim. Since $4K_iU_i \cap E(\La) \neq \emptyset$ by Proposition~\ref{prop_whitney}\,\ref{iv}, it follows that 
\[
    \fiint_{4K_iU_i}f^s \,dz \le \La^\frac{s}{p}.
\]
If $U_i = Q_i$, we have $a(z_i)\lai^q\leq K^2\lai^p$ and therefore $\La^{\frac{s} {p}}\leq((K^2+1)\lai^p)^{\frac{s}{p}}=c\lai^s$, which finishes the proof in the $p$-intrinsic case. On the other hand, if $U_i=G_i$, we have by Proposition~\ref{prop_whitney}\,\ref{viii} that
\[
        [a(z_i)]^{\frac{d}{p}}\fiint_{4\uik} (f^q)^\frac{d}{p}\,dz=\fiint_{4\uik} (a(z_i)f^q)^\frac{d}{p}\,dz \leq 2^\frac{d}{p}\fiint_{4\uik} (a(z)f^q)^\frac{d}{p}\,dz.
\]
Again, by Proposition~\ref{prop_whitney}\,\ref{iv} there exists $w \in 4K_iU_i \cap E(\La)$ and we get 
\[        \fiint_{4\uik} (a(z)f^q)^\frac{d}{p}\,dz   
        \leq M((af^q)^\frac{d}{p})(w)
        \leq M(f^d+(af^q)^\frac{d}{p})(w)
        \leq \La^\frac{d}{p}.
\]
Because $U_i=G_i$, we have $\lai^p<a(z_i)\lai^q$ and it follows that
\begin{align*}
        a(z_i)^{\frac{d}{p}}\fiint_{4\uik} (f^q)^\frac{d}{p}\,dz 
        &\leq c\La^\frac{d}{p}
        =c(\lai^p+a(z_i)\lai^q)^\frac{d}{p}\\
        &\leq c(a(z_i)\lai^q)^\frac{d}{p}
        =ca(z_i)^{\frac{d}{p}}\lai^{\frac{qd}{p}}.
\end{align*}
As $a(z_i)>0$, this implies
\[
    \fiint_{4\uik} f^\frac{qd}{p}\,dz \leq c \lai^{\frac{qd}{p}} 
\]
and we have
\[
    \fiint_{4\uik} f^s\,dz \leq \left(\fiint_{4\uik} f^\frac{qd}{p}\,dz\right)^{\frac{sp}{qd}}\leq c \lai^s,
\]
which completes the proof.
\end{proof}
 A key tool in the estimates for the Lipschitz truncation is the Poincar\'e type inequality in Lemma~\ref{lem_intrinsic_Poincare}. To prove it, we use Lemma~\ref{lem_subintrinsic} and the following lemma which is a standard consequence of elliptic Poincar\'e inequality and the fact that $u-u_0$ is a weak solution to \eqref{11}.

\begin{lemma}\label{lem32}  Let $Q = B_r\times \ell_s$ be a cylinder defined as in \eqref{def_ell} and satisfying $B_r\subset B_{R_2}(x_0)$.
\begin{itemize}
    \item[(i)]  There exists a constant $c = c(\data)$ such that 
\begin{align}\label{t_310}
    \begin{split}
        \fiint_Q|\vh-(\vh)_Q|\,dz 
        &\leq csr^{-1}\fiint_{[Q]_h} (f^{p-1}+a(z)f^{q-1})\,dz\\
        &\qquad+cs\fiint_{[Q]_h}f\,dz+ cr\fiint_{[Q]_h}f\,dz.
    \end{split}    
\end{align}
\item[(ii)]If in addition $ \ell_s\cap \ell_{S_2}(t_0)^c \neq \emptyset$, then there exists a constant $c = c(\data)$ such that 
\begin{align}\label{t_311}
    \begin{split} 
    \fiint_Q|\vh|\,dz &\leq csr^{-1}\fiint_{[Q]_h} (f^{p-1}+a(z)f^{q-1})\,dz\\
    &\qquad+cs\fiint_{[Q]_h}f\,dz+ cr\fiint_{[Q]_h}f\,dz.
    \end{split}    
\end{align}
\end{itemize}
Moreover, the above estimates hold with $v_h$ and $[Q]_h$ replaced by $v$ and $Q$.
\end{lemma}

\begin{proof}
    For $t_1, t_2 \in \ell_s$, $t_1\leq t_2$, let $\zeta_\delta\in W^{1,\infty}_0(\ell_s)$ be a piecewise linear cutoff function
\begin{align*}
	\zeta_{\delta}(t)=
	\begin{cases}
	    0, &  t \in (-\infty, t_1-\delta),\\
		1+\frac{t-t_1}{\delta},& t\in [t_1-\delta, t_1]\\
		1,&t\in (t_1,t_2),\\
		1-\frac{t-t_2}{\delta},&t \in [t_2,t_2+\delta],\\
		0,& t \in (t_2+\delta, \infty).
	\end{cases}
	\end{align*}
Moreover, let $\vp\in C_0^\infty(B_r)$ be a nonnegative function satisfying
    \[
        \fint_{B_r} \vp \,dx = 1, \quad \lVert\na\vp\rVert_\infty \le \frac{c(n)}{r}, \quad \lVert\vp\rVert_\infty \le c(n), 
    \]
    and let $\psi = \vp\eta\zeta \zeta_\delta\in W_0^{1,\infty}(Q\cap Q_{R_2,S_2-h}(z_0))$, where $\eta$ and $\zeta$ are defined in \eqref{42_1}. Taking $\psi$ as a test function in the Steklov averaged weak formulation of \eqref{11}, we obtain
    \[
        	\iint_Q -[u-u_0]_h \cdot\vp\eta\pa_t(\zeta \zeta_\delta) \,dz = \iint_Q \left[-A(\cdot,\na u)+|F|^{p-2}F+a|F|^{q-2}F\right]_h\cdot\na\psi\,dz.
    \]
    After using the product rule of time derivative and rearranging terms, we get 
    \begin{align*}
        \begin{split}
            \mathrm{I} &= \left| \iint_Q-[u-u_0]_h\cdot\eta \vp \zeta \pa_t(\zeta_\delta)\,dz\right|
            \leq \left|\iint_Q[u-u_0]_h\cdot\eta \vp \zeta_\delta\pa_t\zeta \,dz\right| \\
            &\qquad+(L+1)\left|\iint_Q\left[|\na u|^{p-1}+a|\na u|^{q-1}+|F|^{p-1}+a|F|^{q-1}\right]_h \cdot\na\psi \,dz\right|\\
            &\leq \left|\iint_Q[f]_h\eta \vp \zeta_\delta\pa_t\zeta \,dz\right| +c\left|\iint_Q\left[f^{p-1}+af^{q-1}\right]_h |\na\psi| \,dz\right| = \mathrm{II}+\mathrm{III}.
        \end{split}
    \end{align*}
    Lemma~\ref{lem21} implies that
    \[
            \mathrm{II} \leq \iint_Q [f]_h|\pa_t\zeta|\lVert\vp\rVert_\infty  \,dz\leq c sr^n\fiint_{[Q]_h}f\,dz.
    \]
    Again by Lemma~\ref{lem21} and properties of the test functions, we have
    \begin{align*}
        \begin{split}
            \mathrm{III}&\leq c\iint_Q [f^{p-1}+af^{q-1}]_h |\vp\na\eta+\eta\na\vp| \,dz \\
            & \leq c \iint_Q [f^{p-1}+af^{q-1}]_h(\vp+\eta)(|\na\vp|+|\na\eta|) \,dz\\
            &\leq csr^n\left(\frac{1}{R_2-R_1}+\frac{1}{r}\right) \fiint_{[Q]_h} (f^{p-1}+a(z)f^{q-1})\,dz.
        \end{split}
    \end{align*}
    Since $B_r\subset B_{R_2}(x_0)$, we have $r\leq R_2$ and therefore
	\[
		\frac{1}{R_2-R_1}+\frac{1}{r}=\frac{R_2}{R_2-R_1}\frac{1}{R_2}+\frac{1}{r}\le 2\max\left\{\frac{R_2}{R_2-R_1},1\right\}\frac{1}{r}.
	\]
    We conclude that
\[
    \mathrm{III}\leq csr^{n-1} \fiint_{[Q]_h} (f^{p-1}+a(z)f^{q-1})\,dz.
\]

By the one-dimensional Lebesgue differentiation theorem we have  
\begin{align*}
\begin{split}
    \lim_{\delta \to 0^+} \mathrm{I}&=\left|\lim_{\delta \to 0^+}\iint_Q[u-u_0]_h\eta \vp \zeta \pa_t(\zeta_\delta)\,dz \right|\\
    &=\left| \int_{B_r\times \{t_1\}}\vh\vp\,dz - \int_{B_r\times \{t_2\}}\vh\vp\,dz \right|\\
    &=r^n\left|(\vh\vp)_{B_r}(t_1)-(\vh\vp)_{B_r}(t_2)\right|.
    \end{split}
\end{align*}
As the estimates of $\mathrm{II}$ and $\mathrm{III}$ are independent of $\delta$, we conclude that 
\begin{equation} \label{t_320_1}
\begin{split}
    &\esssup_{t_1,t_2 \in \ell_s}|(\vh\vp)_{B_r}(t_1)-(\vh\vp)_{B_r}(t_2)| \leq c s\fiint_{[Q]_h}f\,dz\\
    &\qquad+ csr^{-1} \fiint_{[Q]_h} (f^{p-1}+a(z)f^{q-1})\,dz .
    \end{split}
\end{equation}

To complete the proof we estimate the left-hand sides of \eqref{t_310} and \eqref{t_311} with the supremum above. We start with \eqref{t_310}. Using the standard Poincar\'e inequality, we get
	\begin{align} \label{t_321}
	\begin{split}
		\fiint_{Q}|\vh-(\vh)_{Q}|\,dz
		&\le \fiint_{Q}|\vh-(\vh)_{B_r}|\,dz+ \fiint_{Q}|(\vh)_{B_r}-(\vh)_{Q}|\,dz\\
		&\le c r\fiint_{Q}|\na \vh|\,dz+\fiint_{Q}|(\vh)_{B_r}-(\vh)_{Q}|\,dz.
		\end{split}
	\end{align}
 We estimate the first term with Lemma~\ref{lem21} to obtain
	\begin{align}\label{426_5}
			\fiint_{Q}|\na \vh|\,dz &= \fiint_{Q}\left|[\na u ]_h\eta\zeta+[u-\uu]_h\na \eta \zeta\right| \,dz\le c\fiint_{[Q]_h}f\,dz.
	\end{align}
For the second term on the right-hand side of \eqref{t_321}, we have
\begin{align*}
    \begin{split}
        &\fiint_Q|(\vh)_{B_r}(t)-(\vh)_Q|\,dz= \fint_{\ell_s}\left|\fint_{\ell_s}(\vh)_{B_r}(t)-(\vh)_{B_r}(\tau)\,dt\right|\,d\tau\\
        &\qquad \leq \fint_{\ell_s}2|(\vh)_{B_r}(t)-(\vh\vp)_{B_r}(t)|\,dt + \esssup_{t,\tau\in \ell_s}|(\vh\vp)_{B_r}(t)-(\vh\vp)_{B_r}(\tau)|.
    \end{split}
\end{align*}
The second term above is the same as in \eqref{t_320_1}. It remains to estimate the first term. We use the fact that $(\vp)_{B_r}=1$, the standard Poincar\'e inequality and \eqref{426_5} to get
\begin{equation} \label{t_325}
   \begin{split}
    \fint_{\ell_s}|(\vh)_{B_r}(t)-(\vh\vp)_{B_r}(t)|\,dt
       &=\fint_{\ell_s}\left|\fint_{B_r\times \{t\}}\vh\vp\,dx - (\vh)_{B_r}(t)\fint_{B_r}\vp\,dx \right|\,dt\\
       &=\fint_{\ell_s}\left|\fint_{B_r\times \{t\}}\vp(\vh-(\vh)_{B_r})\,dx \right|\,dt\\
        &\leq \lVert\vp\rVert_\infty \fint_{\ell_s}\fint_{B_r}|\vh-(\vh)_{B_r}|\,dx\,dt\\
        &\leq cr \fiint_Q|\na\vh|\,dz\le c\fiint_{[Q]_h}f\,dz.
    \end{split}
\end{equation}
This completes the proof of \eqref{t_310}.

To prove \eqref{t_311} we argue in a similar way. As in \eqref{t_321}, we have
    	\begin{align*}
		\fiint_{Q}|\vh|\,dz
		&\le \fiint_{Q}|\vh-(\vh)_{B_r}|\,dz+ \fint_{\ell_s}|(\vh)_{B_r}|\,dz\\
  &\le c r\fiint_{Q}|\na \vh|\,dz+\fint_{\ell_s}|(\vh)_{B_r}|\,dz,
	\end{align*}
where the first term was already estimated in \eqref{426_5}. For the second term we have
\[
        \fiint_{Q}|(\vh)_{B_r}|\,dz  \leq \fint_{\ell_s} |(\vh)_{B_r}-(\vh\vp)_{B_r}| \,dt +\fint_{\ell_s} |(\vh\vp)_{B_r}| \,dt.
\]
The first term above was estimated in \eqref{t_325}. For the second term we can choose $t_2 \in \ell_{S_2}(t_0)^c$ in \eqref{t_320_1} because of the assumption  $ \ell_s\cap \ell_{S_2}(t_0)^c \neq \emptyset$. Since $\zeta$ vanishes outside $\ell_{S_2}(t_0)$, we have
\[
        \fint_{\ell_s} |(\vh\vp)_{B_r}| \,dt \leq \esssup_{t\in \ell_s}  |(\vh\vp)_{B_r} (t)| \leq  \esssup_{t_1,t_2 \in \ell_s}|(\vh\vp)_{B_r}(t_1)-(\vh\vp)_{B_r}(t_2)|,
\]
which can be estimated by \eqref{t_320_1}. This completes the proof of \eqref{t_311}. Proofs for the corresponding statements without the Steklov average are analogous and we omit the details.
\end{proof}

To prove Lemma~\ref{lem_intrinsic_Poincare} near the lateral boundary, we use the following boundary version of the Poincar\'e inequality. For a proof we refer to \cite[Example 6.18 and Theorem 6.22]{MR4306765}.
\begin{lemma}\label{lem22}
   Let $B_{\rho}(x_0)\subset\RR^n$ and $B_{r}\subset\RR^n$ be balls satisfying $B_r\cap B_{\rho}(x_0)^c\ne\emptyset$. Assume that $v\in W_0^{1,s}(B_{\rho}(x_0))$ with $1<s<\infty$.  Moreover, let $1\le \sig\le \tfrac{ns}{n-s}$ for $1<s<n$ and $1\le \sig<\infty$ for $n\le s<\infty$. Then there exists a constant $c=c(n,s,\sig)$ such that
	\[
		\left(\fint_{B_{4r}}|v|^\sig\,dx\right)^\frac{1}{\sig}\le cr \left(\fint_{B_{4r}}|\na v|^s\,dx\right)^\frac{1}{s}.
	\]
\end{lemma}

Now we are ready to show the Poincar\'e type inequality.
\begin{lemma} \label{lem_intrinsic_Poincare}
The following holds for $K_iU_i$ defined in \eqref{def_Ui_37} and \eqref{def_Ki}:
\begin{itemize}
\item[(i)] If $K_iU_i\subset Q_{R_2,S_2}(z_0)$, then
\begin{align}\label{t_369'1}
    \fiint_{\uik}|\vh-(\vh)_\uik|\,dz \leq c(\data,\La)r_i 
\end{align}
and
\begin{align}\label{t_369'2}
    \fiint_{\uik}|v-(v)_\uik|\,dz \leq c(\data)r_i\lai.
\end{align}
\item[(ii)] If $K_iU_i\not\subset Q_{R_2,S_2}(z_0)$, then 
\begin{align}\label{t_369''1}
    \fiint_{\uik}|\vh|\,dz \leq c(\data,\La)r_i
\end{align}
and
\begin{align}\label{t_369''2}
    \fiint_{\uik}|v|\,dz \leq c(\data)r_i\lai.
\end{align}
\end{itemize}
\end{lemma}

\begin{proof}
We start by proving \eqref{t_369'2}. By the assumption $K_iU_i\subset Q_{R_2,S_2}(z_0)$, we may apply Lemma~\ref{lem32} with $Q=K_iU_i$. Assume first $U_i=Q_i$. Then Lemma~\ref{lem32} gives
\begin{align}\label{poinc1}
    \fiint_{K_iU_i}|v-(v)_{K_iU_i}|\,dz\le 
    c\la_i^{2-p}r_i\fiint_{K_iU_i}(f^{p-1}+a(z)f^{q-1})\,dz+cr_i\fiint_{K_iU_i}f\,dz,
\end{align}
where we already estimated the third term in \eqref{t_310} with the fourth one as $r_i\le R_2$ and therefore
\begin{align}\label{tt_369}
    r_i^2\la_i^{2-p}\fiint_{K_iU_i}f\,dz\le R_2r_i\fiint_{K_iU_i}f\,dz.
\end{align}
To estimate the term containing $a(z)$, observe that \eqref{15} implies
\begin{align}\label{poinc2}
 \begin{split}
      \fiint_{K_i  U_i} a(z)f^{q-1}\,dz 
      \leq a(z_i)\fiint_{  K_iU_i} f^{q-1}\,dz+[a]_\alpha (K_ir_i)^\al \fiint_{K_i U_i} f^{q-1}\,dz.     
      \end{split}
 \end{align}
As $U_i=Q_i$, we have $a(z_i)\leq K^2\lai^{p-q}$ and Lemma~\ref{lem_subintrinsic} gives
\[
    a(z_i)\fiint_{ K_i U_i} f^{q-1}\,dz \leq c\lai^{p-q}\fiint_{K_iU_i} f^{q-1}\,dz\leq c\lai^{p-1}.  
\]
To estimate the other term in \eqref{poinc2}, we apply Lemma~\ref{lem_subintrinsic} and Proposition~\ref{prop_whitney}\,\ref{UUi} to get
\[
    r_i^\alpha\fiint_{K_iU_i}f^{q-1}\,dz\le r_i^\alpha\left(\fiint_{K_iU_i}f^{d}\,dz\right)^\frac{q-1}{d}\le r_i^\alpha\la_i^{q-1}\le c\la_{i}^{p-1}.
\]
Inserting these estimates and Lemma~\ref{lem_subintrinsic} to \eqref{poinc1}, we have the desired estimate
\[
     \fiint_{K_iU_i}|v-(v)_{K_iU_i}|\,dz\le cr_i\la_i.
\]

In case $U_i= G_i$, Lemma~\ref{lem32} gives
\[
    \fiint_{K_iU_i}|v-(v)_{K_iU_i}|\,dz \leq c\frac{\lai^2r_i}{\La}\fiint_{K_iU_i} (f^{p-1}+a(z)f^{q-1})\,dz+cr_i\fiint_{K_iU_i} f\,dz,
\]
where we again used \eqref{tt_369}. The second term is again estimated by Lemma~\ref{lem_subintrinsic}.
To estimate the first term, we apply Proposition~\ref{prop_whitney}\,\ref{viii} to cancel $a(z)$, which along with Lemma~\ref{lem_subintrinsic} gives
\begin{align*}
\begin{split}
    &\frac{\lai^2r_i}{\lai^p+a(z_i)\lai^q}\fiint_{K_iG_i} (f^{p-1}+a(z)f^{q-1})\,dz \\
    &\qquad\leq \frac{\lai^2r_i}{\lai^p}\fiint_{K_iG_i} f^{p-1}\,dz + \frac{\lai^2r_i}{a(z_i)\lai^q}\fiint_{K_iG_i} a(z)f^{q-1}\,dz\\
    &\qquad\leq cr_i\lai + \frac{2\lai^2r_i}{\lai^q}\fiint_{K_iG_i} f^{q-1}\,dz \leq cr_i\lai.
    \end{split}
    \end{align*}
    This finishes the proof of \eqref{t_369'2}. 

To prove \eqref{t_369''2}, we assume $K_iU_i\not\subset Q_{R_2,S_2}(z_0)$. If in addition $K_iB_i \subset B_{R_2}(x_0)$, then  $K_iI_i\cap \ell_{S_2}(t_0)^c\ne \emptyset$ and we may apply \eqref{t_311}. Since the right-hand sides of \eqref{t_310} and \eqref{t_311} are the same, \eqref{t_369''2} follows from the proof of \eqref{t_369'2}. On the other hand, if $K_iB_i \cap B_{R_2}^c \neq \emptyset$, then we may apply Lemma~\ref{lem22} with $\sigma= 1 $ and $s=d$, since $v(\cdot,t)\in W^{1,d}_0(B_{R_2}(x_0),\RR^N)$. It follows that
\[
    \fiint_{K_i U_i} |v|\,dz \leq cr_i\left(\fiint_{4K_i U_i}|\na v|^d\,dz\right)^\frac{1}{d},
\]
and as $|\na v| \leq cf$, we have by Lemma~\ref{lem_subintrinsic} that
\[
    \left(\fiint_{4K_i U_i}|\na v|^d\,dz\right)^\frac{1}{d}\leq \left(\fiint_{4K_i U_i}f^d\,dz\right)^\frac{1}{d}\leq c\lai.
\]
This gives the desired estimate and we have proven \eqref{t_369''2}.

Finally, we show the corresponding estimates with Steklov averages, that is, \eqref{t_369'1} and \eqref{t_369''1}. 
If $K_iU_i \subset Q_{R_2,S_2}(z_0)$, then Lemma~\ref{lem32} gives
	\begin{align*}
	    \begin{split}
	            &\fiint_{K_iU_i}|\vh-(\vh)_{K_iU_i}|\,dz \leq c|I_i|r_i^{-1}\fiint_{[K_iU_i]_h} (f^{p-1}+a(z)f^{q-1})\,dz\\
             &\qquad+c|I_i|\fiint_{  [K_iU_i]_h}f\,dz+ cr_i\fiint_{ [K_iU_i]_h}f\,dz.
	  	    \end{split}
	\end{align*} 
Note that $|I_i| \leq r_i^2 \leq r_i R_2$ and $a(z)\le\lVert a \rVert_\infty$.
Using Lemma~\ref{lem_subintrinsic} with $s=d$ gives
\[
        \fiint_{K_iU_i}|\vh-(\vh)_{K_iU_i}|\,dz \leq cr_i \max\biggl\{1, \fiint_{  [K_iU_i]_h}f^{d}\,dz \biggr\}^\frac{q-1}{d}\leq cr_i\La^{\frac{q-1}{p}},
\]
which finishes the proof of \eqref{t_369'1}.

We get \eqref{t_369''1} in the case $K_iB_i\subset B_{R_2}(x_0)$ and $K_iI_i\not\subset \ell_{S_2}(t_0)$, from \eqref{t_311} and the proof of \eqref{t_369'1}. 
On the other hand, if $K_iB_i \cap B_{R_2}(x_0)^c \neq \emptyset$, we apply Lemma~\ref{lem22} to get 
\[
    \fiint_{K_i U_i} |\vh|\,dz \leq cr_i\left(\fiint_{4K_i U_i}|\na v_h|^d\,dz\right)^\frac{1}{d}.
\]
By Lemma~\ref{lem21} and Lemma~\ref{lem_subintrinsic} we have 
\begin{align*}
\begin{split}
    \left(\fiint_{4K_i U_i}|\na v_h|^d\,dz\right)^\frac{1}{d} 
    &\leq c\left(\fiint_{4K_i U_i}\left[|\na v|^d\right]_h\,dz\right)^\frac{1}{d}
    \leq c\left(\fiint_{[4K_i U_i]_h}|\na v|^d\,dz\right)^\frac{1}{d}\\
    &\leq c\left(\fiint_{[4K_i U_i]_h}f^d\,dz\right)^\frac{1}{d} \leq c\La^\frac{1}{p},
\end{split}
\end{align*}
and combining the estimates gives \eqref{t_369''2}. This completes the proof.
 \end{proof}	

\begin{remark}
    It follows from Lemma~\ref{lem_intrinsic_Poincare} that 
\begin{align}\label{rmk_46}
    \fiint_{2U_i}|v-v^i|\,dz \leq c(\data)r_i\lai.
\end{align}
Indeed, if $K_iU_i\subset Q_{R_2,S_2}(z_0)$, we have by \eqref{t_369'2} that
\begin{align*}
\begin{split}
    \fiint_{2U_i}|v-v^i|\,dz &\le \fiint_{2U_i}|v-(v)_{K_iU_i}|\,dz+\fiint_{2U_i}|(v)_{K_iU_i}-(v)_{2U_i}|\,dz\\
    &\leq c\fiint_{K_iU_i}|v-(v)_{K_iU_i}|\,dz \leq cr_i\lai.
\end{split}
\end{align*}
On the other hand, if $K_iU_i\not\subset Q_{R_2,S_2}(z_0)$, we apply \eqref{t_369''2} to have
\[
    \fiint_{2U_i}|v-v^i|\,dz\le 2\fiint_{2U_i}|v|\,dz\le c\fiint_{K_iU_i}|v|\,dz\le cr_i\la_i.
\]
\end{remark}

 \subsection{Properties of the Lipschitz truncation}\label{subsec_liptr_props}
 In the following lemmas we provide pointwise estimates for $\vhla$, $\vla$ and their gradients. They allow us to show the Lipschitz regularity of $\vhla$ in Lemma~\ref{lem_lip} and the estimates in Proposition~\ref{prop_liptr1}. 	
\begin{lemma} \label{lem_47} We have 
$|\vh^\Lambda(z)|\leq c(\data,\La)$ and $|v^\Lambda(z)|\leq c(\data)\lambda_i$
for every $z \in U_i$.
\end{lemma}

\begin{proof}
By the definition of $v^j$ and Lemma~\ref{lem_subintrinsic}, we have for every $j\in \mathcal{I}_i$ that
\begin{align}\label{lem37_eq2}
    |v^j|\leq\fiint_{2U_j} v\,dz \leq c\fiint_{K_iU_i} f\,dz \leq c\lai.
\end{align}
Since $
    \vla(z)=\sum_{j\in \mathcal{I}_i}v^j\om_j(z)
    $ by \eqref{t_325_1}, we get from the above estimate and Proposition~\ref{prop_whitney}\,\ref{ix} that
\[
    |\vla(z)| \leq \biggl|\sum_{j\in \mathcal{I}_i}v^j\om_j(z) \biggr| \leq  \sum_{j\in \mathcal{I}_i} |v^j||\om_j(z)|\leq c|v^j|\leq c\lai.
\]
The claim with the Steklov averages is proved in the same way, the only difference is that we apply Lemma~\ref{lem21} and \eqref{claim_subintrinsic1} in \eqref{lem37_eq2}.
\end{proof}

The next lemma is an auxiliary tool to show Lemma~\ref{lem_410}.
\begin{lemma} \label{lem_48} 
We have $|\vh^i-\vh^j|\leq c(\data, \La) r_i$ and $|v^i-v^j|\leq c(\data) r_i\lai$ for every $i\in \mathbb{N}$ and $j\in \mathcal{I}_i$.
\end{lemma}

\begin{proof}
We only prove the claim without the Steklov averages since the corresponding claim with the Steklov averages is similar. First consider the case $K_iU_i,K_jU_j\subset Q_{R_2,S_2}(z_0)$. Proposition~\ref{prop_whitney}\,\ref{Uij} and  Proposition~\ref{prop_whitney}\,\ref{vii} imply that 
\begin{align*}
    \begin{split}
        |v^i-v^j| &\leq |v^i-(v)_\uik|+|(v)_\uik-v^j|\\
         &= \left|\fiint_{2U_i} v-(v)_{K_iU_i} \,dz\right|+ \left|\fiint_{2U_j} v-(v)_{K_{i}U_i} \,dz\right| \\
        &\leq c\fiint_\uik |v-(v)_{K_iU_i}| \,dz \leq cr_i\la,
    \end{split}
\end{align*}
where the last inequality follows from Lemma~\ref{lem_intrinsic_Poincare}. 

On the other hand, if $K_iU_i\not\subset Q_{R_2,S_2}(z_0)$, we apply Proposition~\ref{prop_whitney}\,\ref{Uij} and  Proposition~\ref{prop_whitney}\,\ref{vii} and \eqref{t_369''2} to have
\[
    |v^i-v^j| \leq \fiint_{2U_i}|v|\,dz+\fiint_{2U_j}|v|\,dz \leq c\fiint_\uik |v| \,dz\le cr_i\la_i.
\]
Since $r_i,r_j$ and $\la_i,\la_j$ are comparable by Proposition~\ref{prop_whitney}\,\ref{v} and  Proposition~\ref{prop_whitney}\,\ref{vi}, the case $K_jU_j\not\subset Q_{R_2,S_2}(z_0)$ is proven similarly and we are done.
\end{proof}
\begin{lemma}\label{lem_410}
For any $z\in U_i$, we have
	\begin{align}  \label{t_382}
		|\na \vhla(z)|\le c(\data,\La)
  \quad\text{and}\quad 
  |\na \vla(z)|\le c(\data)\la_i.
	\end{align}
	Moreover, we have
	\begin{equation}  \label{t_383}
			|\pa_t\vhla(z)|\le c(\data, \Lambda)r_i^{-1},
	\end{equation}
and 
\begin{equation} \label{t_384}
    |\pa_t\vla(z)|\le 
    \begin{cases}
       c(\data)r_i^{-1}\la_i^{p-1},&\text{if}\ U_i=Q_i,\\
       c(\data)r_i^{-1}\lai^{-1}\Lambda,&\text{if}\ U_i=G_i.
    \end{cases}
\end{equation}
\end{lemma}

\begin{proof}
We start with the proof of the second inequality in \eqref{t_382}. 
By \eqref{t_325_1} we have 
\[
    \na\vla(z)=\na\biggl(\sum_{j\in \mathcal{I}_i}v^j\om_j(z)\biggr) = \sum_{j\in \mathcal{I}_i}v^j\na\om_j(z).
\]
 Since 
\[
0 = \na\biggl(\sum_{j\in \mathbb{N}}\om_j(z)\biggr)= \na\biggl(\sum_{j\in \mathcal{I}_i}\om_j(z)\biggr)=\sum_{j\in \mathcal{I}_i}\na\om_j
\]
 by Proposition~\ref{prop_whitney}\,\ref{x}, we get
	\[
		\na\vla(z)=\sum_{j\in \mathcal{I}_i}v^j\na\om_j(z)=\sum_{j\in \mathcal{I}_i}v^j\na\om_j(z)-v^i\sum_{j\in \mathcal{I}_i}\na\om_j(z)=\sum_{j\in \mathcal{I}_i}(v^j-v^i)\na\om_j(z).
	\]
	Therefore Lemma~\ref{lem_48} and Proposition~\ref{prop_whitney}\,\ref{xi},\,\ref{v},\,\ref{ix} give
	\[
		|\na\vla(z)|\le \sum_{j\in  \mathcal{I}_i}|v^j-v^i||\na\om_j| \leq c\sum_{j\in  \mathcal{I}_i}r_j^{-1}r_i\lai \leq c\lai.
	\]
Exactly the same argument with $\vhla$ instead of $\vla$ proves the first inequality in \eqref{t_382}. 

	To show \eqref{t_383}, we have by the same argument as above that
\[
    \pa_t\vhla(z)=\sum_{j\in \mathcal{I}_i}\vh^j\pa_t\om_j(z)=\sum_{j\in \mathcal{I}_i}(\vh^j-\vh^i)\pa_t\om_j(z).
\]
Again, Lemma~\ref{lem_48} and Proposition~\ref{prop_whitney}\,\ref{xii},\,\ref{v},\,\ref{ix} give
	\[
	    |\pa_t\vhla(z)| \le \sum_{j\in  \mathcal{I}_i}|\vh^j-\vh^i||\pa_t\om_j| \leq c(\data,\La)\sum_{j\in  \mathcal{I}_i}r_j|\pa_t\om_j| \leq c(\data,\La)r_i^{-1}.
	\]
The same argument with $\vla$ instead of $\vhla$ shows \eqref{t_384} and the proof is completed.
\end{proof}

We are ready to show that $\vhla$ is Lipschitz continuous with respect to the metric 
\[
    d_{\la^p}(z,w)=\max\left\{|x-y|,\sqrt{\la^{p-2}|t-s|}\right\},
\]
where $z,w \in Q_{R_2,S_2}(z_0)$ with $z=(x,t)$ and $w=(y,s)$, and $\la$ is chosen to satisfy $\la^p+\lVert a \rVert_\infty\la^q = \Lambda$.
\begin{lemma} \label{lem_lip}
There exists a constant $c_\La=c(\data,\La)$ such that 
\[
    |\vhla(z)-\vhla(w)|\leq c_\La d_{\la^p}(z,w)
\]
for every $z,w\in\RR^{n+1}$.
\end{lemma}

\begin{proof}
Observe that a cylinder $\Qlw$, as defined in \eqref{def_Q_cylinder}, is a ball of radius $l$ associated with the metric $d_{\la^p}$. Since $\la_i\ge\la$ and $p\geq 2$, it is clear that 
\begin{align} \label{495}
    	d_{\la^{p}}(\cdot,\cdot)\le d_{i}(\cdot,\cdot),
\end{align}
for all $i\in \mathbb{N}$. Therefore, $U_i \subset Q^\la_{r_i}(z_i)$ for any $i \in \mathbb{N}$. By the Campanato characterization of Lipschitz continuity, it is enough to show that there exists a constant $c_\La=c(\data, \Lambda)$, such that 
	\begin{align}\label{479}
		\fiint_{Q_{l}^\la(w)}\frac{|\vhla-(\vhla)_{Q_{l}^\la(w)}|}{l}\,dz\le c_\La
	\end{align}
for all $\Qlw\subset \mathbb{R}^{n+1}$. We fix $\Qlw$ and prove the above estimate, first in the case when $2\Qlw$ is inside the bad set, then when it is touching  the good set.

\textit{Case 1: $2Q_{l}^\la(w)\subset E(\La)^c$.} Let $z\in \Qlw$ and let $i\in \mathbb{N}$ be such that $z\in U_i$. Note that such $i$ exists by  Proposition~\ref{prop_whitney}\ref{i}. Triangle inequality and \eqref{495} give
\[
    l \leq d_{\la^p}(z, E(\Lambda)) \leq d_{\la^p}(z,z_i)+d_{\la^p}(z_i,E(\La))\leq d_{i}(z,z_i)+d_{i}(z_i,E(\La)),
\]
where $z_i$ is the centre of $U_i$. Therefore, Proposition~\ref{prop_whitney}\,\ref{iv} and the fact that $U_i$ is a ball of radius $r_i$ with respect to the metric $d_{i}(\cdot,\cdot)$ implies that
\[
    l\leq d_{i}(z,z_i)+d_{i}(z_i,E(\La)) \leq r_i+ 5r_i=6r_i.
\]
Combining the above inequality with Lemma~\ref{lem_410}, we obtain the uniform estimate
\begin{equation}\label{496}
  	|\pa_t\vhla(z)|\le c(\data,\La)r_i^{-1} \leq c(\data,\La)l^{-1}
\end{equation}
for all $z\in \Qlw$.

To show \eqref{479}, let $z_1, z_2 \in \Qlw$ with $z_1=(x_1,t_1)$ and $z_2=(x_2,t_2)$. Since $\vhla \in C^\infty(E(\La)^c,\RR^N)$ and $\Qlw \subset E(\La)^c$, by the mean value theorem
we obtain
\begin{align*}
    \begin{split}
        |\vhla(z_1)-\vhla(z_2)|&\leq |\vhla(x_1,t_1)-\vhla(x_2,t_1)|+|\vhla(x_2,t_1)-\vhla(x_2,t_2)|\\
			&\le cl\sup_{z\in \Qlw}|\na \vhla(z)|+c\la^{2-p}l^2\sup_{z\in Q_{l}^\la(w)}|\pa_t\vhla(z)|\leq c(\data,\La)l,
    \end{split}
\end{align*}
where the last inequality follows from \eqref{t_382} and \eqref{496}.
The estimate above gives
\[
		\fiint_{Q_{l}^\la(w)}\frac{|\vhla-(\vhla)_{Q_{l}^\la(w)}|}{l}\,dz 
		\leq \fiint_{Q_{l}^\la(w)} \fiint_{Q_{l}^\la(w)}\frac{\left|\vhla(z_1)-\vhla(z_2)\right|}{l} \,dz_2 \,dz_1\leq c(\data, \La).
\]
This proves \eqref{479} in the case $2Q_{l}^\la(w)\subset E(\La)^c$.

\textit{Case 2: $2Q_{l}^\la(w) \cap E(\La) \neq \emptyset$ and $Q_{l}^\la(w) \cap E(\La)^c = \emptyset$.} In this case $\vhla = \vh$ and therefore
\begin{align}\label{3_131}
        \fiint_{Q_{l}^\la(w)}\frac{|\vhla-(\vhla)_{Q_{l}^\la(w)}|}{l}\,dz=\fiint_{Q_{l}^\la(w)}\frac{|\vh-(\vh)_{Q_{l}^\la(w)}|}{l}\,dz.
\end{align}
We denote $w=(y,s)$. If $B_l(y)\subset B_{R_2}(x_0)$, then $l\le R_2$ holds and we apply Lemma~\ref{lem32} and Lemma~\ref{lem_subintrinsic} to get
\[
    \fiint_{Q_{l}^\la(w)}\frac{|\vh-(\vh)_{Q_{l}^\la(w)}|}{l}\,dz\le c(\data,\La)(\la^{2-p}+\la^{2-p}l+1)\le c(\data,\La).
\]
On the other hand, if $B_l(y)\not\subset B_{R_2}(x_0)$, we use Lemma~\ref{lem22} with $\sig=1$ and $s=d$ to get
\[
    \fiint_{Q_{l}^\la(w)}\frac{|\vh-(\vh)_{Q_{l}^\la(w)}|}{l}\,dz\le 2\fiint_{Q_{l}^\la(w)}\frac{|\vh|}{l}\,dz\le c\fiint_{4Q_{l}^\la(w)}|\na \vh|\,dz.
\]
Recalling that in this case $2Q_{l}^\la(w) \cap E(\La) \neq \emptyset$,
we conclude with Lemma~\ref{lem_subintrinsic} that
\[
    \fiint_{4Q_{l}^\la(w)}|\na \vh|\,dz\le c\fiint_{[4Q_{l}^\la(w)]_h}f\,dz\le c(\data,\La),
\]
which finishes the proof of \eqref{479} in the case $2Q_{l}^\la(w) \cap E(\La) \neq \emptyset$ and $Q_{l}^\la(w) \cap E(\La)^c = \emptyset$.

\textit{Case 3: $2Q_{l}^\la(w) \cap E(\La) \neq \emptyset$ and $Q_{l}^\la(w) \cap E(\La)^c \neq \emptyset$.} To estimate in the region where $\Qlw$ touches $E(\La)^c$, we define the index set 
\[
P = \{i\in \mathbb{N} : \Qlw \cap 2U_i\neq \emptyset\}. 
\]
We claim that the radii $r_i$ are bounded uniformly by $l$ when $i\in P$.  Let $i\in P$, $w_1 \in \Qlw\cap 2U_i$ and $w_2\in 2\Qlw\cap E(\Lambda)$ with $w_1=(y_1,s_1)$ and $w_2=(y_2,s_2)$. By Proposition~\ref{prop_whitney}\,\ref{iii}  and $w_1 \in 2U_i$, we have  
\[
        3r_i\leq d_i(U_i, E(\Lambda))\leq d_i(z_i,w_2)\leq d_i(z_i,w_1)+d_i(w_1,w_2)\leq 2r_i+d_i(w_1,w_2),
\]
and therefore $r_i\leq d_i(w_1,w_2)$. Furthermore, since $\la\leq\lai$ we have
\begin{align*}
\begin{split}
    d_i(w_1,w_2)
    &\leq \max\left\{|y_1-y_2|,\sqrt{\La\lai^{-2}|s_1-s_2|}\right\}\\
    &\leq \Lambda^\frac{1}{2}\max\left\{|y_1-y_2|,\sqrt{\la^{p-2}|s_1-s_2|}\right\}\\
    &=\Lambda^\frac{1}{2}d_{\la^p}(w_1,w_2) \leq 4\Lambda^\frac{1}{2}l,
    \end{split}
\end{align*}
where the last inequality follows from the fact that $w_1,w_2 \in 2\Qlw$. We conclude that $r_i \leq c(\Lambda)l$.

To prove \eqref{479} we add and subtract twice to obtain 
\begin{align}
    \begin{split} \label{3_85}
        &\fiint_{\Qlw} \frac{|\vh^\Lambda-(\vh^\Lambda)_{\Qlw}|}{l}\,dz
        \leq 2\fiint_{\Qlw} \frac{|\vh^\Lambda-(\vh)_{\Qlw}|}{l}\,dz \\
        &\qquad\leq 2\fiint_{\Qlw} \frac{|\vh^\Lambda-\vh|}{l}\,dz+2\fiint_{\Qlw} \frac{|\vh-(\vh)_{\Qlw}|}{l}\,dz.
    \end{split}
\end{align}
The second term on the right-hand side is the same as in \eqref{3_131} and can be  estimated in the same way. To estimate the first term, observe from \eqref{t_421} and the fact that $w_i$ is supported in $2U_i$, that
\begin{align*}
    \begin{split}
        \fiint_{\Qlw} \frac{|\vh^\Lambda-\vh|}{l}\,dz 
        &=\fiint_{\Qlw} \frac{\left|\sum_{i\in P}(\vh-\vh^i)\om_i\right|}{l}\,dz\\
        &\leq \fiint_{\Qlw} \sum_{i\in P} \frac{\left|\vh-\vh^i\right|\om_i}{l}\,dz\\
        &= \frac{1}{|\Qlw|}\sum_{i\in P}\iint_{\Qlw \cap 2U_i} \frac{\left|\vh-\vh^i\right|\om_i}{l}\,dz.
    \end{split}
\end{align*}
Since $\om_i\leq 1$ and $r_i \leq c(\Lambda)l$, we have
\[
        \frac{1}{|\Qlw|}\sum_{i\in P}\iint_{\Qlw \cap 2U_i} \frac{\left|\vh-\vh^i\right|\om_i}{l}\,dz 
        \leq \frac{c(\Lambda)}{|\Qlw|}\sum_{i\in P}\iint_{2U_i} \frac{\left|\vh-\vh^i\right|}{r_i}\,dz,
\]
and combining the previous inequalities with \eqref{rmk_46} gives
\[
        \fiint_{\Qlw} \frac{|\vh^\Lambda-\vh|}{l}\,dz \leq \frac{c(\data,\Lambda)}{|\Qlw|}\sum_{i\in P}|U_i|.
\]
Since  $r_i \leq c(\Lambda)l$, we have $U_i \subset c(\data,\La)\Qlw$ for every $i\in P$. From this and Proposition~\ref{prop_whitney}\,\ref{ii}, we conclude that
\[
    \fiint_{\Qlw} \frac{|\vh^\Lambda-\vh|}{l}\,dz \leq \frac{c(\data,\Lambda)}{|\Qlw|}\sum_{i\in P} \left|\tfrac{1}{6K}U_i\right|\leq c(\data,\Lambda).
\]
Combining this estimate with \eqref{3_85}, we have \eqref{479} in the final case and the proof is completed.
\end{proof}

In the following proposition we collect the properties of Lipschitz truncation that will be used in the proof of Theorem~\ref{thm_cac}.
\begin{proposition} \label{prop_liptr1}
    Let $E(\La)$ be as in \eqref{def_E}. There exist functions $\{v_h^\La \}_{h>0}$ and a function $v^\La$ satisfying the following properties:
    \begin{enumerate}[label=(\roman*),series=theoremconditions]
    \item\label{p2-1} $v_h^\La\in W_0^{1,2}(\ell_{S_2-h_0}(t_0);L^2(B_{R_2}(x_0),\RR^N))\cap L^\infty(\ell_{S_2-h_0}(t_0);W^{1,\infty}_{0}(B_{R_2}(x_0),\RR^N))$.
    \item\label{p2-2} $v^\La\in L^\infty(\ell_{S_2}(t_0);W^{1,\infty}_{0}(B_{R_2}(x_0),\RR^N))$.
    \item\label{p2-3} $v_h^\La= v_h,$ $v^\La= v$, $\na v_h^\La= \na v_h,$ $\na v^\La=\na v$ a.e. in $E(\La)$.
    \item\label{p3-1} $v_h^\La \to v^\La$ in $L^\infty(Q_{R_2,S_2}(z_0),\RR^N)$ as $h\to0^+$, taking a subsequence if necessary.
    \item\label{p3-2} $\na v_h^\La\to \na v^\La$ and $\pa_t v_h^\La\to \pa_t v^\La$ a.e. in $E(\La)^c$ as $h\to0^+$.
\end{enumerate}
Moreover, there exist constants $c=c(\data)$ and $c_\La=c_\La(\data,\La)$ such that
\begin{enumerate}[resume*=theoremconditions]
    \item\label{p5} $\begin{aligned}[t]
    \iint_{E(\La)^c}|v-v^\La||\pa_tv^\La|\,dz\le c\La|E(\La)^c|,
    \end{aligned}$
    \item\label{p4} $\begin{aligned}[t]
    \iint_{E(\La)^c}|v_h-v_h^\La||\pa_tv_h^\La|\,dz\le c_\La|E(\La)^c|,
    \end{aligned}$
    \item\label{p6} $H(z,|v^\La(z)|)+H(z,|\na v^\La(z)|)\le c\La$ for a.e. $z\in \mathbb{R}^{n+1}$,
    \item\label{p7} $H(z,|v_h^\La(z)|)+H(z,|\na v_h^\La(z)|)\le c_\La$ for a.e. $z\in \RR^{n+1}.$
\end{enumerate}
\end{proposition}

\begin{proof}
    We stated \ref{p2-1} already at the beginning of this section, while \ref{p2-2} is deduced from Lemma~\ref{lem_lip}. From the construction $v_h^\La= v_h$ and $v^\La= v$ in $E(\La)$. Therefore, as $v_h^\La,v_h,v^\La,v$ belong to $L^p(\ell_{S_2} (t_0);W_0^{1,p}(B_{R_2}(z_0),\RR^N)$, we also have $\na v_h^\La= \na v_h$ and $\na  v^\La= \na v$ a.e. in $E(\La)$, which proves \ref{p2-3}.
    
    Next we prove \ref{p3-1}. Note that $\{v_h^\La\}_{h>0}$ is equicontinuous and uniformly bounded by Lemma~\ref{lem_lip}. Therefore, it follows from the Arzela--Ascoli theorem that, by passing to a subsequence if necessary, there exists $w$ such that $v_h^\La$ converges to $w$ in $L^\infty(Q_{R_2,S_2}(z_0),\RR^N)$. On the other hand, $v_h^\La$ converges to $v^\La$ a.e. as $h\to0^+$, which means $w=v^\La$ a.e.  
    This shows that $v_h^\La \to v^\La$ in $L^\infty(Q_{R_2,S_2}(z_0),\RR^N)$ as $h\to0^+$.
    
    If $z\in E(\La)^c$, then there exists $i\in\mathbb{N}$ such that $z\in U_i$. It follows from \eqref{t_325_1} that
    \[
        \na v_h^\La(z)=\sum_{j\in \mathcal{I}_i}v_h^j\na \om_j(z),\quad \pa_t v_h^\La(z)=\sum_{j\in \mathcal{I}_i}v_h^j\pa_t \om_j(z).
    \]
    By the properties of the Steklov average we have $v_h^j\to v^j$ as $h\to0^+$, which implies the a.e. convergence of $\na v_h^\La$ and $ \pa_t v^\La$ to $\na v^\La$ and $\pa_t v^\La$. This shows that
     $\na v_h^\La\to \na v^\La$ and $\pa_t v_h^\La\to \pa_t v^\La$ a.e. in $E(\La)^c$ as $h\to0^+$.
     This completes the proof of \ref{p3-2}.
    
     Then we discuss \ref{p5} and \ref{p4}. It follows from \eqref{t_421_2} that
\begin{align*}
\begin{split}
    \iint_{E(\La)^c}|v-\vla| |\pa_tv^\La|\,dz 
        &\leq \iint_{E(\La)^c}\sum_{i\in \mathbb{N}}|v-v^i||\om_i||\pa_tv^\La|\,dz \\
     &\leq \sum_{i\in \mathbb{N}} \iint_{2U_i} |v-v^i||\pa_tv^\La|\,dz\\
      & \leq \sum_{i\in \mathbb{N}} \|\pa_tv^\La\|_{L^\infty(2U_i)} \iint_{2U_i} |v-v^i|\,dz.
\end{split}
\end{align*}
By \eqref{rmk_46} and \eqref{t_384} we have
\begin{align*}
    \|\pa_tv^\La\|_{L^\infty(2U_i)} \iint_{2U_i} |v-v^i|\,dz&
    \le cr_i\la_i|U_i|\|\pa_tv^\La\|_{L^\infty(2U_i)}\\
    &\le c\La|U_i|=c\La|\tfrac{1}{6K}U_i|,
\end{align*}
and using the disjointedness of $\{\tfrac{1}{6K}U_i\}_{i\in\mathbb{N}}$ in Proposition~\ref{prop_whitney}\,\ref{ii} and $U_i\subset E(\La)^c$, \ref{p5} follows. Moreover, replacing \eqref{t_369'2}, \eqref{t_369''2} and \eqref{t_384} by \eqref{t_369'1}, \eqref{t_369''1} and \eqref{t_383} respectively, we conclude that \ref{p4} holds.

Then we consider \ref{p6}. If $z\in E(\La)$, it is easy to see from the definition of $E(\La)$, that
\begin{align*}
\begin{split}
    &|v^\La(z)|^p+a(z)|v^\La(z)|^q+|\na v^\La(z)|^p+a(z)|\na v^\La(z)|^q\\
    &\qquad=|v(z)|^p+a(z)|v(z)|^q+|\na v(z)|^p+a(z)|\na v(z)|^q\\
    &\qquad\le c(f^p(z)+a(z)f^q(z))\le c\La.
\end{split}
\end{align*}
On the other hand, if $z\in E(\La)^c$, then $z\in U_i$ for some $i\in \mathbb{N}$ and Lemma~\ref{lem_47} and Lemma~\ref{lem_410} give
\[
    |v^\La(z)|^p+a(z)|v^\La(z)|^q+|\na v^\La(z)|^p+a(z)|\na v^\La(z)|^q\le c(\la_i^p+a(z)\la_i^q).
\]
We estimate the right-hand side in two cases. If $U_i=Q_i$, we use \eqref{15}, Proposition~\ref{prop_whitney}\,\ref{UUi} and the fact that $\La=\la_i^p+a(z_i)\la_i^q$, to get
\[
\la_i^p+a(z)\la_i^q \leq \la_i^p+a(z_i)\la_i^q+[a]_\al r_i^\alpha\la_i^q\le c\La.
\]
If $U_i=G_i$, we have by  Proposition~\ref{prop_whitney}\,\ref{vii}, that
\[
    \la_i^p+a(z)\la_i^q\le c(\la_i^p+2a(z_i)\la_i^q)\le c\La.
\]
This completes the proof of \ref{p6}.

Finally, we prove \ref{p7}. From Lemma~\ref{lem_lip} follows that $|\na v_h^\La(z)|\le c_\La$ for a.e. $z\in \RR^{n+1}$, and therefore 
\[
    |\na v_h^\La(z)|^p+a(z)|\na v_h^\La(z)|^q\le (1+\|a\|_{\infty})(1+|\na v_h^\La(z)|)^q\le c_\La.
\]
Since $v^\La_h=0$ in $Q_{R_2,S_2}(z_0)^c$, Lemma~\ref{lem_lip} implies that $|v_h^\La(z)|\le c_\La$ for a.e. $z\in \RR^{n+1}$ and
\[
    |v_h^\La(z)|^p+a(z)| v_h^\La(z)|^q\le (1+\|a\|_{\infty})(1+| v_h^\La(z)|)^q\le c_\La.
\]
This finishes the proof.
\end{proof}

\section{Energy estimate}

In this section we use the Lipschitz truncation to prove the energy estimate in Theorem~\ref{thm_cac}. The main idea of the proof is to use the estimates in Proposition~\ref{prop_liptr1} and conclude the desired convergence from \eqref{convergence}. A similar argument will be used to prove Theorem~\ref{main2} in Section~\ref{subsec_uni}. 

\subsection{Proof of Theorem~\ref{thm_cac}}

For $\tau\in \ell_{S_2-h}(t_0)$ and sufficiently small $\delta>0$, let
\begin{align}\label{3_zetade}
	\zeta_{\delta}(t)=
	\begin{cases}
		1,&t\in (-\infty,\tau-\delta),\\
		1-\frac{t-\tau+\delta}{\delta},&t\in[\tau-\delta,\tau],\\
		0,&t\in(\tau,t_0+S_2-h).
	\end{cases}
\end{align}
Proposition~\ref{prop_liptr1}\,\ref{p2-1} implies that $\vhla\eta^q\zeta\zeta_{\delta}(\cdot,t) \in W^{1,\infty}_0(B_{R_2}(x_0),\RR^N)$ for every $t \in \ell_{S_2-h}(t_0)$. Therefore, $\vhla\eta^q\zeta\zeta_{\delta}(\cdot,t)$ can be used as a test function in the Steklov averaged weak formulation of \eqref{11}, and we obtain
\begin{align*}
	\begin{split}
		\mathrm{I}+\mathrm{II}
		&=\iint_{Q_{R_2,S_2}(z_0)}\pa_t[u-\uu]_h\cdot\vhla\eta^q\zeta\zeta_\de\,dz\\
		&\qquad+\iint_{Q_{R_2,S_2}(z_0)}[\mA(\cdot,\na u)]_h\cdot\na (\vhla\eta^q\zeta\zeta_{\delta})\,dz\\
		&=\iint_{Q_{R_2,S_2}(z_0)}[|F|^{p-2}F+a|F|^{q-2}F]_h\cdot \na(\vhla\eta^q\zeta\zeta_{\delta})\,dz=\mathrm{III}.
	\end{split}
\end{align*}
We show the desired convergence of each term by splitting the domain of integration into the good set $E(\La)$ and the bad set $E(\La)^c$. First letting $h\to 0^+$ and then $\La\to\infty$, the integral on the good set converges to the corresponding integral without a truncation, while the integral over the bad set converges to zero.

We start with term $\mathrm{I}$. Integration by parts and the product rule give
\begin{align*}
\begin{split}
    \mathrm{I}&=\iint_{Q_{R_2,S_2}(z_0)}(-\vh\cdot\vhla\eta^{q-1}\pa_t\zeta_\delta-\vh\cdot\pa_t\vhla \eta^{q-1}\zeta_{\delta})\,dz\\
    &\qquad-\iint_{Q_{R_2,S_2}(z_0)}[u-\uu]_h\cdot\vhla\eta^q\zeta_{\delta}\pa_t\zeta\,dz=\mathrm{I}_1+\mathrm{I}_2.
\end{split}
\end{align*}
First we consider $\mathrm{I}_1$ and write it as
\begin{align*}
	\begin{split}
		\mathrm{I}_1&=
  \iint_{Q_{R_2,S_2}(z_0)}-|\vh|^2\eta^{q-1}\pa_t\zeta_\delta \,dz + \iint_{Q_{R_2,S_2}(z_0)}\vh\cdot(v_h-v_h^\La)\eta^{q-1}\pa_t\zeta_\delta \,dz\\
  &\qquad-\iint_{Q_{R_2,S_2}(z_0)}(\vh-\vhla)\cdot\pa_t\vhla\eta^{q-1}\zeta_{\delta}\,dz-\iint_{Q_{R_2,S_2}(z_0)}\vhla\cdot\pa_t\vhla\eta^{q-1}\zeta_{\delta}\,dz.
\end{split}
\end{align*}
Integration by parts in the last term above gives
\begin{align*}
	\begin{split}
		\mathrm{I}_1
		&=\iint_{Q_{R_2,S_2}(z_0)}-\Bigl(|\vh|^2-\frac{1}{2}|\vhla|^2\Bigr)\eta^{q-1}\pa_t\zeta_{\delta}\,dz\\
		&\qquad+\iint_{Q_{R_2,S_2}(z_0)}\vh\cdot (\vh-\vhla)\eta^{q-1}\pa_t\zeta_{\delta}\,dz-\iint_{E(\La)^c}(\vh-\vhla)\cdot\pa_t\vhla\eta^q\zeta_{\delta}\,dz,
	\end{split}
\end{align*}
where we also observed that the domain in the third integral can be written as $E(\La)^c$ since $v_h= v_h^\La$ in $E(\La)$ and $v_h= v_h^\La= 0$ in $Q_{R_2,S_2}(z_0)^c$.

We claim that the above terms converge to the corresponding integrals without Steklov averages as $h\to 0^+$. Indeed, in the first integral the term $v_h$ converges to $v$ in $L^2(Q_{R_2,S_2}(z_0),\RR^N)$ and $v_h^\La$ converges to $v^\La$ in $L^\infty(Q_{R_2,S_2}(z_0),\RR^N)$ by the properties of Steklov averages and Proposition~\ref{prop_liptr1}~ \ref{p3-1}. The second integral converges similarly. The third integral is bounded uniformly in $L^1(E(\La)^c)$ for all $h>0$ by Proposition~\ref{prop_liptr1}\,\ref{p4} and since $v_h$, $v_h^\La$ and $\pa_tv_h^\La$ converge a.e. to $v$, $v^\La$ and $\pa_tv^\La$ respectively, we get the desired convergence. Therefore, we obtain
\begin{align*}
	\begin{split}
		\lim_{h \to 0}\mathrm{I}_1
		&=\iint_{Q_{R_2,S_2}(z_0)}-\Bigl(|v|^2-\frac{1}{2}|\vla|^2\Bigr)\eta^{q-1}\pa_t\zeta_{\delta}\,dz+\iint_{Q_{R_2,S_2}(z_0)}v\cdot (v-\vla)\eta^{q-1}\pa_t\zeta_{\delta}\,dz\\
         &\qquad-\iint_{E(\La)^c}(v-\vla)\cdot\pa_t\vla\eta^q\zeta_{\delta}\,dz=\mathrm{I}_{11}+\mathrm{I}_{12}+\mathrm{I}_{13}.
	\end{split}
\end{align*}
Again, we estimate each term separately. Since the Lipschitz truncation is done only in the bad set, we have 
\[
		\mathrm{I}_{11}=
		 \iint_{E(\La)}-\frac{1}{2}|v|^2\eta^{q-1}\pa_t\zeta_{\delta}\,dz-\iint_{E(\La)^c}\Bigl(|v|^2-\frac{1}{2}|\vla|^2\Bigr)\pa_t\zeta_{\delta}\,dz.
\]
In order to conclude that the second term on the right-hand side vanishes as $\La \to \infty$, we write
\begin{align} \label{t_130}
    \begin{split}
        &-\iint_{E(\La)^c}\Bigl(|v|^2-\frac{1}{2}|\vla|^2\Bigr)\pa_t\zeta_{\delta}\,dz \\
        &\qquad= -\iint_{E(\La)^c}|v|^2\pa_t\zeta_{\delta}\,dz
         +\frac{1}{2} \iint_{E(\La)^c}|\vla|^2\pa_t\zeta_{\delta} \,dz.   
        \end{split}
\end{align}
Since $v \in L^2(Q_{R_2,S_2}(z_0),\RR^N)$, by the absolute continuity of integral the first term on the right-hand side of \eqref{t_130} vanishes as $\La \to \infty$. For the second term on the right-hand side of \eqref{t_130}, by Proposition~\ref{prop_liptr1}\,\ref{p6} we obtain
\[
        \iint_{E(\La)^c}|\vla|^2|\pa_t\zeta_{\delta}| \,dz \leq \iint_{E(\La)^c}\La^\frac{2}{p}|\pa_t\zeta_{\delta}| \,dz\leq c(\delta)\La|E(\La)^c|.
\]
By \eqref{convergence} we conclude that
\[
	\lim_{\La\to\infty}\mathrm{I}_{11}=\iint_{Q_{R_{2},S_2}(z_0)}-\frac{1}{2}|v|^2\eta^{q-1}\pa_t\zeta_{\delta}\,dz.
\]

To estimate the second term, $\mathrm{I}_{12}$, note that by Young's inequality $v\cdot(v-\vla) \leq 2|v|^2+|\vla|^2$ and thus
\[
		|\mathrm{I}_{12}|
		\le 2\iint_{E(\La)^c}(|v|^2+|\vla|^2)|\pa_t\zeta_\delta|\,dz.
\]
As the integral above contains the same terms as \eqref{t_130}, we get $\lim_{\La\to\infty}\mathrm{I}_{12}=0$ as before.
Finally, it follows from Proposition~\ref{prop_liptr1}\,\ref{p5} and \eqref{convergence} that $\lim_{\La \to \infty} \mathrm{I}_{13} = 0$.
Combining the convergences for $\mathrm{I}_{11}$, $\mathrm{I}_{12}$ and $\mathrm{I}_{13}$, we conclude that 
\[
	\lim_{\La\to\infty}\lim_{h\to 0^+}\mathrm{I}_{1}=\iint_{Q_{R_{2},S_2}(z_0)}-\frac{1}{2}|v|^2\eta^{q-1}\pa_t\zeta_{\delta}\,dz.
\]

It remains to estimate $\mathrm{I_2}$. As in the argument for $\mathrm{I}_1$, we have
\[
    \lim_{h \to 0} \mathrm{I}_2= -\iint_{Q_{R_2,S_2}(z_0)}(u-\uu)\cdot\vla\eta^q\zeta_{\delta}\pa_t\zeta\,dz.
\]
Estimating with the absolute value and dividing into the good and bad sets, we obtain
\begin{align*}
	\begin{split}
		 &-\iint_{Q_{R_2,S_2}(z_0)}(u-\uu)\cdot\vla\eta^q\zeta_{\delta}\pa_t\zeta\,dz\\
   &\qquad\ge-\iint_{Q_{R_2,S_2}(z_0)\cap E(\La)}|u-\uu|^2 |\pa_t\zeta|\,dz
		-\iint_{ E(\La)^c}|u-\uu||\vla| |\pa_t\zeta|\,dz.
	\end{split}
\end{align*}
We show that the integral on the bad set vanishes as $\La \to \infty$. H\"older's inequality gives
\begin{align*}
	\begin{split}
		&\iint_{E(\La)^c}|u-\uu||\vla| |\pa_t\zeta|\,dz\\
		&\qquad\le \left(\iint_{Q_{R_2,S_2}(z_0)\cap E(\La)^c}|u-\uu|^2|\pa_t\zeta|^2\,dz\right)^\frac{1}{2}	\left(\iint_{E(\La)^c}|\vla|^2\,dz\right)^\frac{1}{2},
	\end{split}
\end{align*}
and the first integral vanishes when $\La \to \infty $ as $|u|\in L^2(\Om_T)$. 
For the second integral on the right-hand side, by Proposition~\ref{prop_liptr1}\,\ref{p6} and \eqref{convergence} we have
\[
    \lim_{\La \to \infty}\iint_{E(\La)^c}|\vla|^2\,dz \leq \lim_{\La \to \infty} \La^\frac{2}{p}|E(\La)^c|\leq \lim_{\La \to \infty}\La|E(\La)^c| =0.
\]
Therefore the integral over the bad set vanishes at the limit and we obtain
\[
	\lim_{\La\to\infty}\lim_{h\to 0^+}\mathrm{I}_2\ge-\iint_{Q_{R_2,S_2}(z_0)}|u-\uu|^2|\pa_t\zeta|\,dz.
\]
Combining the estimates for $\mathrm{I}_1$ and $\mathrm{I}_2$, we conclude that
\[
	\lim_{\La\to\infty}\lim_{h\to 0^+}\mathrm{I}
	\ge\iint_{Q_{R_{2},S_2}(z_0)}-\frac{1}{2}|v|\eta^{q-1}\pa_t\zeta_{\delta}\,dz-\iint_{Q_{R_2,S_2}(z_0)}|u-\uu|^2|\pa_t\zeta|\,dz.
\]

Next we estimate 
\[
		\mathrm{II}=
		\iint_{Q_{R_2,S_2}(z_0) }[\mA(\cdot,\na u)]_h\cdot\na \left(v^\La_h\eta^{q}\zeta\zeta_{\delta}\right)\,dz.
\]
We justify the convergence of the above terms as $h\to 0^+$. Observe that the structure condition \eqref{12} gives 
\[
|[\mathcal{A}(\cdot,\na u)]_h(z)|\le L[|\na u|^{p-1}+a|\na u|^{q-1}]_h(z).
\]
Moreover, we note that $|\na \left(v^\La_h\eta^{q}\zeta\zeta_{\delta}\right)|$
is uniformly bounded in $L^\infty(Q_{R_2,S_2}(z_0))$ for $h>0$ by Lemma~\ref{lem21} and Proposition~\ref{prop_liptr1}\,\ref{p7}. 
By convergence properties of the Steklov average, $[\mathcal{A}(\cdot,\na u)]_h\to \mathcal{A}(\cdot,\na u)$, $\na u_h\to \na u$ and $[u-u_0]_h\to u-u_0$ a.e. as $h\to 0^+$. This implies the convergence to the integral without Steklov averages and we obtain
\begin{align*}
	\begin{split}
		\lim_{h\to 0^+}\mathrm{II}&=
		\iint_{Q_{R_2,S_2(z_0)} \cap E(\La)}\mA(z,\na u)\cdot\na \left((u-\uu)\eta^{q+1}\zeta^2\zeta_{\delta}\right)\,dz\\
		&\qquad+\iint_{Q_{R_2,S_2}(z_0)\cap E(\La)^c}\mA(z,\na u)\cdot\na \left(\vla\eta^q\zeta\zeta_{\delta}\right)\,dz=\mathrm{II}_1+\mathrm{II}_2.
	\end{split}
\end{align*}

There is no $\Lambda$ inside the integral over the good set  and we can immediately let $ \La \to 0$. With the product rule of the gradient and estimating with \eqref{12}
and \eqref{42_1}, we get
\begin{align*}
	\begin{split}
		\lim_{\La\to\infty}\lim_{h\to0^+}\mathrm{II}_1 &= \iint_{Q_{R_2,S_2}(z_0)}\left(\mA(z,\na u)\cdot\na u\right) \eta^{q+1}\zeta^2\zeta_{\delta}\,dz\\
  &\qquad+ \iint_{Q_{R_2,S_2}(z_0)}\mA(z,\na u)\cdot(u-u_0)\na(\eta^{q+1})\zeta^2\zeta_{\delta}\,dz\\
&\ge \nu\iint_{Q_{R_2,S_2}(z_0)} (|\na u|^p+a(z)|\na u|^q)\eta^{q+1}\zeta^2\zeta_{\delta}\,dz\\
		&\qquad-(q+1)L\iint_{Q_{R_2,S_2}(z_0)} (|\na u|^{p-1}+a(z)|\na u|^{q-1})\frac{|u-\uu|}{R_2-R_1} \eta^q\zeta^2\zeta_{\delta} \,dz.
	\end{split}
\end{align*}
Applying Young's inequality twice with conjugate pairs $(\tfrac{p}{p-1},p)$ and $(\tfrac{q}{q-1},q)$, we obtain
\begin{align} \label{t_155}
    \begin{split}
        &(q+1)L\iint_{Q_{R_2,S_2}(z_0)} (|\na u|^{p-1}+a(z)|\na u|^{q-1})\frac{|u-\uu|}{R_2-R_1} \eta^q\zeta^2\zeta_{\delta} \,dz \\
        & \qquad \leq \frac{\nu}{2}\iint_{Q_{R_2,S_2}(z_0)} (|\na u|^p+a(z)|\na u|^q)\eta^{q+1}\zeta^{2}\zeta_{\delta} \,dz\\
        & \qquad\qquad + c(p,q,\nu,L)\iint_{Q_{R_2,S_2}(z_0)}\left(\frac{|u-\uu|^p}{(R_2-R_1)^{p}}+a(z)\frac{|u-\uu|^q}{(R_2-R_1)^q}\right)\,dz,
    \end{split}
\end{align}
where we were also estimated $\eta$ and $\zeta$ by one in the last term. After absorbing terms we arrive at the desired estimate
\begin{align}\label{t_156}
	\begin{split}
\lim_{\La\to\infty}\lim_{h\to0^+}\mathrm{II}_1& \geq \frac{\nu}{2}\iint_{Q_{R_2,S_2}(z_0)} (|\na u|^p+a(z)|\na u|^q)\eta^{q+1}\zeta^2\zeta_{\delta}\,dz\\
		& \qquad -c(p,q,\nu,L)\iint_{Q_{R_2,S_2}(z_0)}\left(\frac{|u-\uu|^p}{(R_2-R_1)^{p}}+a(z)\frac{|u-\uu|^q}{(R_2-R_1)^q}\right)\,dz.
	\end{split}
\end{align}

We claim that $\mathrm{II}_2$ vanishes at the limit. Indeed, by the product rule and \eqref{12} we have
\begin{align*}
    \begin{split}
        |\mathrm{II}_2|&\leq\iint_{Q_{R_2,S_2}(z_0)\cap E(\La)^c}|\mA(z,\na u)|\left|(\na \vla) \eta^q\zeta \zeta_\delta + \vla q(\na\eta) \eta^{q-1}\zeta\zeta_{\delta}\right|\,dz\\
        &\leq c(L,q,R_2,R_1)\iint_{Q_{R_2,S_2}(z_0)\cap E(\La)^c}(|\na u|^{p-1}+a(z)|\na u|^{q-1})(|\na \vla|+ |\vla|)\,dz.
    \end{split}
\end{align*}
Applying Young's inequality, Proposition~\ref{prop_liptr1}\,\ref{p6} and \eqref{convergence}, we conclude that
\[
    \lim_{\La\to \infty}|\mathrm{II}_2| \leq c\lim_{\La \to \infty} \left( \iint_{Q_{R_2,S_2}(z_0)\cap E(\La)^c}(|\na u|^p+a(z)|\na u|^q) \,dz +\La|E(\La)|^c\right)=0,
\]
where the convergence of the integral is justified by the integrability assumption of a weak solution. Therefore, we got from \eqref{t_156}
the desired estimate  
\begin{align*}
	\begin{split}
		\lim_{\La\to\infty}\lim_{h\to0^+}\mathrm{II}&\ge\frac{\nu}{2}\iint_{Q_{R_2,S_2}(z_0)} H(z,|\na u|)\eta^{q+1}\zeta^2\zeta_{\delta}\,dz\\
		&\qquad-c\iint_{Q_{R_2,S_2}(z_0)}H\left(z,\frac{|u-\uu|}{R_2-R_1}\right)\,dz.
	\end{split}
\end{align*}

The argument to estimate $\mathrm{III}$ is very similar to the argument for $\mathrm{II}$. 
Again, we split into good and bad parts to get
\begin{align*}
	\begin{split}
		\lim_{h\to 0^+}\mathrm{III}&=
		\iint_{Q_{R_2,S_2}(z_0)\cap E(\La)}\left(|F|^{p-2}F+a(z)|F|^{q-2}F\right)\cdot \na\left((u-\uu)\eta^{q+1}\zeta^2\zeta_{\delta}\right)\,dz\\
		&\qquad+\iint_{Q_{R_2,S_2}(z_0)\cap E(\La)^c}\left(|F|^{p-2}F+a(z)|F|^{q-2}F\right)\cdot \na\left(\vla\eta^q\zeta\zeta_{\delta}\right)\,dz=\mathrm{III}_1+\mathrm{III}_2.
	\end{split}
\end{align*}

Applying Young's inequality as in \eqref{t_155}, we get 
\begin{align*}
	\begin{split}
		\lim_{\La\to\infty}\lim_{h\to0^+}\mathrm{III}_1
		&\le c\iint_{Q_{R_2,S_2}(z_0)}\left(H\left(z,\frac{|u-\uu|}{R_2-R_1}\right)+H(z,|F|)\right)\,dz\\
		&\qquad+\frac{\nu}{4}\iint_{Q_{R_2,S_2}(z_0)}H(z,|\na u|)\eta^{q+1}\zeta^2\zeta_{\delta}\,dz,
	\end{split}
\end{align*}
where $c=c(p,q,\nu)$.
Using the same arguments with $|\na u|$ replaced by $|F|$, we can estimate $\mathrm{III}_2$ in the same way as $\mathrm{II}_2$. Therefore, we conclude that $\lim_{\La\to\infty}\lim_{h\to0^+}\mathrm{III}_2=0$ and
\begin{align*}
	\begin{split}
		\lim_{\La\to\infty}\lim_{h\to0^+}\mathrm{III}
		& \le  c\iint_{Q_{R_2,S_2}(z_0)}\left(H\left(z,\frac{|u-\uu|}{R_2-R_1}\right)+H(z,|F|)\right)\,dz\\
		&\qquad+\frac{\nu}{4}\iint_{Q_{R_2,S_2}(z_0)}H(z,|\na u|)\eta^{q+1}\zeta^2\zeta_{\delta}\,dz.
	\end{split}
\end{align*}

Combining the estimates for $\mathrm{I}, \mathrm{II} $ and $\mathrm{III}$ at the limit of $\La$ and $h$, reordering and absorbing the last term of $\mathrm{III}$ with the first term of $\mathrm{II}$, we get
\begin{align*}
\begin{split}
    &\iint_{Q_{R_{2},S_2}(z_0)}-\frac{1}{2}|v|^2\eta^{q-1}\pa_t\zeta_{\delta}\,dz + \frac{\nu}{4}\iint_{Q_{R_2,S_2}(z_0)}H(z,|\na u|)\eta^{q+1}\zeta^2\zeta_{\delta}\,dz
		  \\
   &\qquad \leq c\iint_{Q_{R_2,S_2}(z_0)}\left(H\left(z,\frac{|u-\uu|}{R_2-R_1}\right)+|u-\uu|^2|\pa_t\zeta|+H(z,|F|)\right)\,dz.
\end{split}
\end{align*}
Taking $\de\to0^+$ and recalling that $\tau\in \ell_{S_1}(t_0)$ in \eqref{3_zetade} is arbitrary, we conclude that
\begin{align*}
	\begin{split}
		&\sup_{t\in \ell_{S_1}(t_0)}\fint_{B_{R_{1}}(x_0)}\frac{|u-u_0|^2}{S_1}\,dx+\fiint_{Q_{R_1,S_1}(z_0)} H(z,|\na u|) dz\\
		&\qquad\le c\fiint_{Q_{R_2,S_2}(z_0)}\left(H\left(z,\frac{|u-\uu|}{R_2-R_1}\right) + \frac{|u-\uu|^2}{S_2-S_1} + H(z,|F|)\right)\,dz,
	\end{split}
\end{align*}
where we made the integration domain smaller on the left hand side and used the definitions for the cutoff functions in \eqref{42_1}. Also note the use of the fact that $|Q_{R_2,S_2}|\approx |Q_{R_1,S_1}|$ by $R_1\in[R_2/2,R_2)$ and $S_1\in [S_2/2^2,S_2)$. The proof is completed by replacing $u_0$ with $(u)_{Q_{R_1,S_1}(z_0)}$ on the first term on the left-hand side.

\section{Existence and uniqueness}
In this section, we consider the Dirichlet problem \eqref{e2}. 
We apply the following standard estimates for the $p$-Laplace structure with $p\geq2$. 
There exists a constant $c=c(n,N,p)$ such that
\begin{align}\label{e3}
    c|\xi_1-\xi_2|^p\le (|\xi_1|^{p-2}\xi_1-|\xi_2|^{p-2}\xi_2)\cdot(\xi_1-\xi_2)
\end{align}
and
\begin{align}\label{e4}
    ||\xi_1|^{p-2}\xi_1-|\xi_2|^{p-2}\xi_2|\le c(|\xi_1|^{p-2}+|\xi_2|^{p-2})|\xi_1-\xi_2|
\end{align}
for every  $\xi_1,\xi_2\in\mathbb{R}^{Nn}$.

\subsection{Proof of Theorem~\ref{main2}}\label{subsec_uni}
In this subsection, we provide the uniqueness of weak solutions to \eqref{e2}.
Assume that $u$ and $w$ are weak solutions of \eqref{e2}.
Then we have \eqref{e5} with \eqref{e7} and \eqref{e8}.
Note that since $(u-w)(\cdot,R^2)\not\equiv0$ in $D_R$ may happen, the truncated function is not able to be defined as zero on the bad set as in \eqref{422_2}. We overcome this difficulty by extending the system \eqref{e5} to $D_R\times \mathbb{R}$.
We first extend $u\equiv w\equiv g$ in $D_R\times (-\infty,-R^2]$. Then we have
\begin{align}\label{e9}
    (u-w)_t-\dv (\mathcal{A}(z,\na u)-\mathcal{A}(z,\na w))=0 \quad\text{in}\  D_R\times (-\infty,R^2).
\end{align}
Indeed, for any $\varphi\in C_0^\infty(D_R\times (-\infty,R^2),\RR^N)$, we have
\[
        -\iint_{D_R\times (-\infty,R^2)}(u-w)\cdot \varphi_t\,dz
        =-\lim_{\ep\to0^+}\iint_{ C_R}(u-w)\cdot (\varphi\zeta_\epsilon)_t\,dz,
\]
where $\zeta_\epsilon$ is a function defined as
\begin{align*}
    \zeta_\epsilon(t)=
    \begin{cases}
        0,&t\in(-\infty,-R^2),\\
        \frac{1}{\epsilon}t,&t\in[-R^2,-R^2+\ep],\\
        1,&t\in(-R^2+\ep,\infty).
    \end{cases}
\end{align*}
Here we used \eqref{e8} to have
\[
        \lim_{\ep\to0^+}\iint_{ C_R}(u-w)\cdot\varphi(\zeta_\epsilon)_t\,dz=\lim_{\ep\to0^+}\fint_{-R^2}^{-R^2+\ep}\int_{D_R}(u-w)\cdot \varphi\,dx\,dt=0.
\]
Taking $\varphi\zeta_\epsilon$ as a test function in \eqref{e5}, it follows that
\begin{align*}
    \begin{split}
        &-\iint_{D_R\times (-\infty,R^2)}(u-w)\cdot \varphi_t\,dz=-\lim_{\ep\to0^+}\iint_{ C_R}(u-w)\cdot(\varphi\zeta_\epsilon)_t\,dz\\
        &\qquad=-\lim_{\ep\to0^+}\iint_{ C_R}(\mathcal{A}(z,\na u)-\mathcal{A}(z,\na w))\cdot \na\varphi \zeta_\epsilon\,dz\\
        &\qquad=-\iint_{D_R\times (-\infty,R^2)}(\mathcal{A}(z,\na u)-\mathcal{A}(z,\na w))\cdot \na \varphi\,dz.
    \end{split}
\end{align*}
This means that \eqref{e9} holds. For the range $D_R\times \RR$, we remark that there is no possible extension to have $(u-w)_t-\dv (\mathcal{A}(z,\na u)-\mathcal{A}(z,\na w))=0$ in $D_R\times \mathbb{R}$.
In fact, by considering $G(z)=(\mathcal{A}(z,\na u)-\mathcal{A}(z,\na w))$ and extending $u$, $w$ and $G$ to be
\begin{align}\label{e16_2}
    (u-w)(x,t)=(u-w)(x,2R^2-t),\quad G(x,t)=-G(x,2R^2-t)
\end{align}
for $(x,t)\in D_R\times (R^2,\infty)$, we have $(u-w)_t-\dv G=0$ in $D_R\times \mathbb{R}$.
For any $\varphi\in C_0^\infty(D_R\times \RR,\RR^{N})$, consider
\[
    \psi(x,t)=\varphi(x,t)-\varphi(x,2R^2-t).
\]
Note that $\psi\in W^{1,\infty}_0(D_R\times(-\infty,R^2),\RR^N)$. Taking $\psi$ as a test function in \eqref{e9}, we obtain
\begin{align*}
    \begin{split}
        0
    &=\iint_{D_R\times (-\infty,R^2)}-(u-w)\cdot\psi_t\,dz+\iint_{B_R\times (-\infty,R^2)}G\cdot\na\psi\,dz\\
        &=\iint_{D_R\times (-\infty,R^2)}\left(-(u-w)\cdot\varphi_t+G\cdot\na\varphi\right)\,dz\\
        &\qquad+\iint_{D_R\times (-\infty,R^2)}\left(-(u-w)(x,t)\cdot\varphi_t(x,2R^2-t)-G(x,t)\cdot \na \varphi(x,2R^2-t)\right)\,dz,
    \end{split}
\end{align*}
and a change of variables gives
\[
    \iint_{D_R\times\RR}\left(-(u-w)\cdot\varphi_t+G\cdot\na\varphi\right)\,dz=0.
\]
Therefore, $u-w\equiv0$ in $(D_R\times (-R^2,3R^2))^c$ and 
\begin{align}\label{e16}
    (u-w)_t-\dv G=0\quad \text{in}\ D_R\times \RR.
\end{align}
We also observe that if we extend 
\[
    u(x,t)=u(x,2R^2-t),\quad w(x,t)=w(x,2R^2-t),\quad a(x,t)=a(x,2R^2-t)
\]
in $(x,t)\in D_R\times (-R^2,3R^2)$
then the first extension of \eqref{e16_2} is satisfied and
\[
    \na u(x,t)=\na u(x,2R^2-t),\quad \na w(x,t)=\na w(x,2R^2-t)
\]
is well-defined in $(x,t)\in D_R\times (-R^2,3R^2)$ with
\[
    H(z,|\na u|),\, H(z,|\na w|)\in L^1(D_R\times (-R^2,3R^2)).
\]

With these extensions we  define a function $f:\mathbb{R}^{n+1}\to \mathbb{R}$ as
\[
    f(z)=(|\na u(z)|+|\na w(z)|+|u(z)|+|w(z)|)\rchi_{D_{R}\times (-R^2,3R^2)}(z).
\]
It is easy to see that $H(z,f)\in L^1(\mathbb{R}^{n+1})$. The function $f$ plays the same role in constructing an open subset in $\mathbb{R}^{n+1}$ as in Section~\ref{sec_liptr}.
To employ the Lipschitz truncation method, it needs to be shown that $|G|$ is bounded by $f^{p-1}+af^{q-1}$. Indeed, using \eqref{e4} and \eqref{d9}, there exists $c=c(n,N,p,q,L)$ such that
\[
    |G|\le c(|\na u|^{p-2}+|\na w|^{p-2})|\na u-\na w|+ca(|\na u|^{q-2}+|\na w|^{q-2})|\na u-\na w|,
\]
and applying triangle inequality, to have $|\na u-\na w| \leq |\na u|+|\na w|$, and Young's inequality, we obtain
\begin{align}\label{e18}
    |G|\le c\left((|\na u|+|\na w|)^{p-1}+a(|\na u|+|\na w|)^{q-1}\right)\le c(f^{p-1}+af^{q-1})
\end{align}
in $D_{R}\times \RR$ as $G$ is extended $0$ below $\{t=-R^2\}$ and then extended oddly above $\{t=R^2\}$. Denoting $v=u-w$, we have that $v_h$ is well defined in $D_{R}\times \mathbb{R}$ and 
\begin{align}\label{e19}
    \pa_tv_h-\dv G_h=0\quad\text{in}\ D_{R}\times \{t\}
\end{align}
for a.e. $t\in \mathbb{R}$. We follow the proof of Lemma~\ref{lem32} by taking the test function $\varphi\psi_\delta$ in \eqref{e16} and \eqref{e19} with \eqref{e18} to obtain the following estimates. We omit the details. 
We keep track of the dependencies and denote
\[
  \data=n,N,p,q,\alpha,\nu,L,[a]_\alpha,R,K,\|a\|_{L^\infty(\RR^{n+1})}.
\]
      
\begin{lemma}
    Let $Q=B_r\times \ell_s$ be a cylinder and satisfying $B_r\subset D_R$. There exists a constant $c=c(\data)$ such that
    \[
        \fiint_{Q}|v_h-(v_h)_Q|\,dz\le csr^{-1}\fiint_{[Q]_h}\left(f^{p-1}+af^{q-1}\right)\,dz+cr\fiint_{[Q]_h}f\,dz.
    \]
    Moreover, the above estimate holds with $v_h$ and $[Q]_h$ replaced by $v$ and $Q$.
\end{lemma}

 As \eqref{e19} is defined on $\RR$, the second estimate in Lemma~\ref{lem32} is not necessary. On the other hand, we have used Lemma~\ref{lem22} to deal with the case of $B_r$ intersecting with the complement of the given domain in Section~\ref{sec_liptr}. Since a cube $D_R$ satisfies the measure density condition, we may replace $B_{\rho}(x_0)$ by $D_R$ in the statement of Lemma~\ref{lem22}. We refer to \cite[Example 6.18 and Theorem 6.22]{MR4306765} for details. It is also noteworthy to mention that $D_R$ has a Lipschitz boundary and this fact along with \eqref{e7} guarantees that $u-w\in L^s(I_{3R};W_0^{1,s}(D_R;\mathbb{R}^N))$ for $1<s\le p$. Therefore, we may apply Lemma~\ref{lem22} with $1<s<p$ as in Section~\ref{sec_liptr}.
With the Whitney decomposition $\{U_i=B_i\times I_i\}_{i\in\mathbb{N}}$ and partition of unity $\{\om\}_{i\in\mathbb{N}}$, we define the Lipschitz truncation of $v_h$ and $v$ as
\[
    v_h^\La(z)=v_h(z)-\sum_{i\in\mathbb{N}}(v_h(z)-v_h^i)\om_i(z)
    \quad\text{and}\quad
    v^\La(z)=v(z)-\sum_{i\in\mathbb{N}}(v(z)-v^i)\om_i(z),
\]
where
\begin{align*}
    v_h^i=
    \begin{cases}
        (v_h)_{2U_i},&\text{if}\ 2B_i\subset D_R,\\
        0,&\text{otherwise},
    \end{cases}
    \qquad v^i=
    \begin{cases}
        (v)_{2U_i},&\text{if}\ 2B_i\subset D_R,\\
        0,&\text{otherwise}.
    \end{cases}
\end{align*}
We now state the Lipschitz truncation result analogous to Proposition~\ref{prop_liptr1} in Section~\ref{sec_liptr}. Again, the proof is omitted.
\begin{proposition}\label{uni_prop}
Let $E(\La)$ be as in \eqref{def_E}.
There exist functions $\{v_h^\Lambda\}_{h>0}$ and a function $v^\Lambda$ satisfying the following properties: 
\begin{enumerate}[label=(\roman*),series=theoremconditions]
    \item $v_h^\Lambda\in W^{1,2}(I_{3R};L^2(D_R;\RR^N))\cap L^\infty(I_{3R};W_0^{1,\infty}(D_R,\RR^N))$.
    \item $v^\Lambda\in L^\infty(I_{3R};W_0^{1,\infty}(D_R,\RR^N))$.
    \item $v_h^\La= v_h$, $v^\La= v$, $\na v_h^\La= \na v_h$, $\na v^\La= \na v$ a.e. in $E(\La)$.
    \item $v_h^\Lambda\to v^\Lambda$ in $L^\infty(D_{R}\times I_{3R},\RR^N)$ as $h\to0^+$, taking a subsequence if necessary.
    \item $\na v_h^\La\to \na v^\La$ and $\pa_t v_h^\La\to \pa_tv^\La$ a.e. in $E(\La)^c$ as $h\to0^+$.
\end{enumerate}
Moreover, there exist constants $c=c(\data)$ and $c_\La=c_\La(\data,\La)$ such that
\begin{enumerate}[resume*=theoremconditions]
    \item $\begin{aligned}[t]
    \iint_{E(\La)^c}|v-v^{\Lambda}||\pa_tv^\Lambda|\,dz\le c\Lambda|E(\La)^c|,
    \end{aligned}$
    \item $\begin{aligned}[t]
    \iint_{E(\La)^c}|v_h-v_h^\La||\pa_t v_h^\La|\,dz\le c_\La|E(\La)^c|,
    \end{aligned}$
    \item $H(z,|v^\La(z)|)+H(z,|\na v^\La(z)|)\le c\La$ for a.e. $z\in\mathbb{R}^{n+1}$,
    \item $H(z,|v_h^\La(z)|)+H(z,|\na v_h^\La(z)|)\le c_\La$ for a.e. $z\in \RR^{n+1}.$
\end{enumerate}
\end{proposition}

In the remaining of this subsection, we apply the energy estimate and conclude Theorem~\ref{main2}.
For $\tau_1,\tau_2\in I_{R}$ with $\tau_1<\tau_2$, let $\delta>0$ and $h>0$ be sufficiently small so that $2\delta<\tau_2-\tau_1$ and $\tau_1,\tau_2\in I_{R-h}$. We let $\zeta_\delta\in W^{1,\infty}_0(I_{R-h})$ be a function defined as
\begin{align*}
    \zeta_\delta(t)=
    \begin{cases}
        \tfrac{1}{\delta}(t-\tau_1),&\tau_1\le t\le \tau_1+\delta,\\
        1, &\tau_1+\delta\le t\le \tau_2-\delta,\\
        1-\tfrac{1}{\delta}(t-\tau_2+\delta),&\tau_2-\delta\le t\le \tau_2,\\
        0, &\text{otherwise}.
    \end{cases}
\end{align*}
Let $t\in I_{R}$. We take $v_h^{\Lambda}\zeta_\delta$ as a test function in \eqref{e19}, that is equivalently
\[
    \pa_t[u-w]_h-\dv\left[\mathcal{A}(\cdot,\na u)-\mathcal{A}(\cdot,\na w)\right]_h=0\quad\text{in}\ D_R\times\{t\}.
\]
Recalling that $u-w=v$, we obtain
\begin{align}\label{e25}
    0=\iint_{ C_R}\pa_tv_h\cdot v_h^\Lambda\zeta_\delta\,dz+\iint_{ C_R}\left[\mathcal{A}(\cdot,\na u)-\mathcal{A}(\cdot,\na w)\right]_h\cdot \na v^\Lambda_h\zeta_\delta\,dz.
\end{align}
To estimate the first term on the right-hand side, we add and subtract $\pa_tv_h^\La$, then apply integration by parts to have
\begin{align*}
\begin{split}
    &\iint_{ C_R}\pa_tv_h\cdot v_h^\Lambda\zeta_\delta\,dz
        =\iint_{ C_R}\pa_t(v_h-v_h^\La)\cdot v_h^\Lambda\zeta_\delta\,dz+\iint_{ C_R}\frac{1}{2}(\pa_t|v_h^\Lambda|^2)\zeta_\delta\,dz\\
        &\qquad=-\iint_{ C_R}(v_h-v_h^\La)\cdot \pa_tv_h^\Lambda \zeta_\delta\,dz
         -\iint_{ C_R}(v_h-v_h^\La)\cdot v_h^\Lambda \pa_t\zeta_\delta\,dz
         -\iint_{ C_R}\frac{1}{2}|v_h^\Lambda|^2\pa_t\zeta_\delta\,dz.
\end{split}
\end{align*}
Following the argument in the proof of Theorem~\ref{thm_cac} by using Proposition~\ref{uni_prop}, we deduce
\begin{align*}
\begin{split}
    &\lim_{\La\to\infty}\lim_{h\to0^+}\iint_{ C_R}\pa_tv_h\cdot v_h^\Lambda\zeta_\delta\,dz\\
    &\qquad=\fint_{\tau_2-\delta}^{\tau_2}\int_{D_R}\frac{1}{2}|v|^2 \,dx\,dt-\fint_{\tau_1}^{\tau_1+\delta}\int_{D_R}\frac{1}{2}|v|^2 \,dx\,dt.
\end{split}
\end{align*}
The last term in \eqref{e25} is estimated as in the proof of Theorem~\ref{thm_cac} with Proposition~\ref{uni_prop}. Then we obtain
\begin{align*}
    \begin{split}
        &\lim_{\La\to \infty}\lim_{h\to0^+}\iint_{ C_R}\left[\mathcal{A}(\cdot,\na u)-\mathcal{A}(\cdot,\na w)\right]_h\cdot \na v^\Lambda_h\zeta_\delta\,dz\\
        &\qquad=\iint_{ C_R}(\mathcal{A}(\cdot,\na u)-\mathcal{A}(\cdot,\na w))\cdot (\na u-\na w)\zeta_\delta\,dz.
    \end{split}
\end{align*}
By \eqref{d9} and \eqref{e3} there exists a constant $c=c(n,N,p,q,\nu)$ such that
\begin{align*}
\begin{split}
    &\iint_{ C_R}(\mathcal{A}(\cdot,\na u)-\mathcal{A}(\cdot,\na w))\cdot (\na u-\na w)\zeta_\delta\,dz\\
    &\qquad\geq c\iint_{ C_R}\left(|\na u-\na w|^p+a|\na u-\na w|^q\zeta_\delta\right)\,dz\\
    &\qquad=c\iint_{ C_R}H(z,|\na u-\na w|)\zeta_\delta\,dz\ge 0.
\end{split}
\end{align*}
Hence, it follows from \eqref{e25} that
\[
    \fint_{\tau_2-\delta}^{\tau_2}\int_{D_R}|u-w|^2 \,dx\,dt\le\fint_{\tau_1}^{\tau_1+\delta}\int_{D_R}|u-w|^2 \,dx\,dt,
\]
where $\tau_1,\tau_2\in I_R$ are arbitrary and $\delta$ is sufficiently small. We set $\tau_1=\delta$ to have 
\[
    \fint_{\tau_2-\delta}^{\tau_2}\int_{D_R}|u-w|^2 \,dx\,dt
    \le \fint_{\delta}^{2\delta}\int_{D_R}|u-w|^2 \,dx\,dt\le 2\fint_{0}^{2\delta}\int_{D_R}|u-w|^2 \,dx\,dt.
\]
Letting $\delta \to 0^+$, it follows from \eqref{e8} that
\[
    \int_{D_R}|u-w|^2(x,\tau_2)\,dx\le 0
\]
for a.e. $\tau_2\in I_R$. We conclude $u= w$ a.e. in $ C_R$. This completes the proof of Theorem~\ref{main2}.

\subsection{Proof of Theorem~\ref{main3}}
In this subsection we assume $g\equiv0$ and $b(\cdot)$ satisfies VMO condition \eqref{la5}. We construct a sequence of weak solutions of the Dirichlet boundary problem for perturbed $q$-Laplace type systems. This sequence converges to a weak solution of \eqref{e2}.

\subsubsection{Approximation} 
For $k\in\mathbb N$, let $F^k$ be the truncated function
\[
    F^k(z)=F(z)\rchi_{\{|F(z)|\le k\}}.
\]
For $\epsilon>0$, we consider the function
\begin{align}\label{d98}
    F^k_\epsilon(z)=F^k(z)\rchi_{\{\left|F^k(z)\right|\le \ep^{-\frac{1}{q-p}}\}}.
\end{align}
and the perturbed operator and map
\[
    H_\epsilon(z,|\xi|)=|\xi|^p+a_\ep(z)|\xi|^q,\quad
    \mathcal{A}_{\epsilon}(z,\xi)=b(z)(|\xi|^{p-2}\xi +a_\ep(z)|\xi|^{q-2}\xi),
\]
where $a_\ep(z)=a(z)+\epsilon$.
We observe $\mathcal{A}_\epsilon(z,\xi)$ is $q$-Laplace type operator as it satisfies
\[
    \epsilon|\xi|^q\le \mathcal{A}_\epsilon(z,\xi)\cdot\xi,\quad |\mathcal{A}_\epsilon(z,\xi)|\le 2^{q}L(\|a\|_{\infty}+1)(|\xi|^{q-1}+1)
\]
and \eqref{e3} implies the monotonicity condition
\[
    (\mathcal{A}_\epsilon(z,\xi_1)-\mathcal{A}_\epsilon(z,\xi_2))\cdot (\xi_1-\xi_2)> 0 \quad\text{for any}\ \xi_1,\xi_2\in\mathbb{R}^{Nn}\ \text{with}\ \xi_1\ne\xi_2.
\]
For each $k\in\mathbb N$ and $\epsilon>0$,
there exists by standard existence theory for $q$-parabolic equations the unique weak solution $u^k_\epsilon\in C(I_R;L^2(D_R,\RR^N))\cap L^q(I_R;W^{1,q}_0(D_R,\RR^N))$ to the Dirichlet problem
\begin{align}\label{d99}
    \begin{cases}
        \pa_tu^k_\epsilon-\dv \mathcal{A}_{\epsilon}(z,\na u_\epsilon^k)=-\dv(|F_\ep^k|^{p-2}F_\ep^k+a_\ep|F_\ep^k|^{q-2}F_\ep^k)&\text{in}\ C_R,\\
        u^k_\epsilon=0\null\hfill&\text{on}\ \pa_p C_R.
    \end{cases}
\end{align}
\subsubsection{First convergence}
By the standard energy estimate for the $q$-Laplace problems, there exists a constant $c=c(p,q,\nu)$ such that
\begin{align}\label{e40}
        \sup_{t\in I_R}\int_{D_R}|u_\ep^k(x,t)|^2\,dx+\iint_{ C_R} H_\epsilon(z,|\na u^k_\epsilon|)\,dz\le c\iint_{ C_R}H_\epsilon(z,|F^k_\epsilon|)\,dz.
\end{align}
We observe from \eqref{d98} that
\begin{align}\label{e41}
    \epsilon|F^k_\epsilon(z)|^q=\epsilon|F^k_\epsilon(z)|^{q-p}|F^k_\epsilon(z)|^p\le |F^k(z)|^p\le|F(z)|^p.
\end{align}
Therefore, \eqref{e40} becomes
\begin{align}\label{e42}
    \sup_{t\in I_R}\int_{D_R}|u_\ep^k(x,t)|^2\,dx+\iint_{ C_R}H_\epsilon(z,|\na u^k_\epsilon|)\,dz\le c\iint_{ C_R}H(z,|F|)\,dz.
\end{align}
Let $k\in\mathbb N$. First we investigate the weak limit of $\{u^k_\epsilon\}_{\epsilon>0}$ when $\epsilon \to 0$. 
To apply a compactness result, we estimate $\{\pa_t u^k_\epsilon\}_{\ep>0}$.
Let $\varphi\in C_0^\infty( C_R,\RR^N)$ be a test function in \eqref{d99}. By \eqref{d9} there exists a constant $c=c(L)$ such that
\begin{align*}
    \begin{split}
        &\left|\iint_{ C_R}\pa_tu^k_\epsilon\cdot\varphi\,dz\right|
        \le c\iint_{ C_R}(|\na u^k_\epsilon|^{p-1}+|F^k_\epsilon|^{p-1})|\na \varphi|\,dz\\
        &\qquad+c\iint_{ C_R}a_\ep(z)(|\na u^k_\epsilon|^{q-1}+|F^k_\epsilon|^{q-1})|\na \varphi|\,dz.
    \end{split}
\end{align*}
By H\"older's inequality, there exists a constant $c=c(n,p,q,L,\lVert a\rVert_{\infty},R)$ such that
\begin{align*}
    \begin{split}
        &\left|\iint_{ C_R}\pa_tu^k_\epsilon\cdot\varphi\,dz\right|
        \le c\left(\iint_{ C_R}|\na u^k_\epsilon|^{p}+|F_\epsilon^k|^{p})\,dz\right)^\frac{1}{p'}\left(\iint_{ C_R}|\na\varphi|^{q}\,dz\right)^\frac{1}{q}\\
        &\qquad+c\left(\iint_{ C_R}a_\ep(z)(|\na u^k_\epsilon|^{q}+|F_\epsilon^k|^{q})\,dz\right)^\frac{1}{q'}\left(\iint_{ C_R}|\na\varphi|^q\,dz\right)^\frac{1}{q},
    \end{split}
\end{align*}
and using \eqref{e40} and \eqref{e42}, we obtain
\[
    \left|\iint_{ C_R}\pa_tu^k_\epsilon\cdot \varphi\,dz\right|\le c\left(\iint_{ C_R} H(z,|F^k|)\,dz+1\right)^\frac{1}{q'}\left(\iint_{ C_R}|\na\varphi|^q\,dz\right)^\frac{1}{q}.
\]
Thus we have
\begin{align}\label{e45}
     \lVert \pa_t u_\epsilon^k\rVert_{L^{q'}(I_R;W^{-1,q'}(D_R,\RR^N))}\le c\left(\iint_{ C_R} H(z,|F|)\,dz+1\right)^\frac{1}{q'}.
\end{align}

Since the function $b(\cdot)$ satisfies the VMO condition \eqref{la5}, we may apply a global Calder\'on-Zygmund estimate in \cite[Theorem 2.5]{CZ} for a weak solution to \eqref{d99}. Since \eqref{e41} holds, $[a+\ep]_{\alpha}=[a]_\alpha$ and $\|a+\epsilon\|_{L^\infty( C_R)}\le \|a\|_{L^\infty( C_R)}+1$, there exists a constant 
\[
c_k=c_k(n,N,p,q,[a]_{\alpha},R,\lVert H(z,|F^k|)\rVert_{L^\frac{q}{p}( C_R)},\|a\|_{L^\infty( C_R)})
\]
such that
\[
        \iint_{ C_R} H_\epsilon(z,|\na u^k_\epsilon|)^\frac{q}{p}\,dz\le c_k \quad\text{for every}\ \ep>0.
\]
In particular, we have
\begin{align}\label{e47}
    \iint_{ C_R}|\na u^k_\epsilon|^q\,dz\le c_k \quad\text{for every}\ \ep>0.
\end{align}
Since estimates \eqref{e45} and \eqref{e47} hold for every $\ep>0$ and we have 
\[
    L^q(I_R;W^{1,q}_0(D_R,\RR^N))\Subset L^q( C_R,\RR^N)\Subset  L^2( C_R,\RR^N) \subset L^{q'}(I_R;W^{-1,q'}(D_R,\RR^N)),
\]
a compactness result in \cite[Corollary 4, Section 8]{MR916688} implies that there exists a limit function $u^k\in C(I_R;L^2(D_R,\RR^N))\cap L^q(I_R;W^{1,q}_0(D_R,\RR^N))$ such that, by passing to a subsequence if necessary, we have
\begin{align*}
\begin{split}
    &u^k_\epsilon \to  u^k    \quad \text{strongly in}\  C(I_R;L^2(D_R,\RR^N))\cap L^q( C_R,\RR^N),\\
     &u^k_\epsilon \to u^k    \quad \text{weakly in}\ L^q(I_R;W^{1,q}_0(D_R,\RR^N)),\\
     &\pa_tu^k_\epsilon \to \pa_t u^k   \quad \text{weakly in}\ L^{q'}(I_R;W^{-1,q'}(D_R,\RR^N)),
\end{split}
\end{align*}
as $\ep\to0^+$. We claim that $u^k$ is the weak solution of
\begin{align}\label{e50}
    \begin{cases}
        \pa_tu^k-\dv\mathcal{A}(z,\na u^k)=-\dv(|F^k|^{p-2}F^k+a|F^k|^{q-2}F^k)&\text{in}\ C_R,\\
        u^k=0\null\hfill&\text{on}\ \pa_p C_R.
    \end{cases}
\end{align}
It suffices to show that
\begin{align}\label{e51}
    \iint_{ C_R}H(z,|\na u_\epsilon^k-\na u^k|)\,dz\to0 \quad\text{as}\ \epsilon\to0^+.
\end{align}
Indeed, the above strong convergence provides the weak convergence of operators. For any $\varphi\in C_0^\infty( C_R,\RR^N)$ there holds
\begin{align*}
    \begin{split}
        &\left|\iint_{ C_R}\left(\mathcal{A}_\epsilon(z,\na u^k_{\epsilon})-\mathcal{A}(z,\na u^k)\right)\cdot \na \varphi\,dz\right|\\
        &\qquad\le  \iint_{ C_R}|\mathcal{A}(z,\na u_\ep^k)-\mathcal{A}(z,\na u^k)||\na \varphi|\,dz+L\epsilon\iint_{ C_R}|\na u^k_\epsilon|^{q-1}|\na \varphi|\,dz.
    \end{split}
\end{align*}
By H\"older's inequality and \eqref{e42}, we obtain
\[
    \lim_{\ep\to0^+}\epsilon\iint_{ C_R}|\na u^k_\epsilon|^{q-1}|\na \varphi|\,dz\le \lim_{\ep\to0^+}\epsilon| C_R|^{1-\frac{q-1}{p}}\left(\iint_{ C_R}|\na u^k_\epsilon|^{p}\,dz\right)^\frac{q-1}{p}\| \na\varphi\|_{\infty}=0.
\]
We apply \eqref{e4} to estimate the remaining term. There exists a constant $c=c(n,N,p,q)$ such that
\begin{align*}
    \begin{split}
       &\iint_{ C_R}|\mathcal{A}(z,\na u_\ep^k)-\mathcal{A}(z,\na u^k)||\na \varphi|\,dz\\
    &\qquad\le c\iint_{ C_R}\left(|\na u^k_\epsilon|^{p-2}+|\na u^{k}|^{p-2}\right)|\na u^k_\epsilon-\na u^k||\na\varphi|\,dz\\
    &\qquad\qquad+c\iint_{ C_R}a(z)\left(|\na u^k_\epsilon|^{q-2}+|\na u^{k}|^{q-2}\right)|\na u^k_\epsilon-\na u^k||\na \varphi|\,dz.
    \end{split}
\end{align*}
Applying H\"older's inequality and \eqref{e42}, there exists $c=c(n,N,p,q,\lVert a\rVert_{\infty},R)$ such that
\begin{align*}
   \begin{split}
       &\iint_{ C_R}|\mathcal{A}(z,\na u_\ep^k)-\mathcal{A}(z,\na u^k)||\na \varphi|\,dz\\
       &\qquad\le c\left( \iint_{ C_R}H(z.|F|)\,dz\right)^\frac{p-2}{p}\left( \iint_{ C_R}|\na u^k_\epsilon-\na u^k|^p\,dz\right)^\frac{1}{p}\| \na\varphi\|_{\infty}\\
       &\qquad\qquad +c\left( \iint_{ C_R}H(z,|F|)\,dz\right)^\frac{q-2}{q}\left( \iint_{ C_R}a(z)|\na u^k_\epsilon-\na u^k|^q\,dz\right)^\frac{1}{q}\| \na\varphi\|_{\infty}.
   \end{split}
\end{align*}
Therefore, \eqref{e51} leads to
\[
    \lim_{\ep\to0^+}\iint_{ C_R}|\mathcal{A}(z,\na u_\ep^k)-\mathcal{A}(z,\na u^k)||\na \varphi|\,dz=0,
\]
which means that $\mathcal{A}_\epsilon(z,\na u^k_\epsilon)\to \mathcal{A}(z,\na u^k)$ weakly in $\mathcal{D}'(Q_R,\RR^{Nn}) $ as $ \ep\to 0^+.$

To prove \eqref{e51}, we take $u^k_\epsilon-u^k_{\epsilon'}\in C(I_R;L^2(D_R,\RR^N)) \cap L^q(I_R;W^{1,q}_0(D_R,\RR^N))$ as a test function to the difference of corresponding systems
\begin{align*}
\begin{split}
    &\pa_t\bigl(u^k_{\epsilon}-u^k_{\epsilon'}\bigr)-\dv\bigl(\mathcal{A}(z,\na u_\ep^k)-\mathcal{A}(z,\na u_{\epsilon'}^k)\bigr)\\
    &\qquad=-\dv\bigl(|F_\ep^k|^{p-2}F_\ep^k-|F_{\ep'}^k|^{p-2}F_{\ep'}^k+a(z)\bigl(|F_\ep^k|^{q-2}F_\ep^k-|F_{\ep'}^k|^{q-2}F_{\ep'}^k\bigr)\bigr)\\
    &\qquad\qquad -\dv \bigl(\epsilon|F^k_{\epsilon}|^{q-2}F^k_\epsilon-\epsilon'|F^k_{\epsilon'}|^{q-2}F^k_{\epsilon'}\bigr)\\
    &\qquad\qquad+\dv\bigl(b(z)\bigl(\epsilon|\na u^k_{\epsilon}|^{q-2}\na u^k_\epsilon-\epsilon'|\na u^k_{\epsilon'}|^{q-2}\na u^k_{\epsilon'}\bigr)\bigr).
\end{split}
\end{align*}
The standard comparison estimate gives a constant $c=c(n,N,p,q,\nu,L)$ such that
\begin{align*}
\begin{split}
    &\iint_{ C_R} H(z,|\na u_\ep^k-\na u^k_{\ep'}|)\,dz\\
    &\qquad\le c\iint_{ C_R}\left(||F^k_{\epsilon}|^{p-2}F^k_{\epsilon}-|F^k_{\epsilon'}|^{p-2}F^k_{\epsilon'}|^\frac{p}{p-1}+a(z)||F^k_{\epsilon}|^{q-2}F^k_{\epsilon}-|F^k_{\epsilon'}|^{q-2}F^k_{\epsilon'}|^\frac{q}{q-1}\right)\,dz\\
    &\qquad\qquad +c(\epsilon+\epsilon')\iint_{ C_R}\left(|F^k_{\epsilon}|^q+|F^k_{\epsilon'}|^q+|\na u^k_{\epsilon}|^q+|\na u^k_{\epsilon'}|^q\right)\,dz.
\end{split}
\end{align*}
The last term also converges to zero as $\ep,\ep'\to0^+$, because of $|F^k_\epsilon|,|F^k_{\epsilon'}|\le k$ and \eqref{e47}. Furthermore, the Lebesgue dominated convergence theorem implies that the first term on the right-hand side also converges to zero. Therefore, the claim in \eqref{e51} is proved. Since $k\in\mathbb N$ was arbitrary, we obtain a sequence $(u^k)_{k\in\mathbb N}\subset C(I_R;L^2(D_R,\RR^N))\cap L^q(I_R;W^{1,q}_0(D_R,\RR^N))$ of weak solutions to \eqref{e50}. Moreover, from \eqref{e42} and \eqref{e45} we have
\begin{align}\label{e58}
\begin{split}
    &\sup_{t\in I_R}\int_{D_R}|u^k(x,t)|^2\,dx+\iint_{ C_R}H(z,|\na u^k|)\,dz+\lVert \pa_t u^k\rVert_{L^{q'}(I_R;W^{-1,q'}(D_R,\RR^N))}\\
    &\qquad\le c\left(\iint_{ C_R}H(z,|F|)\,dz+1\right).
\end{split}    
\end{align}

\subsubsection{Second convergence}
Note that \eqref{e58} guarantees $(\na u^k)_{k\geq1}$ is bounded only in $L^p( C_R,\RR^{Nn})$. Since 
\[
        L^p(I_R;W^{1,p}_0(D_R,\RR^N))\Subset L^p( C_R,\RR^N)\Subset L^2( C_R,\RR^N) \subset L^{q'}(I_R;W^{-1,q'}(D_R,\RR^N)),
\]
there exist limit functions $u\in C(I_R;L^2(D_R,\RR^N))\cap L^p(I_R;W^{1,p}_0(D_R,\RR^N))$ and $v\in L^q( C_R,\RR^{Nn})$ such that 
\begin{align*}
\begin{split}
    &u^k \to  u    \quad\text{strongly in}\ C(I_R;L^2(D_R,\RR^N))\cap L^p( C_R,\RR^N),\\
    & u^k\to u    \quad\text{weakly in}\ L^p(I_R;W^{1,p}_0(D_R,\RR^N)),\\
    & \pa_tu^k_\epsilon \to \pa_t u^k   \quad\text{weakly in}\ L^{q'}(I_R;W^{-1,q'}(D_R,\RR^N)),\\
    & a^\frac{1}{q}\na u^k\to v \quad\text{weakly in}\ L^q( C_R,\RR^{Nn}),
\end{split}
\end{align*}
as $k\to\infty$.
As in the first convergence result, it is enough to show
\begin{align}\label{e61}
    \iint_{ C_R}H(z,|\na u^k-\na u|)\,dz\to0 \quad\text{as}\ k\to\infty.
\end{align}
We take $u^k-u^{k'}\in C(I_R;L^2(D_R,\RR^N))\cap L^q(I_R;W^{1,q}_0(D_R,\RR^N))$ as a test function in
\begin{align*}
    \begin{split}
        &\pa_t\bigl(u^{k}-u^{k'}\bigr)-\dv\bigl(\mathcal{A}(z,\na u^k)-\mathcal{A}(z,\na u^{k'})\bigr)\\
        &\qquad=-\dv\bigl(|F^k|^{p-2}F^k-|F^{k'}|^{p-2}F^{k'}+a(z)\bigl(|F^k|^{q-2}F^k-|F^{k'}|^{q-2}F^{k'}\bigr)\bigr)
    \end{split}
\end{align*}
in $ C_R$ and obtain a constant $c=c(n,N,p,q,\nu)$ such that
\begin{align*}
\begin{split}
    &\iint_{ C_R} H(z,|\na u^k-\na u^{k'}|)\,dz\\
    &\qquad\le c\iint_{ C_R}\left(||F^k|^{p-2}F^k-|F^{k'}|^{p-2}F^{k'}|^\frac{p}{p-1}+a(z)||F^k|^{q-2}F^k-|F^{k'}|^{q-2}F^{k'}|^\frac{q}{q-1}\right)\,dz.
\end{split}
\end{align*}
As we have $|F^k|,|F^{k'}|\le |F|$,
we conclude that
\[
    \lim_{k,k'\to\infty}\iint_{ C_R}H(z,|\na u^k-\na u^{k'}|)\,dz=0.
\]
This shows that $\na u^k $ is a Cauchy sequence with respect to the double phase functional and by the weak convergence mentioned above the limit is $\na u$. This proves \eqref{e61} and therefore $u$ is a weak solution to \eqref{e2} with $g\equiv0$. 

To conclude the approximation statement of Theorem~\ref{main3}, we use \eqref{e51}.  For every $k\in\mathbb N$, we have $\ep_k=\ep_k(k)$ small enough so that
\[
    \iint_{ C_R}H(z,|\na u^k_{\ep_k}-\na u^k|)\,dz\le \frac{1}{k}.
\]
It follows from \eqref{e61} that
\[
    \lim_{k\to\infty}\iint_{ C_R}H(z,|\na u^k_{\ep_k}-\na u|)\,dz=0.
\]
Letting $u_l=u^l_{\ep_l}$ for every $l\in\mathbb N$, the proof is completed.


\begin{thebibliography}{10}

\bibitem{MR751305}
E.~Acerbi and N.~Fusco.
\newblock Semicontinuity problems in the calculus of variations.
\newblock {\em Arch. Rational Mech. Anal.}, 86(2):125--145, 1984.

\bibitem{MR970512}
E.~Acerbi and N.~Fusco.
\newblock An approximation lemma for {$W^{1,p}$} functions.
\newblock In {\em Material instabilities in continuum mechanics ({E}dinburgh,
  1985--1986)}, Oxford Sci. Publ., pages 1--5. Oxford Univ. Press, New York,
  1988.

\bibitem{MR4281251}
K.~Adimurthi, S.~Byun, and J.~Park.
\newblock End point gradient estimates for quasilinear parabolic equations with
  variable exponent growth on nonsmooth domains.
\newblock {\em Calc. Var. Partial Differential Equations}, 60(4):Paper No. 145,
  67, 2021.

\bibitem{MR3348922}
P.~Baroni, M.~Colombo, and G.~Mingione.
\newblock Harnack inequalities for double phase functionals.
\newblock {\em Nonlinear Anal.}, 121:206--222, 2015.

\bibitem{MR4351758}
V.~B\"ogelein.
\newblock {\em Regularity results for weak and very weak solutions of higher
  order parabolic systems}.
\newblock ProQuest LLC, Ann Arbor, MI, 2007.
\newblock Thesis (Ph.D.)--Friedrich-Alexander-Universit\"at Erlangen-N\"urnberg
  (Germany).

\bibitem{MR3059060}
V.~B\"{o}gelein, F.~Duzaar, and G.~Mingione.
\newblock The regularity of general parabolic systems with degenerate
  diffusion.
\newblock {\em Mem. Amer. Math. Soc.}, 221(1041):vi+143, 2013.

\bibitem{MR3985550}
M.~Bul\'{\i}\v{c}ek, J.~Burczak, and S.~Schwarzacher.
\newblock Well posedness of nonlinear parabolic systems beyond duality.
\newblock {\em Ann. Inst. H. Poincar\'{e} C Anal. Non Lin\'{e}aire},
  36(5):1467--1500, 2019.

\bibitem{MR3985549}
I.~Chlebicka, P.~Gwiazda, and A.~Zatorska-Goldstein.
\newblock Parabolic equation in time and space dependent anisotropic
  {M}usielak-{O}rlicz spaces in absence of {L}avrentiev's phenomenon.
\newblock {\em Ann. Inst. H. Poincar\'{e} C Anal. Non Lin\'{e}aire},
  36(5):1431--1465, 2019.

\bibitem{MR3294408}
M.~Colombo and G.~Mingione.
\newblock Regularity for double phase variational problems.
\newblock {\em Arch. Ration. Mech. Anal.}, 215(2):443--496, 2015.

\bibitem{MR3447716}
M.~Colombo and G.~Mingione.
\newblock Calder\'{o}n-{Z}ygmund estimates and non-uniformly elliptic
  operators.
\newblock {\em J. Funct. Anal.}, 270(4):1416--1478, 2016.


\bibitem{MR3985927}
C.~De~Filippis and G.~Mingione.
\newblock A borderline case of {C}alder\'{o}n-{Z}ygmund estimates for
  nonuniformly elliptic problems.
\newblock {\em St. Petersburg Math. J.}, 31(3):455--477, 2020.

\bibitem{MR1230384}
E.~DiBenedetto.
\newblock {\em Degenerate parabolic equations}.
\newblock Universitext. Springer-Verlag, New York, 1993.

\bibitem{MR2668872}
L.~Diening, M.~Rů\v{z}i\v{c}ka, and J.~Wolf.
\newblock Existence of weak solutions for unsteady motions of generalized
  {N}ewtonian fluids.
\newblock {\em Ann. Sc. Norm. Super. Pisa Cl. Sci. (5)}, 9(1):1--46, 2010.

\bibitem{MR3672391}
L.~Diening, S.~Schwarzacher, B.~Stroffolini, and A.~Verde.
\newblock Parabolic {L}ipschitz truncation and caloric approximation.
\newblock {\em Calc. Var. Partial Differential Equations}, 56(4):Paper No. 120,
  27, 2017.

\bibitem{MR2076158}
L.~Esposito, F.~Leonetti, and G.~Mingione.
\newblock Sharp regularity for functionals with {$(p,q)$} growth.
\newblock {\em J. Differential Equations}, 204(1):5--55, 2004.

\bibitem{CZ}
W.~Kim.
\newblock Calder\'on-{Z}ygmund type estimate for the parabolic double-phase
  system.
\newblock {\em arXiv}, 2023.

\bibitem{KKM}
W.~Kim, J.~Kinnunen, and K.~Moring.
\newblock Gradient higher integrability for degenerate parabolic double-phase
  systems.
\newblock {\em arXiv}, 2022.

\bibitem{MR4306765}
J.~Kinnunen, J.~Lehrb\"{a}ck, and A.~V\"{a}h\"{a}kangas.
\newblock {\em Maximal function methods for {S}obolev spaces}, volume 257 of
  {\em Mathematical Surveys and Monographs}.
\newblock American Mathematical Society, Providence, RI, [2021] \copyright
  2021.

\bibitem{MR1948889}
J.~Kinnunen and J.~L. Lewis.
\newblock Very weak solutions of parabolic systems of {$p$}-{L}aplacian type.
\newblock {\em Ark. Mat.}, 40(1):105--132, 2002.

\bibitem{MR916688}
J.~Simon.
\newblock Compact sets in the space {$L^p(0,T;B)$}.
\newblock {\em Ann. Mat. Pura Appl. (4)}, 146:65--96, 1987.

\bibitem{MR3532237}
T.~Singer.
\newblock Existence of weak solutions of parabolic systems with {$p,
  q$}-growth.
\newblock {\em Manuscripta Math.}, 151(1-2):87--112, 2016.

\end{thebibliography}
\end{document}